%% file: Newton_General.tex
\documentclass[11pt]{article}
\usepackage[a4paper,top=2cm,bottom=2cm,left=1.5cm,right=1.5cm]{geometry}
\usepackage[utf8]{inputenc}
\usepackage{amsmath,amsfonts,amssymb,amsthm}
\usepackage{natbib,hyperref}
\usepackage{graphicx}
\usepackage{float}
\usepackage{xcolor}

\usepackage{multirow}

\input{ListeDefinitions}

\begin{document}

\footnotetext[1]{antoine.godichon$\_$baggioni@upmc.fr,   
	  Laboratoire de Probabilités, Statistique et Modélisation (LPSM),
	Sorbonne Université,
	4 Place Jussieu, 75005 Paris, France.}
\footnotetext[2]{Laboratoire de Mathématiques de l'INSA Rouen Normandie,
	INSA Rouen Normandie,
	BP 08 - Avenue de l'Université, 76800 Saint-Etienne du Rouvray, France}

\title{Online estimation of the inverse of the Hessian for stochastic optimization with application to universal stochastic Newton algorithms}

\author{ Antoine Godichon-Baggioni   \footnotemark[1],	 
  Wei Lu\footnotemark[2] and 
	  Bruno Portier\footnotemark[2]
}

\maketitle

\begin{abstract}
This paper addresses second-order stochastic optimization for estimating the minimizer of a convex function written as an expectation. A direct recursive estimation technique for the inverse Hessian matrix using a Robbins-Monro procedure is introduced. This approach enables to drastically reduces computational complexity. Above all, it allows to develop universal stochastic Newton methods and investigate the asymptotic efficiency of the proposed approach. This work so expands the application scope of second-order algorithms in stochastic optimization.
\end{abstract}


\medskip

\input{Introduction}

\input{Framework}

\input{Estimation_A}
\input{WASNA}

\input{Applications}
 \input{Conclusion}

\section*{Acknowledgments} The authors would like to thank Guillaume Sallé for his help in improving the proofs.
 
\input{proof}

\bibliographystyle{apalike}
\bibliography{Biblio_Rapport}

\begin{appendix}
\section{How to verify assumptions in practice?}\label{appendix}
\subsection{Linear model}
We consider the case of the linear model, i.e., where $(X,Y)$ lies in $\mathbb{R}^{d} \times \mathbb{R}$ and
\[
Y = X^{T}\theta + \epsilon,
\]
with $\epsilon$ independent of $X$ and $\theta \in \mathbb{R}^{d}$. The functional to minimize is defined for all $h \in \mathbb{R}^{d}$ by
\[
G(h) = \mathbb{E} \left[ \frac{1}{2} \left( Y - X^{T} h \right)^{2} \right] =: \mathbb{E} \left[ g(X,Y,h) \right].
\]
Then, for all $x,y \in \mathbb{R}^{d} \times \mathbb{R}$, the function $g$ is twice differentiable, with
\[
\nabla_{h} g(x ,y , h) = - (y - x^{T} h) x, \quad \quad \text{and} \quad \quad \nabla_{h}^{2} g(x , h) = xx^{T}.
\]
If $X$ admits a moment of order $2$, then $G$ is twice differentiable and, for all $h \in \mathbb{R}^{d}$,
\[
\nabla^{2} G(h) = \mathbb{E} \left[ XX^{T} \right],
\]
so that assumption \textbf{(A3)} is satisfied. Suppose also that $\mathbb{E} \left[ XX^{T} \right]$ is positive, ensuring assumption \textbf{(A2)}. 
Additionally, suppose that $X$ admits a moment of order $2q$, which satisfies Assumption \textbf{(A4)}. Finally, suppose that $X$ (resp. $\epsilon$) admits a moment of order $8$ (resp. $4$), ensuring that assumption \textbf{(A1)} holds. Indeed, 
\begin{align*}
\mathbb{E} \left[  \left\| \nabla_{h}  g(X, Y ,h) \right\|^{2} \right] & \leq \mathbb{E} \left[ \| X \|^{2} \left| X^{T} (\theta - h) + \epsilon \right|^{2} \right]  \\
& \leq 2 \mathbb{E} \left[ \| X \|^{4} \right] \| \theta - h \|^{2} + 2 \mathbb{E} \left[ \epsilon^{2} \right] \mathbb{E} \left[ \| X \|^{2} \right].
\end{align*}
Then, since for all $h \in \mathbb{R}^{d}$, $\nabla^{2} G(h) = \mathbb{E} \left[ XX^{T} \right]$ is positive, we have
\[
\| \theta - h \|^{2} \leq \frac{2}{\mu} \left( G(h) - G(\theta) \right),
\]
where $\mu = \lambda_{\min} \left(\mathbb{E} \left[ XX^{T} \right]\right)$, so that
\[
\mathbb{E} \left[ \nabla_{h}  g(X, Y ,h)^{2} \right] \leq 2 \mathbb{E} \left[ \| X \|^{2} \right] \frac{2}{\mu} (G(h) - G(\theta)) + 2 \mathbb{E} \left[ \epsilon^{2} \right] \mathbb{E} \left[ \| X \|^{2} \right].
\]
In addition,
\begin{align*}
\mathbb{E} \left[ \nabla_{h}  g(X, Y ,h)^{4} \right] & \leq \mathbb{E} \left[ \| X \|^{4} \left| X^{T} (\theta - h) + \epsilon \right|^{4} \right]  \\
& \leq 8 \mathbb{E} \left[ \| X \|^{8} \right] \| \theta - h \|^{4} + 8 \mathbb{E} \left[ \epsilon^{4} \right] \mathbb{E} \left[ \| X \|^{4} \right],
\end{align*}
and Assumption \textbf{(A1b)} is satisfied taking any $R>0$.

\medskip
\paragraph{Summary:} If $\epsilon$ admits a moment of order $4$, if $X$ admits a moment of order $\max \lbrace 8 , 2q \rbrace $, and if $\mathbb{E} \left[ XX^{T} \right]$ is positive, then the assumptions are satisfied.

\subsection{Logistic regression}

Let us consider the case of logistic regression, where $(X,Y) \in \mathbb{R}^{d} \times \{0,1\}$ and the conditional distribution of $Y$ given $X$ follows a Bernoulli law:
\[
Y | X \sim \mathcal{B}(\pi(\theta^{T}X)),
\]
where $\pi$ is the sigmoid function, defined for all $z \in \mathbb{R}$ as
\[
\pi(z) = \frac{e^{z}}{1+e^{z}}.
\]
The objective functional to minimize is given, for all $ h \in \mathbb{R}^{d} $, by
\[
G(h) = \mathbb{E} \left[ \log \left( 1+ \exp \left( h^{T}X \right) \right) - YX^{T}h \right] =: \mathbb{E} \left[ g(X,Y,h) \right].
\]
For all $(x,y) \in \mathbb{R}^{d} \times \{0,1\}$, the function $ g(x,y, \cdot) $ is twice differentiable, with
\[
\nabla_h g(x,y,h) = (\pi(x^{T}h) - y)x
\]
and
\[
\nabla_h^2 g(x,y,h) = \pi(x^{T}h) (1 - \pi(x^{T}h)) xx^{T}.
\]
If $X$ has a finite fourth-order moment, Assumption \textbf{(A1)} holds, i.e.,
\[
\mathbb{E} \left[ \left\| \nabla_h g(X,Y,h) \right\|^{4} \right] \leq \mathbb{E} \left[ \left\| X \right\|^{4} \right].
\]
Furthermore, $ G $ is also twice differentiable, with
\[
\nabla^2 G(h) = \mathbb{E} \left[ \pi(X^{T}h)(1 - \pi(X^{T}h)) XX^{T} \right],
\]
and
\[
\left\| \nabla^2 G(h) \right\|_{op} \leq \frac{1}{4} \mathbb{E}\left[ \|X\|^{2} \right].
\]
If $\nabla^2 G(\theta)$ is positive, a condition satisfied, for instance, when $X$ follows an elliptic distribution (see \cite{gadat2017optimal}), then Assumption \textbf{(A2)} is verified.
Additionally, if $X$ has a finite third-order moment, Assumption \textbf{(A3)} holds (see Lemma 6.2 in \cite{bercu2020efficient}):
\[
\left\| \nabla^2 G(h) - \nabla^2 G(\theta) \right\| \leq \frac{1}{12\sqrt{3}} \mathbb{E} \left[ \| X \|^{3} \right] \| h - \theta \| .
\]
Finally, Assumption \textbf{(A4)} is satisfied if $X$ has a finite $2q$-th moment, since
\[
\mathbb{E} \left[ \left\| \nabla_h^2 g(X,Y,h) \right\|_{op}^{q} \right] \leq \frac{1}{4^{q}} \mathbb{E} \left[ \| X \|^{2q} \right].
\]

\paragraph*{Summary:} If $X$ admits a moment of order $\max\lbrace 3,2q \rbrace$ and if $\nabla^{2}G(\theta)$ is positive (which is satisified if $X$ follows an elliptics distribution for instance), then the assumptions are satisfied.

\subsection{Estimation of the median}
Let us consider a continuous random vector $X$ in $\mathbb{R}^{d}$ satisfying the following two standard assumptions (see \cite{HC} among others):
\begin{itemize}
    \item The distribution of $X$ is not concentrated on a single straight line: for all $h \in \mathbb{R}^{d}$, there exists $h' \in \mathbb{R}^{d}$  such that $h^{T} h' = 0$ and 
    \[
    \mathbb{V} \left[   X^{T} h'   \right] > 0.
    \]
    \item The distribution of $X$ is not concentrated around single points: there exists a constant $C$ such that for all $h \in \mathbb{R}^{d}$,
    \[
 \mathbb{E} \left[ \frac{1}{\| X - h \|^{2q}} \right] \leq C.
    \]
\end{itemize}
This latter assumption is closely related to small-ball probabilities and is not restrictive as long as $d \geq 2q+2$ (see \cite{HC} among others). 
Under these conditions, the geometric median of $X$ is the unique minimizer of the functional defined for all $h \in \mathbb{R}^{d}$ by
\[
G(h) := \mathbb{E} \left[ \| X - h \| - \| X \| \right] =: \mathbb{E} \left[ g(X,h) \right].
\]
For all $h \in \mathbb{R}^{d}$ and $x \neq h$, the function $g(x, \cdot)$ is twice differentiable at $h$, with
\[
\nabla_{h} g(x,h) = - \frac{x - h}{\| x - h \|}, \quad \quad \text{and} \quad \quad \nabla_{h}^{2} g(x,h) = \frac{1}{\|x - h \|} \left( I_{d} - \frac{(x - h)(x - h)^{T}}{\| x - h \|^{2}} \right).
\]
Thus, Assumption \textbf{(A1)} holds with
\[
\mathbb{E} \left[ \left\| \nabla_{h} g(X , h) \right\|^{4} \right] = 1.
\]
Furthermore, Assumption \textbf{(A2)} is satisfied since
\[
\nabla^{2}G(h) = \mathbb{E} \left[ \frac{1}{\|X - h \|} \left( I_{d} - \frac{(X - h)(X - h)^{T}}{\| X - h \|^{2}} \right) \right],
\]
and
\[
\left\| \nabla^{2}G(h) \right\|_{op} \leq \mathbb{E} \left[ \frac{1}{\| X - h \|} \right] \leq C^{1/2q}.
\]
The positivity of $\nabla^{2}G(m)$ (where $m$ is the median of $X$) is established by Proposition 2.1 in \cite{HC}. Moreover, Assumption \textbf{(A4)} is also satisfied, and Assumption \textbf{(A3)} is verified in Lemma 3.3 of \cite{HC}.

\paragraph*{Summary:} If the distribution of the continuous random vector $X$ is not concentrated along a single straight line nor around single points, then all assumptions are satisfied.


\subsection{Estimation of $p$-means}
Let us now consider the estimation of the $p$-mean of a continuous random vector $X$ in $\mathbb{R}^{d}$, where $p \in (1,2)$. The $p$-mean $m$ of $X$ is defined as the minimizer of the functional defined for all $h \in \mathbb{R}^d$
\[
G(h) = \mathbb{E} \left[ \left\| X - h \right\|^{p} \right] =: \mathbb{E}\left[ g(X,h) \right] .
\]
For any $x$, the function $g(x,.)$ is differentiable, and for any $h$, and for all $x \neq h$, it is twice differentiable at $h$, with
\begin{align*}
    \nabla_{h}g(x,h) &= - (x-h)\|x-h\|^{p-2},
    \end{align*}
    and
    \begin{align*}
    \nabla_{h}^2 g(x,h) &= \|x-h\|^{p-2}\left(I_d-(2-p)\frac{(x-h)(x-h)^T}{\|x-h\|^2}\right).
\end{align*}
Since $p \in (1,2)$, we have $2(p-1) < p$. Denoting by $m_{p}$ the $p$-mean of $X$, we obtain
\begin{align*}
\mathbb{E} \left[ \left\| \nabla_{h}g(X,h) \right\|^{2} \right]  = \mathbb{E} \left[ \left\| X - h\right\|^{2(p-1)} \right] \
 \leq 1 + \mathbb{E} \left[ \left\| X - h\right\|^{p} \right] \
 \leq 1 + G(m_{p}) + G(h) - G(m_{p}).
\end{align*}
Moreover,
\begin{align*}
\mathbb{E} \left[ \left\| \nabla_{h}g(X,h) \right\|^{4} \right] &= \mathbb{E} \left[ \left\| X - h\right\|^{4(p-1)} \right] \\
&\leq 2^{4p-5} \mathbb{E}\left[ \left\| X - m_{p} \right\|^{4(p-1)} \right] + 2^{4p-5} \left\| m_{p} - h \right\|^{4(p-1)}.
\end{align*}
Thus, Assumption \textbf{(A1)} holds for any $R>0$. 
Now, suppose that the distribution of $X$ is not concentrated around single points, i.e., there exists a constant $C_{q}$ such that for all $h \in \mathbb{R}^{d}$,
\[
\mathbb{E} \left[ \frac{1}{\| X - h \|^{q(2-p)}} \right] \leq C_{q}.
\]
As in the case of the median, this assumption is closely related to small-ball probabilities and is non-restrictive as long as $d \geq 2q +2$. Under this assumption, Assumption \textbf{(A4)} is satisfied with
\[
\mathbb{E}\left[ \left\|  \nabla_{h}^{2}g(X , h ) \right\|_{op}^{q} \right] \leq  (2-p)^{q}  \mathbb{E} \left[ \frac{1}{\| X - h \|^{q(2-p)}} \right] \leq (2-p)^{q} C_{q}.
\]
Furthermore, analogous to the median, one can verify that Assumption \textbf{(A3)} holds (for instance, by adapting the proof of Lemma 5.1 in \cite{CCG2015}). Finally, we have
\[
\nabla^{2}G(h) = \mathbb{E} \left[ \|X-h\|^{p-2}\left(I_d-(2-p)\frac{(X-h)(X-h)^T}{\|X-h\|^2}\right) \right]
\]
and
\[
\lambda_{\min} \left( \nabla^{2}G\left(m_{p} \right) \right) \geq (p-1) \mathbb{E} \left[ \frac{1}{\left\| X - m_{p} \right\|^{2-p}} \right] > 0,
\]
thus satisfying Assumption \textbf{(A2)}.

\medskip

\paragraph*{Summary:} If the distribution of $X$ is not concentrated around single points and if $X$ has a finite moment of order $4(p-1)$, then all the required assumptions are satisfied.

\section{Useful lemmas}
The following lemma is a corollary of the Robbins-Siegmund theorem \cite{robbins1971convergence}.
\begin{lem}\label{corDuflo}
	Let $(V_n)$, $(B_n)$, $(D_n)$ and $(a_n)$ be positive sequences adapted to
	$\mathbb{F} = (\mathcal{F}_n)$. Assume that $V_0$ is integrable and, for all $n \ge 0$,
	$$\mathbb{E}\cro{V_{n+1}|\mathcal{F}_n}\le V_n + B_n -D_n \quad a.s.$$
	Assume also that $\sum_{n=0}^\infty \frac{B_n}{a_n} < +\infty$ a.s. If $a_n \rightarrow \infty$, then  $V_n = o\pa{a_n} a.s.$
\end{lem}
The proof of this lemma is given in Chapter 1.III in \cite{duflo1990methodes}, and here we give a generalized version of it.
\begin{lem}\label{corRS}
	Let $(V_n)$, $(B_n)$, $(D_n)$, $E_n$ and $(a_n)$ be positive sequences adapted to
	$\mathbb{F} = (\mathcal{F}_n)$. Assume that $V_0$ is integrable and, for all $n \ge 0$,
	$$\mathbb{E}\cro{V_{n+1}|\mathcal{F}_n}\le \pa{1+E_n}V_n + B_n -D_n \quad a.s.$$
	Assume also that $\sum_{n=0}^\infty E_n < +\infty$ a.s. and $\sum_{n=0}^\infty \frac{B_n}{a_n} < +\infty$ a.s. If $a_n \rightarrow \infty$, then $V_n = o\pa{a_n} a.s.$
\end{lem}
\begin{proof}[Proof of Lemma \ref{corRS}]
Note that the case where $E_n = 0$ is exactly the case of Lemma \ref{corDuflo}. Therefore, we are going to study the case where $E_n \neq 0$. 
We define $\alpha_n := \prod_{k=0}^n (1+E_k)$. Note that since $\sum_{n=0}^\infty E_n$ converges almost surely, $\alpha_n$ converges almost surely to a finite random variable $\alpha_\infty$. Moreover, noting
$$V_{n}' = \frac{V_n}{\alpha_{n-1}},\quad B_{n}' = \frac{B_n}{\alpha_{n }},\quad D_{n}' = \frac{D_n}{\alpha_{n }},$$
we observe that
$$\mathbb{E}\cro{V_{n+1}'|\mathcal{F}_n} \leq V_{n}' +B_{n}'-D_{n}'.$$
In addition, since $\alpha_{n}\ge1$, we have
$$\sum_{n=0}^\infty B'_n \le \sum_{n=0}^\infty B_n < +\infty \quad a.s.$$
According to Lemma \ref{corDuflo}, we have $V'_n = o (a_{n})$ a.s., and
therefore $V_n = o (a_{n})$ a.s. Furthermore, 
$$\sum_{n=0}^\infty D_n \le \alpha_\infty\sum_{n=0}^\infty D'_n < +\infty \quad a.s.$$
\end{proof}
We now give two lemmas which will be tools for the study of the rate of convergence associated to the estimates $A_n$.
\begin{lem}\label{lemmartbeta}
	Let us denote by $ \mathcal{H} = \mathcal{M}_{q} (\mathbb{R})$ the set of squared matrices of size $q \times q$. Let us consider
	\[
	M_{n+1} = \sum_{k=1}^{n} \beta_{n,k} \gamma_{k}  R_{k} \xi_{k+1},
	\]
	where
	\begin{itemize}
		\item $ ( \xi_{n}  )$ is a $\mathcal{H}$-valued martingale differences  sequence adapted  to a filtration $ ( \mathcal{F}_{n}  )$ such that \begin{align}
			\notag & \mathbb{E}\left[ \left\| \xi_{n+1} \right\|_{F}^{2} |\mathcal{F}_{n} \right] \leq C + R_{2,n} \quad a.s, \\
			\label{hypmart} & \sum_{n\geq 1}\gamma_{n}\mathbb{E}\left[ \left\| \xi_{n+1} \right\|_{F}^{2}\mathbf{1}_{\left\| \xi_{n+1} \right\|_{F}^{2} \geq \gamma_{n}^{-1}(\ln n)^{-1}} |\mathcal{F}_{n} \right] < + \infty \quad a.s, 
		\end{align}
		where $C$ is a non-negative random variable and $(R_{2,n})_n$ converges almost surely to $0$;
		\item $\gamma_{n}=cn^{-\gamma}$ with $c> 0$ and $\gamma \in (1/2,1)$;
		\item $(R_{n})$ is a sequence of matrices lying in $\mathcal{H}$ such that
		\[
		\left\| R_{n} \right\|_{F} = o \left( v_{n} \right) \quad a.s \quad \quad \text{where} \quad \quad v_{n} = \frac{(\ln n)^{a}}{n^{b}} ,
		\]
		with $a,b \in \mathbb{R}$;
		\item For all $n \geq 1$ and $1 \leq k \leq n$, and for all $A \in \mathcal{H}$, 
		\[
		\beta_{n,k} A = \prod_{j=k+1}^{n} \left( I_{q} - \gamma_{j} \Gamma \right) A \prod_{j=k+1}^{n}\left( I_{q} - \gamma_{j} \Gamma \right)\quad \text{and} \quad \beta_{n,n} A = A,
		\]
		where $\Gamma \in \mathcal{H} $ is symmetric and satisfies $0 < \lambda_{\min} (\Gamma ) \leq \lambda_{\max} (\Gamma) < + \infty $.
	\end{itemize}
	Then, 
	\[
	\left\| M_{n+1} \right\|_{F}^{2} = O \left(  \gamma_{n}v_{n}^{2}\ln n \right)  \quad a.s.
	\]
\end{lem}
\begin{proof}[Proof of Lemma \ref{lemmartbeta}]
	Let us now consider the events
	\begin{align*}
		& A_{n} = \left\lbrace \left\| R_{n} \right\|_{F} > v_{n} \quad \text{or} \quad R_{2,n}> C  \right\rbrace \\
		& B_{n+1} =  \left\lbrace \left\| R_{n} \right\|_{F} \leq v_{n}, R_{2,n} \leq C, \left\| \xi_{n+1} \right\|_{F}\leq \delta_{n}  \right\rbrace \\
		& C_{n+1} = \left\lbrace \left\| R_{n} \right\|_{F} \leq v_{n}, R_{2,n} \leq C, \left\| \xi_{n+1} \right\|_{F} > \delta_{n}  \right\rbrace
	\end{align*}
	with $\delta_{n} = \gamma_{n}^{-1/2}(\ln n)^{-1/2}$. One can check that $A_{n}^{c} = B_{n+1} \sqcup C_{n+1}$. Then, one can write $M_{n+1}$ as
	\begin{align*}
		M_{n+1} & = \sum_{k=1}^{n} \beta_{n,k} \gamma_{k}  R_{k} \xi_{k+1}\mathbf{1}_{A_{k}} + \sum_{k=1}^{n} \beta_{n,k} \gamma_{k}  R_{k} \xi_{k+1}\mathbf{1}_{A_{k}^{c}} \\
		& = \sum_{k=1}^{n} \beta_{n,k} \gamma_{k}  R_{k} \xi_{k+1}\mathbf{1}_{A_{k}} + \sum_{k=1}^{n} \beta_{n,k} \gamma_{k}  R_{k} \left( \xi_{k+1}\mathbf{1}_{B_{k+1}} - \mathbb{E}\left[ \xi_{k+1} \mathbf{1}_{B_{k+1}} |\mathcal{F}_{k} \right] \right) \\
		& + \sum_{k=1}^{n} \beta_{n,k} \gamma_{k}  R_{k} \left( \xi_{k+1}\mathbf{1}_{C_{k+1}} - \mathbb{E}\left[ \xi_{k+1} \mathbf{1}_{C_{k+1}} |\mathcal{F}_{k} \right] \right).
	\end{align*}
	Let us now give the rates of convergence of these three terms.
	
	\bigskip
	
	\noindent\textbf{Bounding $ M_{1,n+1} := \sum_{k=1}^{n} \beta_{n,k} \gamma_{k}  R_{k} \xi_{k+1}\mathbf{1}_{A_{k}} $. }
	There exists a rank $n_{0}$ such that for all $n \geq n_{0}$, $ \| I_{q} - \gamma_{n} \Gamma  \|_{op} \leq ( 1- \lambda_{\min} \gamma_{n}  ) $. Furthermore, $M_{1,n+1} =  ( I_{q} - \gamma_{n} \Gamma  ) M_{1,n} ( I_{q} - \gamma_{n} \Gamma  ) + \gamma_{n}  R_{n} \xi_{n+1} \mathbf{1}_{A_{n}}$. Then, for all $n \geq n_{0}$,
	\[
	\mathbb{E}\left[ \left\| M_{1,n+1} \right\|_{F}^{2} |\mathcal{F}_{n} \right] \leq \left( 1- \lambda_{\min}\gamma_{n} \right)^{4} \left\| M_{1,n} \right\|_{F}^{2} + \gamma_{n}^{2} \left\| R_{n} \right\|_{F}^{2} \left( C + R_{2,n} \right) \mathbf{1}_{A_{n}} .
	\]
	Considering $V_{n+1} = \prod_{k=1}^{n}  ( 1 +  \lambda_{\min} \gamma_{k}  )^{4} \| M_{1,n+1}  \|_{F}^{2}$, it follows that
	\[
	\mathbb{E}\left[ V_{n+1} |\mathcal{F}_{n} \right] \leq \left( 1 - \lambda_{\min}^{2} \gamma_{n}^{2} \right)^{4} V_{n} + \prod_{k=1}^{n} \left( 1 +  \lambda_{\min} \gamma_{k} \right)^{4}\gamma_{n}^{2}  \left\| R_{n} \right\|_{F}^{2} \left( C + R_{2,n} \right) \mathbf{1}_{A_{n}} 
	\] 
	Moreover, $\mathbf{1}_{A_{n}}$ converges almost surely to $0$ implying that
	\[
	\sum_{n\geq 1} \prod_{k=1}^{n} \left( 1 +  \lambda_{\min} \gamma_{k} \right)^{4}\gamma_{n}^{2}  \left\| R_{n} \right\|_{F}^{2} \left( C + R_{2,n} \right) \mathbf{1}_{A_{n}} < + \infty \quad a.s
	\]
	and applying Robbins-Siegmund Theorem, $V_{n}$ converges almost surely to a finite random variable, i.e
	\[
	\left\| M_{1,n+1} \right\|_{F}^{2} = \mathcal{O} \left( \prod_{k=1}^{n} \left( 1 +  \lambda_{\min} \gamma_{k} \right)^{-4}  \right) \quad a.s
	\]
	and converges exponentially fast.
	
	\bigskip
	
	\noindent\textbf{Bounding $M_{2,n+1} := \sum_{k=1}^{n} \beta_{n,k} \gamma_{k}  R_{k} ( \xi_{k+1}\mathbf{1}_{B_{k+1}} - \mathbb{E} [ \xi_{k+1} \mathbf{1}_{B_{k+1}} |\mathcal{F}_{k}  ]  )$. }
	Let us denote $\Xi_{k+1} =   R_{k}  ( \xi_{k+1}\mathbf{1}_{B_{k+1}} - \mathbb{E} [ \xi_{k+1} \mathbf{1}_{B_{k+1}} |\mathcal{F}_{k} ]  )$. Remark that $ ( \Xi_{n}  )$ is a sequence of martingale differences adapted to the filtration $ ( \mathcal{F}_{n}  )$. As in \cite{pinelis1994optimum} (proofs of Theorems~3.1 and 3.2), let $\lambda > 0$ and consider for all $t \in [ 0, 1]$ and $j \leq n$,
	\[
	\varphi(t) = \mathbb{E}\left[ \cosh \left( \lambda \left\| \sum_{k=1}^{j-1} \beta_{n,k}\gamma_{k} \Xi_{k+1} + t \beta_{n,j} \gamma_{j}\Xi_{j+1} \right\|_{F} \right) \left|\mathcal{F}_{j}\right. \right].
	\]
	One can check that $\varphi '(0) = 0$ and (see Pinelis for more details)
	\[
	\varphi '' (t) \leq \lambda^{2} \mathbb{E}\left[ \left\| \beta_{n,j} \gamma_{j}\Xi_{j+1} \right\|_{F}^{2}e^{ \lambda t \left\| \beta_{n,j} \gamma_{j}\Xi_{j+1} \right\|_{F}} \cosh \left( \lambda \left\| \sum_{k=1}^{j-1} \beta_{n,k}\gamma_{k} \Xi_{k+1} \right\|_{F} \right) |\mathcal{F}_{j} \right]
	\]
	Then,
	\begin{align*}
		\mathbb{E}\left[ \cosh \left( \lambda \left\| \sum_{k=1}^{j} \beta_{n,k}\gamma_{k} \Xi_{k+1} \right\|_{F} \right) | \mathcal{F}_{j} \right]	  = \varphi (1) & = \varphi (0) + \int_{0}^{1} (1-t) \varphi ''(t) dt \\
		& \leq \left( 1+ e_{j,n } \right)\cosh \left( \lambda \left\| \sum_{k=1}^{j-1} \beta_{n,k}\gamma_{k} \Xi_{k+1} \right\|_{F} \right) 
	\end{align*}
	with $e_{j,n} = \mathbb{E} [ e^{\lambda  \| \beta_{n,j} \gamma_{j}\Xi_{j+1}  \|_{F}} -1 - \lambda \| \beta_{n,j} \gamma_{j}\Xi_{j+1}  \|_{F} |\mathcal{F}_{j}  ]$, which is well defined since $\Xi_{j+1}$ is a.s. finite. Additionally, considering
	\[
	G_{n+1} = \frac{\cosh \left( \lambda \left\| \sum_{k=1}^{n} \beta_{n,k}\gamma_{k} \Xi_{k+1} \right\|_{F} \right)}{\prod_{j=1}^{n} \left( 1 + e_{j,n} \right)} \quad \quad \text{and} \quad \quad G_{0} = 1
	\]
	and since $\mathbb{E} [ G_{n+1} |\mathcal{F}_{n}  ] =G_{n}$, it comes $\mathbb{E} [ G_{n+1}  ] = 1 $. For all $r > 0$,
	\begin{align*}
		\mathbb{P}\left[ \left\| M_{2,n+1} \right\|_{F} \geq r \right] & = \mathbb{P}\left[ G_{n+1} \geq \frac{\cosh (\lambda r )}{\prod_{j=1}^{n}\left( 1+ e_{j,n} \right)} \right]  \leq \mathbb{P}\left[ 2G_{n+1} \geq \frac{e^{\lambda r }}{\prod_{j=1}^{n} \left( 1+ e_{jn} \right)} \right].
	\end{align*}
	Furthermore, let $\epsilon_{j+1} = \xi_{j+1}\mathbf{1}_{B_{j}} - \mathbb{E} [\xi_{j+1} \mathbf{1}_{B_{j}} | \mathcal{F}_{j} ]$ and note that $\mathbb{E} [  \| \epsilon_{j+1}  \|_{F}^{2} |\mathcal{F}_{j}  ] \leq 2C$. Then, recalling that $\delta_{n} = \gamma_{n}^{-1/2}(\ln n)^{-1/2}$, and since for all $k \geq 2$, 
	\[
	\mathbb{E}\left[ \left\| \epsilon_{j+1}\right\|_{F}^{k} |\mathcal{F}_{j} \right] \leq 2^{k-2}\delta_{j}^{k-2} \mathbb{E}\left[ \left\| \xi_{j+1} \right\|_{F}^{2} \mathbf{1}_{B_{j}} |\mathcal{F}_{j} \right] \leq 2^{k-1}C\delta_{j}^{k-2} ,
	\]
	ans since for any $A \in \mathcal{H}$ one has $\left\| \beta_{n,k} A \right\|_{F} \leq \left\| \beta_{n,j} \right\|_{op} \left\| A \right\|_{F}$,
	\begin{align*}
		e_{j,n}  \leq \sum_{k=2}^{\infty} \lambda^{k} \left\| \beta_{n,j}\right\|_{op}^{k} \gamma_{j}^{k} &  \mathbb{E}\left[ \left\|  \Xi_{j+1} \right\|_{F}^{k} |\mathcal{F}_{j} \right]  
		\leq \sum_{k=2}^{\infty} \lambda^{k} \left\| \beta_{n,j}\right\|_{op}^{k} \gamma_{j}^{k} v_{j}^{k} \mathbb{E}\left[ \left\| \epsilon_{j+1} \right\|_{F}^{k} |\mathcal{F}_{k} \right] \\
		& \leq \sum_{k=2}^{\infty} \lambda^{k} \left\| \beta_{n,j}\right\|_{op}^{k} \gamma_{j}^{k} v_{j}^{k} 2^{k-1}C\delta_{j}^{k-2} \\
		& \leq 2C\lambda^{2} \left\| \beta_{n,j} \right\|_{op}^{2} \gamma_{j}^{2}v_{j}^{2} \sum_{k=2}^{\infty} (2\lambda)^{k-2} \left\| \beta_{n,j}\right\|_{op}^{k-2}\gamma_{j}^{\frac{k-2}{2}} v_{j}^{k-2}\ln j^{-\frac{k-2}{2}} \\
		& = 2C\lambda^{2} \left\| \beta_{n,j} \right\|_{op}^{2} \gamma_{j}^{2}v_{j}^{2} \exp \left( 2 \lambda \left\| \beta_{n,j}\right\|_{op}\sqrt{\gamma_{j}} v_{j} \right)
	\end{align*}
	Then,
	\begin{align*}
	&\mathbb{P}\left[ \left\| M_{2,n+1} \right\|_{F} \geq r \right] \\
	\leq &\mathbb{P}\left[ 2G_{n+1} \geq \frac{e^{\lambda r}}{\prod_{j=1}^{n}\left( 1+  2C\lambda^{2} \left\| \beta_{n,j} \right\|_{op}^{2} \gamma_{j}^{2} v_{j}^{2}\exp \left( 2 \lambda \left\| \beta_{n,j}\right\|_{op} v_{j}\sqrt{\gamma_{j}\ln j} \right) \right) } \right]
\end{align*}
	Applying Markov's inequality, 
	\[
	\mathbb{P}\left[ \left\| M_{2,n+1} \right\| \geq r \right] \leq 2 \exp \left( - \lambda r + 2C\lambda^{2} \sum_{j=1}^{n} \left\| \beta_{n,j} \right\|_{op}^{2} \gamma_{j}^{2}v_{j}^{2} \exp \left( 2 \lambda \left\| \beta_{n,j}\right\|_{op}v_{j}\sqrt{\gamma_{j}\ln j} \right) \right) .
	\]
	Take $\lambda = \gamma_{n}^{-1/2}v_{n}^{-1}\sqrt{\ln n}$. Let $C_{0} =  \| \beta_{n_{0},0}  \|_{op}$ and remark that for $n \geq 2n_{0}$ (i.e such that $\gamma_{n/2} \lambda_{\max}(\Gamma) \leq 1$),  and for all $j \leq n/2$,
	\[
	\left\| \beta_{n,j} \right\|_{op} \leq C_{0}\exp \left( - c\lambda_{\min}(n/2)^{1-\alpha} \right),
	\]
	so that for all $j \leq n/2$,
	\[
	\lambda \left\|\beta_{n,j} \right\|_{op} \gamma_{j}v_{j}   \limite{n\to + \infty}{a.s} 0 .
	\]
	Furthermore, for all $n \geq 2n_{0}$, and for all $j \geq n/2$, 
	\[
	\lambda \left\|\beta_{n,j} \right\|_{op} \frac{\sqrt{\gamma_{j}}v_{j}}{\sqrt{\ln j}} \leq C_{0} 2^{2b+2a+\alpha+1}.
	\]
	Then, there is a positive constant $C''$ such that for all $n \geq 1$ and $j \leq n$,
	\[
	\exp \left( \lambda \left\|\beta_{n,j} \right\|_{op} \sqrt{\gamma_{j}}v_{j} \right) \leq C''
	\]
	Finally, one can easily check that (see Lemma E.2 in \cite{CG2015})
	\[
	\sum_{j=1}^{n} \left\| \beta_{n,j}\right\|_{op}^{2} \gamma_{j}^{2} \frac{(\ln j)^{2a}}{j^{2b}} =\mathcal{O} \left( \frac{(\ln n)^{2a}}{n^{2b+\alpha}} \right) .
	\]
	There is a positive constant $C'''$ such that
	\[
	\mathbb{P}\left[ \left\| M_{2,n+1} \right\|_{F} \geq r \right] \leq \exp \left( - r v_{n}^{-1}\gamma_{n}^{-1/2}\sqrt{\ln n} + C''' \ln n \right)
	\]
	Then , taking $r =  ( 2+C'''  ) v_{n} \sqrt{\gamma_{n}\ln n}$, it comes
	\[
	\mathbb{P}\left[ \left\| M_{2,n+1} \right\|_{F} \geq \left( 2+C''' \right) v_{n} \sqrt{\gamma_{n}\ln n} \right] \leq \exp \left( - 2 \ln n \right) = \frac{1}{n^{2}}
	\]
	and applying Borell Cantelli's lemma,
	\[
	\left\| M_{2,n+1} \right\|_{F} = \mathcal{O} \left( v_{n}\sqrt{\gamma_{n}\ln n}  \right) \quad a.s.
	\]

	\noindent\textbf{Bounding $M_{3,n+1} := \sum_{k=1}^{n} \beta_{n,k} \gamma_{k}  R_{n} ( \xi_{k+1}\mathbf{1}_{C_{k+1}} - \mathbb{E} [ \xi_{k+1} \mathbf{1}_{C_{k+1}} |\mathcal{F}_{k}  ]  )$. } Let us denote 
	\[
	\epsilon_{k+1} = \xi_{k+1}\mathbf{1}_{C_{k+1}} - \mathbb{E} [ \xi_{k+1} \mathbf{1}_{C_{k+1}} |\mathcal{F}_{k}  ]
	\] 
	and remark that for $n \geq n_{0} $,
	\begin{align*}
		\mathbb{E}\left[ \left\| M_{3,n+1} \right\|_{F}^{2} |\mathcal{F}_{n} \right] & \leq  \left( 1- \lambda_{\min} \gamma_{n} \right)^{4} \left\| M_{3,n} \right\|_{F}^{2} +  \gamma_{n}^{2} v_{n}^{2} \mathbb{E}\left[ \left\| \epsilon_{n+1} \right\|_{F}^{2} |\mathcal{F}_{n} \right] \\
		& \leq \left( 1- \lambda_{\min} \gamma_{n} \right)^{4} \left\| M_{3,n} \right\|_{F}^{2} +  \gamma_{n}^{2}v_{n}^{2} \mathbb{E}\left[ \left\| \xi_{n+1} \right\|_{F}^{2} \mathbf{1}_{\left\| \xi_{n+1} \right\|_{F}^{2} \geq \gamma_{n}^{-1}} |\mathcal{F}_{n} \right]
	\end{align*}
	Let $V_{n}' = \gamma_{n}^{-1}v_{n}^{-2}  \| M_{3,n}  \|_{F}^{2}$. There are a rank $n_{1}$ and a positive constant $c$ such that for all $n \geq n_{1}$
	\[
	\mathbb{E}\left[ V_{n+1} |\mathcal{F}_{n} \right] \leq \left( 1- c\gamma_{n} \right) V_{n} + \mathcal{O} \left( \gamma_{n}\mathbb{E}\left[ \left\| \xi_{n+1} \right\|_{F}^{2} \mathbf{1}_{\left\| \xi_{n+1} \right\|_{F}^{2} \geq \gamma_{n}^{-1}} |\mathcal{F}_{n} \right] \right)  \quad a.s.
	\]
	Applying Robbins-Siegmund Theorem as well as equation \eqref{hypmart}, it comes
	\[
	\left\| M_{3,n+1} \right\|_{F}^{2} = \mathcal{O} \left( \gamma_{n}v_{n}^{2} \right) \quad a.s.
	\]
	
\end{proof}
\begin{lem} \label{lemcours}
	Let $J_n$, $K_n$, $r_n$ be sequences of positive random variables and $c$ be a positive constant such that $r_n$ converges almost surely to $0$ and
	$$J_{n+1} = (1-c\widetilde{\gamma}_{n+1})J_n+\widetilde{\gamma}_{n+1}r_n(J_n+K_n)$$
	where $\widetilde{\gamma}_{n}=c_{\widetilde{\gamma}} n^{- \widetilde{\gamma}}$ 
	with $1/2<\widetilde{\gamma}<1$ and $c_{\widetilde{\gamma}}>0$. In addition, it is assumed that 
	$$K_n = \mathcal{O}(v_n) \quad a.s.$$
	where $v_n = c_vn^v(\ln n)^b$ with $v\in\mathbb{R}$ and $b\ge0$. Then
	$$J_n = \mathcal{O}(v_n) \quad a.s.$$
\end{lem}
\begin{proof}[Proof of Lemma \ref{lemcours}]
		For the sake of simplicity, let us assume that for every \( n \geq 0 \), \( c\widetilde{\gamma}_{n+1} \leq 1 \) (up to take \( n \) large enough). Now, consider the event \( E_{n,c} = \left\lbrace \left| r_{n} \right| \leq c/2 \right\rbrace \), and therefore \( \mathbf{1}_{E_{n,c}^{C}} \) converges almost surely to \( 0 \). Hence, \( J_{n+1} \) can be rewritten as:
		\begin{align*}
			J_{n+1} & \leq  \left( 1- c\widetilde{\gamma}_{n+1} \right) J_{n} + \frac{c}{2}\widetilde{\gamma}_{n+1} \left(  J_{n} + K_{n} \right) + \overbrace{\widetilde{\gamma}_{n+1} r_{n} \left( J_{n} + K_{n} \right)}^{=: \delta_{n}}\mathbf{1}_{E_{n,c}^{C}} \\
			& \leq \left( 1- \frac{c}{2}\widetilde{\gamma}_{n+1} \right) J_{n} + \frac{c}{2}\widetilde{\gamma}_{n+1} K_{n} + \delta_{n}\mathbf{1}_{E_{n,c}^{C}}
		\end{align*}
		By induction, one can check that for all \( n \geq 0 \):
		\begin{align*}
			J_{n} &\leq \tilde{\beta}_{n,0}J_{0} + \underbrace{\frac{c}{2} \sum_{k=0}^{n-1}\tilde{\beta}_{n,k+1}\widetilde{\gamma}_{k+1} K_{k}}_{=: J_{1,n}}
			  + \underbrace{\sum_{k=0}^{n-1} \tilde{\beta}_{n,k+1} \delta_{k}\mathbf{1}_{E_{k,c}^{C}}}_{=: J_{2,n}}
		\end{align*}
		with \( \tilde{\beta}_{n,k} := \prod_{j=k+1}^{n} \left( 1- \frac{c}{2}\widetilde{\gamma}_{j} \right) \) and \( \tilde{\beta}_{n,n} := 1 \). Using standard calculations, we can easily show that \( \tilde{\beta}_{n,0} \) converges at an exponential rate. Furthermore, \( J_{2,n} \) can be written as \( \tilde{\beta}_{n,0} \sum_{k=0}^{n-1} \tilde{\beta}_{k,0}^{-1} \delta_{k} \mathbf{1}_{E_{k,c}^{C}} \) and since \( \mathbf{1}_{E_{n,c}^{C}} \) converges almost surely to \( 0 \), the sum is almost surely finite, leading to
		\[ J_{2,n} = O \left( \tilde{\beta}_{n,0} \right) \quad \text{a.s.} \]
		and this term thus converges at an exponential rate.
		Finally, there exists a random variable \( 	K \) such that for every \( n \geq 1 \), \( K_{n} \leq Kv_{n} \) almost surely, leading to the induction relation:
		\begin{align*}
			J_{1,n+1} = \left( 1- \frac{c}{2}\widetilde{\gamma}_{n+1} \right) J_{1,n} + \frac{c}{2}\widetilde{\gamma}_{n+1}K_{n} \le \left( 1- \frac{c}{2}\widetilde{\gamma}_{n+1} \right) J_{1,n} +\frac{c}{2}\widetilde{\gamma}_{n+1}Kv_{n}
		\end{align*}
		By applying Proposition B.4 in \cite{godichon2023non}, we obtain:
		\[ J_{1,n} = O \left( v_{n} \right) \quad \text{a.s.} \]
\end{proof}

\end{appendix}
\end{document}

%% file: ListeDefinitions.tex
\newtheorem{prop}{Proposition}[section]

\newtheorem{rmq}{Remark}[section]  
\newtheorem{theo}{Theorem}[section] 
\newtheorem{lem}{Lemma}[section]

\newcommand{\limite}[2]{\xrightarrow[#1]{#2}}

\DeclareMathOperator*{\argmin}{arg\,min}

\DeclareUnicodeCharacter{B0}{\textdegree}
 
\usepackage{hyperref}
\hypersetup{
pdfpagemode=UseOutlines,      
pdfstartview=Fit,             
pdffitwindow=true,            
pdfpagelayout=TwoColumnsRight,
pdftoolbar=true,              
pdfmenubar=true,              
bookmarksopen=false,          
bookmarksnumbered=true,       
colorlinks=true,              
pdfauthor={Ton nom},     
pdftitle={Titre PDF},    
pdfcreator=PDFLaTeX,          %
pdfproducer=PDFLaTeX,         %
linkcolor=blue,               
urlcolor=blue,                
anchorcolor=blue,            
citecolor=blue,              
frenchlinks=true,             
pdfborder={0 0 0}             
}
\newcommand{\cb}[1]{}

\input epsf
\def\build#1_#2^#3{\mathrel{
\mathop{\kern 0pt#1}\limits_{#2}^{#3}}}

\def\Annexe{{\large Annexe}}
\newcommand{\bqa}{\begin{eqnarray}}
\newcommand{\eqa}{\end{eqnarray}}
\newcommand{\bdesc}{\begin{description}}
\newcommand{\edesc}{\end{description}}
\newcommand{\bqan}{\begin{eqnarray*}}
\newcommand{\eqan}{\end{eqnarray*}}

\def\finl{\hfill\break}

\def\norm#1{\left\| #1 \right\|}

\def\acc#1{\left\{ #1 \right\}}
\def\pa#1{\left( #1 \right)}

\def\cro#1{\left[ #1 \right]}
\def\acco#1{\left \{  #1\right \}}

\def\ndx2{n\,d_x^2}

\def\bkE{{\mathbb{E}}}

\def\bkR{{\mathbb{R}}}

\def\bkF{\mbox{{\rm I\kern-.17em F}}}

\def\bkNp{\mbox{{\small\rm I\kern-.20em N}}}
\def\bkO{{\rm \kern.24em
            \vrule width.05em height1.4ex depth-.05ex
            \kern-.26em O}}
\def\bkZ{\mbox{{\rm Z\kern-.32em Z}}}

\def\non{\mbox{$^{\mbox{\bf --}}$\kern-.1em \vrule width.05em height1.5ex}}

\def\emptysq{\mathbin{\vbox{\hrule\hbox{\vrule height1ex \kern.5em 
       \vrule height1ex}\hrule}}}

\usepackage{braket}

%% file: Introduction.tex
\section{Introduction}

In this paper, we consider the usual stochastic optimization problem, which consists of estimating the parameter $\theta \in \mathbb{R}^d$ defined by
$$\theta = \arg \min_{h \in \bkR^d} G(h)$$
where the function $G$ is defined for all $h\in\bkR^d$ by :
$G(h) = \bkE\cro{g(X, h)}$
and $X$ is a random vector of $\mathbb{R}^p$. 
The function $g: \mathbb{R}^p \times \mathbb{R}^d \longrightarrow \mathbb{R}$ is assumed to be twice continuously differentiable. 
This problem arises in various contexts such as estimating the parameters of logistic regressions \citep{bach2014adaptivity, cohen2017projected}, geometric median and quantiles \citep{HC, CCG2015}, or superquantiles \citep{bercu2020stochastic, costa2020non}. We denote $\nabla_h g$ and $\nabla^2_h g$ as the gradient and Hessian matrix of $g$ with respect to the second variable $h$, and $\nabla G$ and $\nabla^2 G$ as the gradient and Hessian matrix of $G$. 
It is assumed that the matrix $\nabla^2 G(\theta)$ is positive definite.

Starting from a sequence of independent random vectors $(X_n)_{n\geq 1}$ with the same distribution as $X$, 
we aim to online estimate the parameter $\theta$. 
One of the most well-known methods in this context is certainly the stochastic gradient algorithm, recursively defined for all $n \geq 1$ by:
\[
\theta_{n}^{SG} = \theta_{n-1}^{SG} - \nu_{n} \nabla_{h} g \left( X_{n} , \theta_{n-1}^{SG} \right)  
\]
where $\theta_{0}^{SG}$ is an arbitrarily chosen initial value and $(\nu_n)_{n\geq 1}$ is a sequence of positive real numbers decreasing towards 0. 
These algorithms have been extensively studied, 
with asymptotic results found by \cite{pelletier1998almost, pelletier2000asymptotic} 
and non-asymptotic results, such as uniform bounds of the quadratic mean error, presented by \cite{moulines2011non,gadat2017optimal} to name a few. 
To ensure asymptotic efficiency, an additional step consists of considering an averaged version of the estimates \citep{polyak1992acceleration}.

Despite their known efficiency, these methods can be very sensitive to ill-conditioned problems, where the Hessian has eigenvalues at different scales \citep{leluc2020asymptotic, bercu2020efficient}. 
To overcome this problem, second-order stochastic algorithms of the form 
$$\theta_n = \theta_{n-1} - \nu_{n}\,A_n \nabla_h g(X_n, \theta_{n-1})$$
have been proposed and recently studied. 
Here, $(\nu_n)_{n\geq 1}$ is a sequence of positive real numbers decreasing towards 0 and 
the matrix $A_n$ is a recursive estimate of the inverse of the Hessian matrix of $G$ at $\theta$, 
i.e a recursive estimate of $H^{-1}$ with  $H = \nabla^2 G(\theta)$. 
The challenge lies in constructing the recursive estimate $A_n$.

Several recursive second-order algorithms have been proposed and studied. 
For example, \cite{bercu2020efficient} propose an efficient stochastic Newton algorithm for estimating the parameters of a logistic regression model. 
In a recent work, \cite{bercu2023stochastic} propose a stochastic Gauss-Newton algorithm to estimate the entropically regularized Optimal Transport cost between two discrete probability measures. 
\cite{cenac2020efficient} study the asymptotic properties of a stochastic Gauss-Newton algorithm for estimating the parameters of a non-linear regression model. 
\cite{godichon2022recursive} propose second-order algorithms to solve the Ridge regression problem in the linear and logistic framework, 
while the case of the geometric median is introduced and studied by \cite{godichon2023online}. 
In all these algorithms, the estimate of the inverse of the Hessian matrix is recursively computed using the Riccati inversion formula (also called Sherman-Morrison formula, see e.g. \cite{duflo1990methodes} p. 96). 
This calculation is made possible thanks to the particular form of the estimate of the Hessian matrix $H$, presented as $(1/n)\sum_{k=1}^n a_k \phi_k\phi_k^T$, 
where $(a_n)_{n\geq 1}$ is a sequence of positive real random variables 
and $(\phi_n)_{n\geq 1}$ is a sequence of random vectors in $\mathbb{R}^d$.

However, it is not always possible to obtain such an estimate of the Hessian matrix. 
In this work, we propose to construct a direct recursive estimate of $H^{-1}$ 
without first attempting to construct an estimate of $H$. 
This approach is based on the fact that we have $HH^{-1} = H^{-1}H = I_d$ 
and, consequently, the following relation:
\begin{equation}\label{Relation_A}
\bkE\cro{H^{-1} \nabla_h^2 g(X, \theta) + \nabla^2_h g(X, \theta) H^{-1}\ -\ 2 I_d} = 0	
\end{equation}
where $I_d$ denotes the  identity matrix of order $d$. 
Using a Robbins-Monro type algorithm, we propose a recursive estimate of the matrix $H^{-1}$ defined for all $n\geq 1$ by:
$$A_n = A_{n-1} - \gamma_n \pa{ 
	A_{n-1} \nabla^2_h g(X_n, \theta_{n-1}) + \nabla^2_h g(X_n, \theta_{n-1}) A_{n-1} - 2\,I_d}$$
where $(\gamma_n)_{n\geq 1}$ is a sequence of positive real numbers, decreasing towards 0 and $\theta_{n-1}$ is an estimate of $\theta$.

However, the complexity of computing this estimate is of order $\mathcal{O}(d^3)$, 
which is the same as directly calculating the inverse of an estimate of matrix $H$. 
Nevertheless, we can introduce an algorithm with complexity of order $\mathcal{O}(d^2)$ based on the following observation: let $Z$ be a centered random vector in $\mathbb{R}^d$ with variance-covariance matrix $I_d$, independent of the vector $X$. Then, 
\begin{equation}\label{Relation_A_avec_Z}
	\bkE\cro{H^{-1}\,Z\,Z^T \nabla_h^2 g(X, \theta)\ + \nabla^2_h g(X, \theta) Z\,Z^T H^{-1}\ -\ 2 I_d} = 0 .
\end{equation}
Therefore, considering a sequence $(Z_n)_{n\geq 1}$ of random vectors in $\mathbb{R}^d$ independent of the sequence $(X_n)_{n\geq 1}$ leads to an estimate of the form:
$$A_n = A_{n-1} - \gamma_n \pa{ 
	A_{n-1} Z_n Z_n^T \nabla^2_h g(X_n, \theta_{n-1}) + \nabla^2_h g(X_n, \theta_{n-1}) Z_n Z_n^T A_{n-1} - 2\,I_d}.$$

We thus obtain a universal estimate of the inverse of the Hessian, 
and thus with reduced calculus time. 
 {
Note that one could consider using random matrices instead of vectors, and this injection can be linked to Randomized Subspace Newton methods (see, for instance, \cite{gower2019rsn}). 
Additionally, various approaches exist for sketching the Hessian (see \cite{pilanci2016iterative, pilanci2017newton}, among others).
Notably, in the offline iterative literature, several methods approximate the inverse of the Hessian. One of the most well-known is the BFGS algorithm \citep{shanno1970conditioning}, along with its stochastic variants \citep{schraudolph2007stochastic, byrd2016stochastic}. Other notable approaches include those proposed by \cite{ye2017approximate} and \cite{agarwal2017second}.}

To further enhance convergence rate, 
we also consider its weighted averaged version, as discussed by \cite{mokkadem2011generalization,boyer2022asymptotic}. 
We establish the almost sure rates of convergence for the proposed estimates, after making slight modifications. 
These results remain true for any consistent estimates $ \theta_{n} $. 
Based on this concept, we introduce a universal recursive Newton algorithm 
and its weighted averaged version. 
Additionally, we provide their convergence rates and demonstrate the asymptotic efficiency of the weighted averaged estimates.

This paper is organized as follows:
Section 2 concerns the framework and the main assumptions.
Section 3 deals with the estimation of  $H^{-1}$ and the main convergence results
while Section 4 concerns the Universal Weighted Averaged Stochastic Newton algorithm. A simulation study highlights the performance of the proposed methods in Section \ref{sec::simu}.  
The proofs of the different results are postponed  in Section \ref{proof}.

%% file: Framework.tex
\section{Framework}
We consider the problem of minimizing the convex function $G : \mathbb{R}^d \longrightarrow \mathbb{R}$ 
defined for all $h \in \mathbb{R}^d$ by:
$$G(h):=\mathbb{E}\cro{g(X,h)},$$
%
 {where $g(X, \cdot)$ is a convex function and $X$ is a random vector of $\bkR^p$. Furthermore, let us suppose that for all $h \in \mathbb{R}^{d}$, there is a subset $\Omega_{h}  \subseteq \mathbb{R}^{p}$ such that for all $x \in \Omega_{h}$, the functional $g(x,.)$ is twice-differentiable at $h$ and $\mathbb{P} \left[ X \in \Omega_{h} \right] = 1$.
}
We assume also that there exists a unique value $\theta \in \bkR^d$
such that 
$$\nabla G(\theta) = 0. $$
This assumption, couple with strict convexity, ensures the existence of a minimizer for $G$ 
and provides a well-defined optimization problem. 
Now, let's introduce the assumptions that underlie the parameter estimation framework for $\theta$:

\finl
\textbf{(A1a)} There exists $C>0$ such that for all $h \in \mathbb{R}^d$,
 \[
 \mathbb{E}\cro{\norm{\nabla_hg(X,h)}^2}\le C\pa{1+G(h)-G(\theta)}.
 \]
{\noindent \textbf{(A1b)} There are positive constants $R,C_{R}>0$ such that for all $h \in \mathbb{R}^d$,
 \[
 \mathbb{E}\cro{\norm{\nabla_hg(X,h)}^4  \mathbf{1}_{\| h - \theta \| \leq R}}\leq C_{R}.
 \]
 }
 \finl
\textbf{(A2)} The functional $G$ is twice continuously differentiable and  $\nabla^2G(\theta)$ is positive. In addition, the Hessian is uniformly bounded, i.e there exists a positive constant $L_{\nabla G}$ such that for all $h\in\mathbb{R}^d$, 
$$\norm{\nabla^2G(h)}_{op} \le L_{\nabla G}.$$

\noindent
\textbf{(A3)} The function $\nabla^2G$ is Lipschitz on a neighborhood of $\theta$, i.e. there exist positive constants $r>0$ and $L_{r}$ such that for all $h\in\mathcal{B}(\theta,r)$
$$\norm{\nabla^2G(h)-\nabla^2G(\theta)}_{op} \le L_{r}\norm{\theta-h},$$
where $\mathcal{B}(\theta,r)$ denotes a ball of radius $r$ centered at $\theta$.

\noindent
\textbf{(A4)} There exists $q >2$ and $C_q$ such that for all $h \in \mathbb{R}^{d}$, 
\[
\mathbb{E}\cro{\norm{\nabla^2_hg(X,h)}_F^q}\le C_q.
\]

These assumptions are very close to those presented in the literature  \citep{pelletier2000asymptotic,gadat2017optimal,godichon2019online}.
Assumption \textbf{(A1)} controls the growth of the gradient and guarantees the stability of the estimation process. It ensures that the gradient remains bounded as the estimation progresses. {Observe that Assumption \textbf{(A1b)} plays a crucial role in determining the convergence rate of Newton estimates and is already required to establish the almost sure convergence rate of stochastic gradient algorithms \citep{pelletier1998almost}. However, it can be relaxed to accommodate any order, depending on the choice of the step size (see, for instance, \cite{boyer2022asymptotic, pelletier1998almost}). For readability, we set it to 4 in this work. Additionally, note that this assumption differs slightly from equation (10) in \cite{boyer2022asymptotic}, but all results in that paper remain valid under this assumption.}
Assumption \textbf{(A2)} ensures that the curvature of $G$ at $\theta$ is well-behaved, allowing the estimation algorithm to reliably exploit the local structure of $G$  {and is present in most of the offline Newton's literrature (see \cite{byrd2016stochastic,ye2017approximate,agarwal2017second} among others)}. 
This assumption also guarantees that the gradient of $G$ is Lipschitz continuous with a constant $L_{\nabla G}$. 
This Lipschitz continuity is crucial as convergence results are obtained using a Taylor's expansion of $G$ up to the second order.  {Note that Assumption \textbf{(A1a)}, when combined with \textbf{(A2)}, can be related to the well-known expected smoothness condition \citep{gower2019sgd}.}
Assumption \textbf{(A3)} indicates that the Hessian matrix does not exhibit abrupt changes within a neighborhood around $\theta$  { and can be related to self-concordance property \citep{bach2014adaptivity}}. 
This assumption ensures the stability of the Hessian estimates during the estimation process.
Assumption \textbf{(A4)} guarantees that the second-order derivative of $G$ does not exhibit excessive fluctuations. 
It imposes a bound on the variation of the Hessian matrix, providing further stability to the estimation algorithm.
It's worth noting that Hölder's inequality leads to
$$\norm{\nabla^2G(h)}_F \le \mathbb{E}\cro{\norm{\nabla^2_hg(X,h)}_F^q}^{1/q} \le C_q^{1/q}.$$ 
Of course, this inequality intertwines with the bound in Assumption \textbf{(A2)}, but we keep the notation $L_{\nabla G}$ for the sake of clarity.

These assumptions, along with the differentiability properties of $g$ and $G$, provide a solid foundation for developing efficient second order methods to solve the minimization problem and obtain reliable estimates of $\theta$.  {Appendix \ref{appendix} provides a detailed verification of these assumptions in the contexts of linear and logistic regressions, as well as for estimating the geometric median and $p$-means.}

%% file: Estimation_A.tex
\def\Achap{\widehat A}
\def\wtheta{\widehat\theta}

\section{Estimation of the Hessian inverse}

In this section, our focus is solely on estimating the inverse Hessian of function $G$ in $\theta$, 
denoted as $H^{-1}$ with $H = \nabla^2 G(\theta)$. 
Even if our motivation is to estimate $H^{-1}$ for proposing a second-order algorithm, 
the estimation of $H^{-1}$ can also be valuable for recursively constructing confidence intervals or significance statistical tests for a component of parameter $\theta$
when the parameter $\theta$ is estimated using another asymptotically efficient algorithm like the averaged stochastic gradient algorithm.
Indeed, in most cases, the asymptotic variance  involved in the central limit theorem generally depends on matrix $H^{-1}$ and its estimation is then required.

Let $(X_n)_{n\geq 1}$ be a sequence of independent random vectors in $\mathbb{R}^p$ with the same distribution as  $X$. Assume first that $\theta$ is known.
From equality (\ref{Relation_A}), 
the matrix $H^{-1}$ satisfies an equation of the form $\Phi(H^{-1}) = 0$. 
We can then employ the Robbins-Monro procedure \citep{robbins1951stochastic} to recursively estimate the parameter $H^{-1}$. 
Denoting this estimator as $\Achap_n$, for any $n\geq 1$, we have:
$$\Achap_n = \Achap_{n-1} - \gamma_n \pa{ 
	\Achap_{n-1} \nabla^2_h g(X_n, \theta) + \nabla^2_h g(X_n, \theta) \Achap_{n-1} - 2\,I_d},$$
where $\Achap_0$ is an arbitrary symmetric positive definite matrix, 
and $\gamma_n = c_\gamma n^{-\gamma}$ with $\frac{1}{2}<\gamma<1$ and $c_\gamma > 0$. 
It is important to note that $\Achap_n$ is symmetric for any $n \geq 1$ due to its construction. 
However, since $\theta$ is unknown, we need to estimate it. 
Assuming we have an efficient recursive estimator $\hat{\theta}_{n-1}$ of $\theta$ (e.g., a stochastic gradient estimator), 
we can easily derive an estimator of $H^{-1}$ using a plug-in procedure:
$$\Achap_n = \Achap_{n-1} - \gamma_n \pa{ 
	\Achap_{n-1} \nabla^2_h g(X_n, \wtheta_{n-1}) + \nabla^2_h g(X_n, \wtheta_{n-1}) \Achap_{n-1} - 2\,I_d}.$$
This estimator is always symmetric but not necessarily positive definite. 
To {partially} ensure positive definiteness, 
we introduce a truncation based on the norm of $\nabla_{h}^{2} g (X_n, \wtheta_{n-1})$, 
leading to the following estimator of $H^{-1}$:
\begin{equation}
\Achap_n = \Achap_{n-1} - \gamma_n \pa{
\Achap_{n-1} \nabla_h^2 g(X_n, \wtheta_{n-1}) + \nabla_h^2 g(X_n, \wtheta_{n-1}) - 2 I_d}
\mathbf{1}_{\acc{\norm{\nabla^2_h g(X_n, \wtheta_{n-1})}_{op}\le \beta_n}} ,
\end{equation}
where $\beta_n = c_\beta n^\beta$ with $\frac{1-\gamma}{q-1}<\beta<\gamma - \frac{1}{2}$ and {$\gamma_{n}\beta_{n} \leq 1/2$}. 
Additionally, this truncation enables control over the smallest eigenvalue of $\Achap_n$, 
which is useful for studying an estimator of the parameter $\theta$ involving the matrix $\Achap_n$. 
This is particularly important in establishing the consistency of the Stochastic Newton algorithm presented in Section \ref{sec::WASNA}.

However, although this estimator is efficient, each update requires matrix multiplications, resulting in a computational complexity of order $\mathcal{O}(d^3)$, which is the same as matrix inversion. 
Hence, it is necessary to improve the complexity of each update of $\Achap_n$.

Building on equality \eqref{Relation_A_avec_Z}, 
considering a sequence $(Z_n)_{n\geq 1}$ of independent and identically distributed 
bounded random vectors of $\bkR^d$ such that $\bkE\cro{Z_n} = 0$ and $\bkE\cro{Z_n Z_n^T} = I_d$,
  and independent of $(X_n)_{n\geq 1}$, 
  we can propose another estimate of $H^{-1}$ defined for any $n\geq 1$ as follows:
\begin{align}
\notag 
P_n & = A_{n-1} Z_n \\
\notag 
Q_n & = \nabla_h^2 g(X_n, \wtheta_{n-1}) Z_n\\
A_n &=  {\Pi_{\beta_{n}'}} \left(  A_{n-1}-\gamma_n \pa{ P_n Q_n^T + Q_n P_n^T - 2\,I_d}
\mathbf{1}_{\{\norm{Q_n} \norm{Z_n}\le \beta_n\}} \right)
\label{An}
\end{align}
where $A_0$ is an arbitrary symmetric and positive definite matrix  {and $\Pi_{\beta_{n}'}$ is the projection onto the ball of radius $\beta_{n}'$ (with respect to the Frobenius norm), with $\beta_{n}' = c_{\beta}'n^{\beta '}$ satisfying for all $n$, $\gamma_{n}^{2}\beta_{n}^{2}\beta_{n}' \leq \gamma_{n}$ and $\beta ' > 1- \gamma$ (i.e. one has to take $\beta' \in (1-\gamma , \gamma -2 \beta)$, which is possible as soon as $\beta < \gamma -1/2$). Observe that the projection step only necessitates $O(d^{2})$ operations and ensures the largest eigenvalue of $A_{n}$ not to be too large, and also enables to ensure the positive definiteness of $A_{n}$}.

We can observe that in this algorithm the truncation is only based
on $\norm{Q_n}\norm{Z_n}$,
and not on $\norm{Q_n Z_n^T}_{op}$ as expected,
because we have
$\norm{Q_n Z_n^T}_{op} = \norm{Q_n} \norm{Z_n}$.
Notably, the computational complexity of each update of $A_n$ is now reduced to $\mathcal{O}(d^2)$. 
Moreover, following \cite{mokkadem2011generalization,boyer2022asymptotic}, we can propose a weighted averaged estimate $A_{n,\tau}$, which performs better in practice when the initialization is poor. It is given by
\begin{align}
A_{n,\tau} &= \left(1-\frac{\ln (n+1)^\tau}{\sum_{k=0}^n \ln (k+1)^\tau}\right) A_{n-1,\tau} + \frac{\ln (n+1)^\tau}{\sum_{k=0}^n \ln (k+1)^\tau}A_n . \label{Antau}
\end{align}
This estimator can be recursively computed since for any $n \geq 1$, 
  $\sum_{k=0}^n \ln (k+1)^\tau = \ln (n+1)^\tau + \sum_{k=0}^{n-1} \ln (k+1)^\tau$.
The following theorem establishes the consistency of the estimators $A_n$ and $A_{n,\tau}$ 
for the parameter $H^{-1}$ in the context of estimating the inverse of the Hessian matrix. 
It states that the convergence rate depends on multiple factors, 
including the step sequence $\gamma_n$, the truncation parameter $\beta_n$, the regularization parameter $\tau$, and the convergence rate of the estimate $\wtheta_n$. 

\begin{theo}\label{consisAn}
	Assume that Assumptions \textbf{(A2)} to \textbf{(A4)} hold, { that $\frac{1-\gamma}{q-1}<\beta<\gamma - \frac{1}{2}$,} and that there is an estimate $\hat{\theta}_n$ satisfying for all $\delta >0$
	$$\norm{\hat{\theta}_n-\theta}^2 = o\pa{\frac{\ln n^{1+\delta}}{n^a}} a.s.,$$
	with $a >0$. Then $A_n$ and $A_{n,\tau}$ defined by \eqref{An} and \eqref{Antau} satisfy 
	$$\norm{A_n - H^{-1}}^2 = o\pa{\frac{\ln n^{1+\delta}}{n^{\min\{\gamma,a,2\beta (q-1)\}}}} a.s. \quad\text{and}\quad \norm{A_{n,\tau} - H^{-1}}^2 = o\pa{\frac{\ln n^{1+\delta}}{n^{\min\{1,a,2\beta (q-1)\}}}} a.s. $$
\end{theo}
The proof is given in Section \ref{proof}. 
Observe that if $a=1$, 
one can achieve the usual rate of convergence taking $\beta > \frac{1}{2(q-1)}$, 
which is only possible if $q > 1+ \frac{1}{2\gamma -1} $. 
For instance, taking the usual parametrization $\gamma = 3/4$ or $2/3$, 
$q$ must satisfy $q > 3$ or $q> 4$.

%% file: WASNA.tex
\section{Universal Weighted Averaged Stochastic Newton Algorithm}\label{sec::WASNA}

In this section, we introduce the Universal Weighted Averaged Stochastic Newton algorithm 
and discuss its main properties. 
As mentioned in Theorem \ref{consisAn}, using a Weighted Averaged version of the inverse Hessian estimate can yield improved theoretical results. 
Therefore, we incorporate this choice into the Stochastic Newton algorithm. 
Furthermore, we have observed that the convergence rate of the estimate for $\theta$ significantly influences the theoretical behavior of the estimates of $H^{-1}$. 
Consequently, we incorporate the best possible estimate for parameter $\theta$, 
 namely the Weighted Averaged Stochastic Newton estimates, into the latter.
This reasoning leads to the following Weighted Averaged Stochastic Newton algorithm defined for all $n \geq 1$ by 
\begin{align}
\notag 
P_n & = A_{n-1} Z_n \\
\notag 
Q_n & = \nabla_h^2 g(X_n, \theta_{n-1,\tau '}) Z_n \\
	\theta_{n} &= \theta_{n-1} -\nu_{n}A_{n-1,\tau}\nabla_h g(X_n, {\theta}_{n-1}) \label{thetan}\\
	\theta_{n,\tau'} &= \pa{1-\frac{\ln (n+1)^{\tau'}}{\sum_{k=0}^n \ln (k+1)^{\tau'}}}\theta_{n-1,\tau'}+\frac{\ln (n+1)^{\tau'}}{\sum_{k=0}^n \ln (k+1)^{\tau'}}\theta_{n}  \label{thetatau} \\
		A_n &=  {\Pi_{\beta_{n}'}} \left( A_{n-1}-\gamma_n \pa{ P_n Q_n^T +  Q_n P_n^T - 2\,I_d}
	\mathbf{1}_{\acc{\norm{Q_n}\norm{Z_n} \le \beta_n}} \right) \label{AnNS}\\ 
	A_{n,\tau} &= \left(1-\frac{\ln (n+1)^\tau}{\sum_{k=0}^n \ln (k+1)^\tau}\right)A_{n-1,\tau} + \frac{\ln (n+1)^\tau}{\sum_{k=0}^n \ln (k+1)^\tau}A_n \label{AntauNS} 
\end{align}
where  $( \nu_n)_{n\geq 1}$  is a sequence of positive real numbers 
defined for any $n\geq 1$ by $\nu_n = c_\nu n^{-\nu}$ with $c_\nu >0$ and $ \nu \in (1/2,1-\beta)$ satisfying $\gamma + \nu > 3/2$. In addition,  $ \tau , \tau'  \ge 0$. 
The following theorem gives the consistency of the estimates defined by \eqref{thetan} and \eqref{thetatau}.

\begin{theo}\label{consisTheta}
	Assume that Assumptions \textbf{(A1a)}, \textbf{(A2)} to \textbf{(A4)} hold.  {Suppose also that $\frac{1-\gamma}{q-1}<\beta<\gamma - \frac{1}{2}$, $\gamma + \nu > 3/2$, $\nu \in (1/2,1-\beta)$ and $\beta ' \in (1-\gamma, \gamma-2\beta)$.}
   Let $\theta_{n}$ and $\theta_{n,\tau'}$ be defined as in \eqref{thetan} and \eqref{thetatau}. Then, 
	$$\theta_{n}\xrightarrow[n\rightarrow\infty]{a.s.}\theta  \quad\text{and}\quad \theta_{n,\tau'}\xrightarrow[n\rightarrow\infty]{a.s.}\theta. $$
\end{theo}
The proof is given in Section \ref{proof}. 
Note that the constraint $\gamma + \nu > 3/2$ is of a purely technical nature
and is crucial for the application of the Robbins-Siegmund Theorem and so that to get the consistency of the estimates. However, we believe this condition might not be necessary in practical applications.
We can now give the almost sure rate of convergence of the estimates.
\begin{theo}\label{rateTheta}
	Assume that Assumptions \textbf{(A1)} to \textbf{(A4)} hold. {Suppose also that $\frac{1-\gamma}{q-1}<\beta<\gamma - \frac{1}{2}$, $\gamma + \nu > 3/2$, $\nu \in (1/2,1-\beta)$ and $\beta ' \in (1-\gamma, \gamma-2\beta)$.}
	Then $\theta_{n}$ and $\theta_{n,\tau'}$ defined by \eqref{thetan} and \eqref{thetatau} satisfy for all $\delta >0$
	$$\norm{\theta_{n} - \theta}^2 = o\pa{\frac{\ln n^{1+\delta}}{n^{\nu}}} a.s. \quad\text{and}\quad \norm{\theta_{n,\tau'} - \theta}^2 = o\pa{\frac{\ln n^{1+\delta}}{n}} a.s. $$
	In addition, $A_n$ and $A_{n,\tau}$ defined by \eqref{AnNS} and \eqref{AntauNS} satisfy
	$$\norm{A_n - H^{-1}}^2 = o\pa{\frac{\ln n^{1+\delta}}{n^\gamma}} a.s.  \quad\text{and}\quad \norm{A_{n,\tau} - H^{-1}}^2 = o\pa{\frac{\ln n^{1+\delta}}{n}} a.s. $$
	Moreover, the estimates $\theta_{n,\tau'}$ defined by \eqref{thetatau} satisfy
	$$\sqrt{n}(\theta_{n,\tau'}-\theta)\xrightarrow[n\rightarrow\infty]{\mathcal{L}}\mathcal{N}(0,H^{-1}\Sigma H^{-1}),$$
	where $\Sigma = \bkE\cro{\nabla_hg(X,\theta)\nabla_hg(X,\theta)^T}$.
\end{theo}
The proof is given in Section \ref{proof}. 
The Universal Weighted Averaged Stochastic Newton estimates so achieve the asymptotic efficiency, 
and so, under very weak assumptions.

\begin{rmq}
Mention that for estimating parameter $\theta$,
it is also possible to consider the following simpler algorithm,
which we refer to as the Universal Stochastic Newton Algorithm,
and  only relies on $A_n$:
\begin{align*}
\widehat{P}_n &= \Achap_{n-1}Z_n \\
\widehat{Q}_n &= \nabla_h^2 g(X_n, \wtheta_{n-1}) Z_n \\
\Achap_n & = \Achap_{n-1} - \gamma_n\pa{ \widehat{P}_n \widehat{Q}_n^T +  \widehat{Q}_n \widehat{P}_n^T -2 I_d}
\mathbf{1}_{\acc{\norm{\widehat{Q}_n} \norm{Z_n}}\le \beta_n}  \\ 
\wtheta_n & = \wtheta_{n-1} - \nu_n \Achap_{n-1}\nabla_h g(X_n , \wtheta_{n-1}). 
\end{align*}
By following the same scheme of proof as for Theorem \ref{rateTheta}, 
one could check that:
$$ \norm{ \wtheta_n - \theta}^2 = O \left( \frac{\ln n}{n^{\nu}} \right) \quad a.s.$$
However, it can be observed that the convergence rate of $\wtheta_n$ is not optimal.
Nevertheless, this algorithm has the merit of being much simpler. 
In addition, mention that following the reasoning presented by \cite{bercu2020efficient}, 
one could take a step sequence of the form  $\nu_{n} = \frac{1}{n}$ 
leading to the Stochastic Newton algorithm.
However, we are unfortunately not able to obtain the consistency of the estimates in this context.
\end{rmq}

%% file: Applications.tex
\section{Applications}\label{sec::simu}

The simulation section of this paper focuses on evaluating the performance of our novel methods, 
called Universal Stochastic Newton Algorithm (USNA)  and Universal Weighted Averaged Stochastic Newton Algorithm (UWASNA).
We begin by analyzing the performance of USNA and UWASNA
in the context of logistic regression and the geometric median estimation.
In both contexts, it was already feasible to employ second-order algorithms 
such as the Stochastic Newton Algorithm (SNA) and its Weighted Averaged version (WASNA). 
These two algorithms use the Riccati formula to recursively compute the inverse of the Hessian estimator. 
By demonstrating comparable results to SNA and WASNA, 
we establish USNA and UWASNA as viable alternatives with efficient performance. 
Additionally, we investigate the applicability of our method in challenging scenarios where using the Riccati formula in SNA is not feasible, particularly in estimating $p$-means and 
parameters  of a spherical distribution. 
In these two cases,  we compare the performances of USNA and UWASNA with the one of 
the Averaged Stochastic Gradient Descent (ASGD) introduced by \cite{polyak1992acceleration}.
Through comprehensive simulations, our goal is to verify that USNA and UWASNA consistently exhibit favourable performances even in these contexts. 
To close this section, we extend our evaluation to real-world datasets to showcase the practicality of deploying USNA and UWASNA algorithm in applications. 

%
%
\subsection{Choice of the hyperparameters}

In our experiments, the choice of hyperparameters involved in the different algorithms, 
plays a significant role in achieving desired outcomes. 
The decision to set these values is based on both theoretical justifications and empirical observations.

\begin{enumerate}
	\item \textbf{Setting of \( \nu_n \) for USNA}: 
	Despite the lack of a theoretical proof demonstrating the convergence rate of USNA 
	when \( \nu_n=1/n \), this setting is adopted in our experiments for a direct comparison with SNA. 
	Empirical results, as presented later, validate that this choice is effective in practice.%
	\item \textbf{Conditions on \( \beta \) and \( \gamma \)}: 
	 {	In practice, we find that the condition \( \beta < \gamma -1/2 \) is not necessary for the convergences of the estimates, and we advocate for its removal. 
		On the contrary, condition \( \beta_{n+1}\gamma_{n+1} \leq 1/2 \) 
		is essential. 
		It ensures the positiveness of the estimate of the inverse of the Hessian. 
		For this reason, in all simulations, we set \( \beta_{n} = \frac{1}{2}n^{3/4} \) and $\gamma_n = n^{-3/4}$.
}
	\item \textbf{Initialization of Estimators}: 
	For initializing the estimators of the Hessian inverse, we consistently use \( S^{-1}_0 = I_d \) for both SNA and WASNA. 
	Similarly, \( A_0 = I_d \) is chosen for USNA and UWASNA.
\end{enumerate}

\subsection{Comparison with Riccati Newton}

The objective of this section is to demonstrate the comparable performance of USNA and UWASNA when contrasted with SNA and WASNA, particularly in scenarios 
where the use of the Riccati formula is applicable
to recursively compute the inverse  of an estimator of the Hessian.
Let us begin by revisiting this context.
If we can estimate the Hessian matrix $H = \nabla^2 G(\theta)$ by an estimator
of the form $S_n / n$ with $S_n$ defined by 
$S_n = \sum_{j=1}^n \varphi_j \varphi_j^T$ 
where $(\varphi_n)_{n\geq 1}$ is a sequence of random vectors of $\bkR^d$,
then, thanks to the Riccati  formula \ { (also called Sherman-Morrison formula)},
  we can recursively calculate matrix $S_n^{-1}$ for any $n\geq 1$:
\begin{equation}
\label{RiccatiFormula}
S_n^{-1} = S_{n-1}^{-1} - \dfrac{1}{1+\varphi_n^T S_{n-1}^{-1} \varphi_n} S_{n-1}^{-1} \varphi_n \varphi_n^T S_{n-1}^{-1},
\end{equation} 
with $S_0 = I_d$ to avoid the invertibility problem.
This formula finds application in various scenarios, 
as previously demonstrated by \cite{bercu2020efficient}, for instance, 
to obtain efficient stochastic Newton algorithms.
In light of this, we can define the stochastic Newton algorithm (SN) 
for estimating parameter $\theta$ as followed:
\begin{align}
\label{SNA1}
U_n & = S_{n-1}^{-1} \varphi_n  \\
\label{SNA2}
S_n^{-1} & = S_{n-1}^{-1} - (1 + \varphi_n^T U_n)^{-1} U_n U_n^T \\
\label{SNA3}
\theta_n^{SN} &= \theta_{n-1}^{SN} - S_n^{-1} \nabla_{h} g(X_n, \theta_{n-1}^{SN}) 
\end{align}
where $S_0^{-1} = I_d$ and $\theta_0^{SN}$ 
is arbitrarily chosen.
 Note that the random vector $\varphi_n$ is dependent on the current observation $X_n$ 
 and the previous estimation $\theta^{SN}_{n-1}$.
The Weighted Averaged Stochastic Newton Algorithm is defined by: 
\begin{align*}
	\overline{U}_n & = \overline{S}_{n-1}^{-1} \overline{\varphi}_n  \\
	\overline{S}_n^{-1} & = \overline{S}_{n-1}^{-1} - (1 + \overline{\varphi}_n^T \overline{U}_n)^{-1} \overline{U}_n\overline{U}_n^T \\
    \overline{\theta}_n&= \overline{\theta}_{n-1} - \gamma_{n+1}\overline{S}_n^{-1} \nabla_hg(X_n, \overline{\theta}_{n-1}) \\
    \theta_n^{WASN} &= \pa{1-\frac{\ln (n+1)^{\tau'}}{\sum_{k=0}^n \ln (k+1)^{\tau'}}}\theta_{n-1}^{WASN}+\frac{\ln (n+1)^{\tau'}}{\sum_{k=0}^n \ln (k+1)^{\tau'}} \overline{\theta}_n  
\end{align*}
where $S_0^{-1} = I_d$, $\overline{\theta}_0$ and $\theta_0^{WASN}$
are arbitrarily chosen,
the random vector $\overline{\varphi}_n$ is dependent on the current observation $X_n$ 
and the previous estimation $\theta_{n-1}^{WASN}$.

\subsubsection{Logistic regression}

Let $(X, Y)$ be  a random vector taking values in $\bkR^p \times \{0, 1\}$
and let us set $\phi = (1, X^T)^T$.
In the binary logistic regression framework,
function $G$ to minimize is defined for any $h \in \bkR^{p+1}$ by:
$$G(h) = \bkE\cro{\log( 1+ \exp( h^T \phi )) - h^T\phi Y} 
= \bkE\cro{g(X, Y, h)}$$
where the conditional distribution of the binary response  $Y$ knowing $\phi$ is a Bernoulli distribution
of parameter $\pi(\theta^T \phi)$ with for any $x\in\bkR$, $\pi(x) = \exp(x) / (1  + \exp(x))$
and $\theta \in \bkR^{p+1}$ is the unknown parameter to be estimated.
It is easy to show that
$$\theta = \arg\min_{h \in \bkR^{p+1}} G(h)$$
and $\theta$ is the unique solution of equation $\nabla G(h) = 0$.
In addition, we have
$$H = \bkE\cro{a(\theta^T \phi) \phi \phi^T}\quad \mbox{\rm with}\quad 
a(z) = \pi(z) (1 - \pi(z)).$$
Let $(\phi_n,Y_n)_{n\geq 1}$ be a sequence of independent random vectors in $\bkR^{p+1} \times \{0, 1\}$ 
with the same distribution as  $(\phi, Y)$.
In this context, \cite{bercu2020efficient}
uses algorithm (\ref{SNA1})-(\ref{SNA3}) to estimate $\theta$
with $\varphi_n = \sqrt{a(\phi_n^T \theta_{n-1}^{SN})} \phi_n$ and 
$\nabla_{h} g(X_n, \theta_{n-1}^{SN}) = - \Phi_{n} \left( Y_n - \pi(\phi_n^T\theta_{n-1}^{SN}) \right)$.

We conduct extensive simulations to evaluate the performance of our novel methods, USNA and UWASNA, 
in comparison to SNA and WASNA.
For this purpose,  we consider the logistic regression model introduced by \cite{bercu2020efficient} 
where $p=10$ and the true coefficients are set as follows:
$$\theta= (0, 3, -9, 4, -9, 15, 0, -7, 1, 0)^T.$$
We compare USNA and UWASNA against SNA and WASNA in terms of their ability to accurately estimate the true coefficients. 
The evaluation of the algorithms' performance is carried out using the mean squared error (MSE) metric. 
We simulate $N=100$ independent sample of size $n=10\,000$.
The results are averaged to mitigate the effects of sampling fluctuations.
For all algorithms, we initialize the estimate of the parameter with $\theta_{init} = \theta + e\epsilon,$
where $\epsilon \sim \mathcal{N}\pa{0,I_{p+1}}$ and $e=1$ or $2$.

\begin{figure}[h!]
	\centering
	\includegraphics[width=0.8\textwidth]{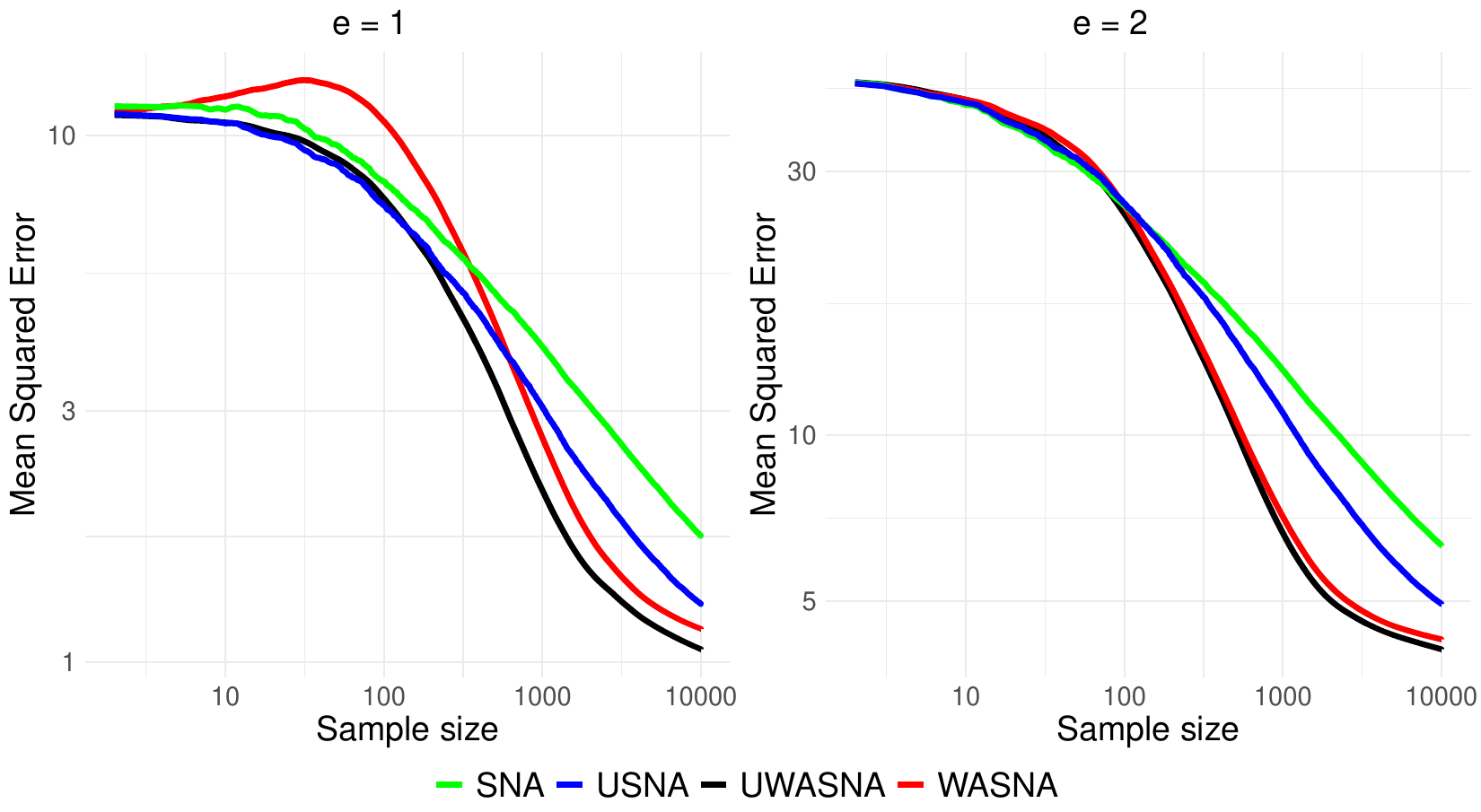}
	\caption{Evolution of the mean squared error with respect to the sample size for logistic regression.}
	\label{graph::log}
\end{figure}

As shown in the figure \ref{graph::log}, 
the weighted averaged estimators converge more rapidly than the other two methods.
Both USNA and UWASNA perform comparably to SNA and WASNA in accurately estimating the true coefficients of the logistic regression model. 
Remarkably, USNA and UWASNA achieve this without using the Riccati formula.

\subsubsection{Geometric Median}
Next, we conduct simulations in the context of geometric median estimation for a multivariate distribution.
We focus on the model introduced by \cite{godichon2023online}. 
We generate $n=10\,000$ copies of the random vector $X$ of $\mathbb{R}^p$ with $p=10$, where $X \sim \mathcal{N}\pa{0,\mathbf{\Sigma}}$ with $\mathbf{\Sigma}_{i,j} = |i-j|^{0.5}$. 
Recall that the geometric median is defined by:
$$m = \argmin_{h\in\mathbb{R}^p} \mathbb{E}\cro{\norm{X-h}-\norm{X}}.$$
In this model, the result leads to  $m = (0, \ldots,0)^T$ in this model. 
In addition, the Hessian matrix $H$ is defined by:
$$H = \mathbb{E}\cro{\frac{1}{\norm{X-m}}\pa{I_p - \frac{(X-m)(X-m)^T}{\norm{X-m}^2}}}.$$
In this context, algorithm (\ref{SNA1})-(\ref{SNA3}) can be used to estimate $m$
taking $\nabla_{h} g(X_n, m_{n-1}^{SN}) = -\frac{X_{n}-m_{n-1}}{\norm{X_{n}-m_{n-1}}}$ and
$\varphi_n = \frac{\sqrt{\norm{X_n-m_{n-1}}}}{\alpha_n}\pa{\nabla_{h} g(X_n, m_{n-1}^{SN}+\alpha_n Z_n) - \nabla_{h} (X_n, m_{n-1}^{SN})}$, where $\alpha_n = \frac{1}{n\ln{(n+1)}}$ and $(Z_n)_{n \geq 1}$ is a sequence of independent standard Gaussian vectors.

For a comprehensive evaluation of our methods' effectiveness in geometric median estimation, we also compare them against two baselines : SNA and WASNA.
For all four algorithms, we initialize the estimate of the geometric median with
$m_{init} = e\epsilon,$
where $\epsilon \sim \mathcal{N}\pa{0,I_{p+1}}$ and $e=1$ or $2$.
Throughout the simulations, we recorded the MSE of the estimated medians for the four algorithms. 
The simulation results were averaged over multiple iterations ($N = 100$).

\begin{figure}[h!]
	\centering
	\includegraphics[width=0.8\textwidth]{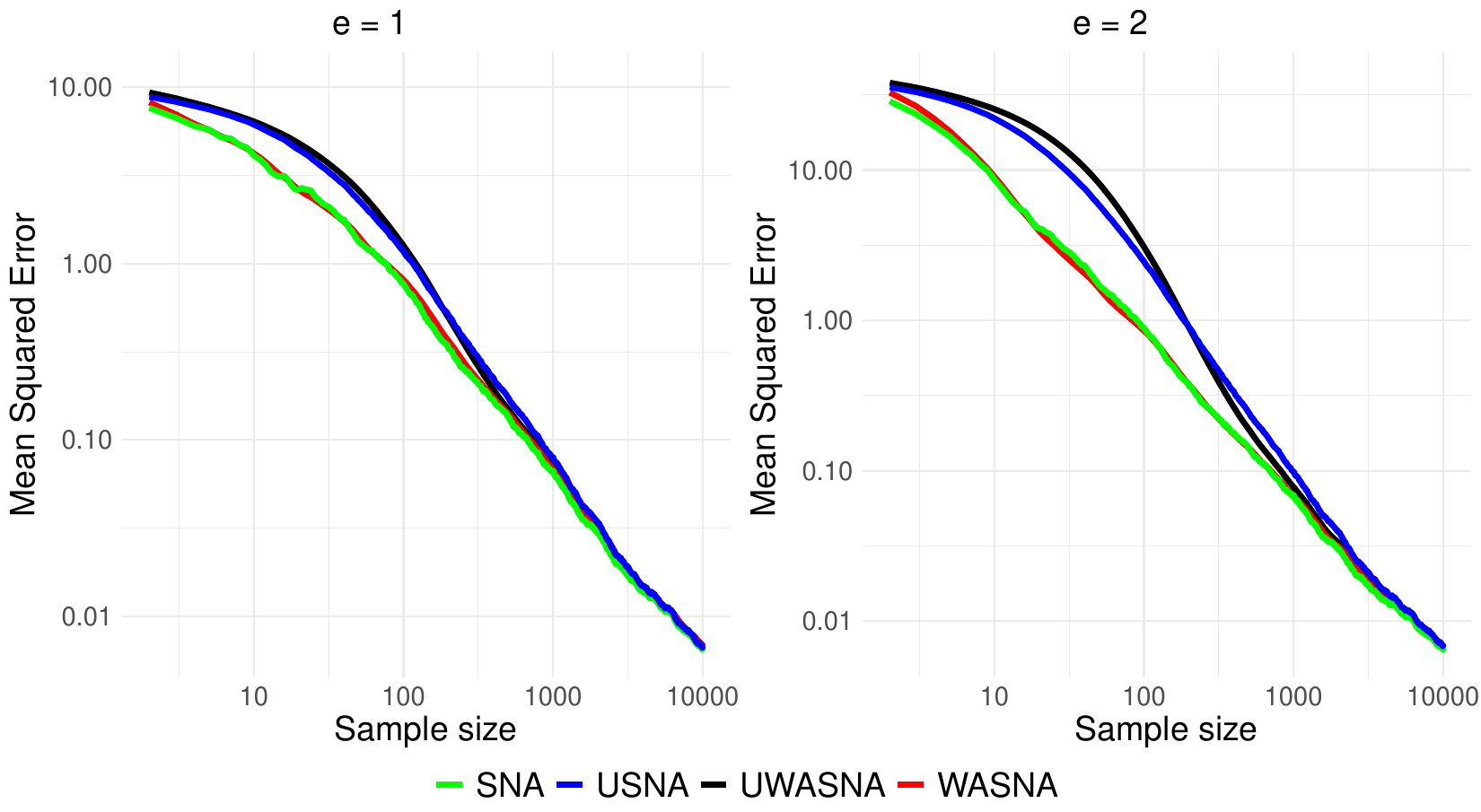}
	\caption{Evolution of the mean squared error with respect to the sample size for geometric median estimation.}
	\label{graph::median}
\end{figure}

In Figure \ref{graph::median}, the results consistently indicate that when the sample size is relatively small (around 100), USNA and UWASNA converge slightly slower. However, beyond that point, they achieve performance on par with both WASNA and SNA in estimating the geometric median.

\subsection{Cases where the Riccati formula cannot be used}
In this section, we focus on the scenarios where using the Riccati formula in SNA or WASNA is not feasible. 
We explore the performance of USNA and UWASNA in contrast to an alternative 
averaged stochastic gradient-based method ("ASGD") proposed by \cite{polyak1992acceleration}, 
which is defined as:
\begin{align}
	\label{ASGD1}
	\theta_n^{SGD} & = \theta_{n-1}^{SGD}- \eta_n \nabla_{h} g(X_n, \theta_{n-1}^{SGD})  \\
	\label{ASGD2}
	\theta_n^{ASGD} &= \theta_{n-1}^{ASGD} + 1/n(\theta_n^{SGD}-\theta_{n-1}^{ASGD}) 
\end{align}
where $(\eta_n)_{n\geq 1}$ is a sequence of learning rates and 
$\theta_0^{ASGD}=\theta_0^{SGD}$.

\subsubsection{Spherical Distribution}
In this paragraph,  we focus on the estimation of the parameters of a spherical distribution
\citep{godichon2017averaged}.  {This is an interesting example since we are able to explicitely calculate $\Sigma^{-1/2}$. Nevertheless we are not able to verify that our assumptions are satisfied in this context.}
The aim of the task is to fit a sphere onto a 3D point cloud with noise. 
In this context, we assume that the observations represent independent realizations of a random vector $X$, which is defined as 
$$X= \mu + rWU,$$
where $U$ is uniformly distributed on the unit sphere of $\mathbb{R}^3$, 
$W\sim\mathcal{U}\pa{[1-\delta,1+\delta]} $ with $\delta>0$, $W$ and $U$ are independent. The radius $r>0$ and the center $\mu \in \mathbb{R}^3$ are parameters to be estimated.
The unknown parameter $\theta = (\mu,r)^T$ is a local minimizer of the function 
$G:\mathbb{R}^3\times\mathbb{R}_+^* \longrightarrow \mathbb{R}$ 
defined for all $h=(a,b) \in \mathbb{R}^3\times\mathbb{R}_+^*$ by:
$$G(h) := \mathbb{E}\cro{g(X,h)} =\frac{1}{2}\mathbb{E}\cro{\pa{\norm{X-a}-b}^2}.$$
In this scenario, one can use algorithm (\ref{ASGD1})-(\ref{ASGD2}) to estimate $\theta$ with
$$\nabla_hg(X,h)= \begin{pmatrix}
	a-X+\frac{b(X-a)}{\norm{X-a}} \\
	b-\norm{X-a}
\end{pmatrix}.$$
Second-order methods USNA and UWASNA can also be used with
$$\nabla^2_hg(X,h)=\begin{pmatrix}
	\pa{1-\frac{b}{\norm{X-a}}}I_3+\frac{b(X-a)(X-a)^T}{\norm{X-a}^3} & \frac{X-a}{\norm{X-a}}\\
	\frac{(X-a)^T}{\norm{X-a}} & 1
\end{pmatrix}.$$
We emphasize that in this specific case, neither the conventional Stochastic Newton Algorithms (SNA) nor the Weighted Averaged version (WASNA) using the Riccati formula are applicable due to the nature of the problem.  Thus, we conduct simulations to compare the performance of the USNA, UWASNA, and ASGD.
The synthetic datasets were generated with a sample size of $10\,000$ and the true parameters of the spherical distribution were set as follows: $\mu = (0,0,0)$ and $r=2$. In addition, we set $\delta=0.2$, which results in a Hessian matrix with eigenvalues of different order sizes. For all three algorithms, we initialize the estimate of the parameter by
$$\theta_{init} = (\mu_{init},r_{init})^T = (0,0,0,2)^T + e\epsilon $$
where $\epsilon \sim \mathcal{N}\pa{0,I_{4}}$ and $e =0.5$ or $1$.
Multiple iterations ($N = 100$) were performed to reduce the impact of sampling variations.

\begin{figure}[h!]
	\centering
	\includegraphics[width=0.8\textwidth]{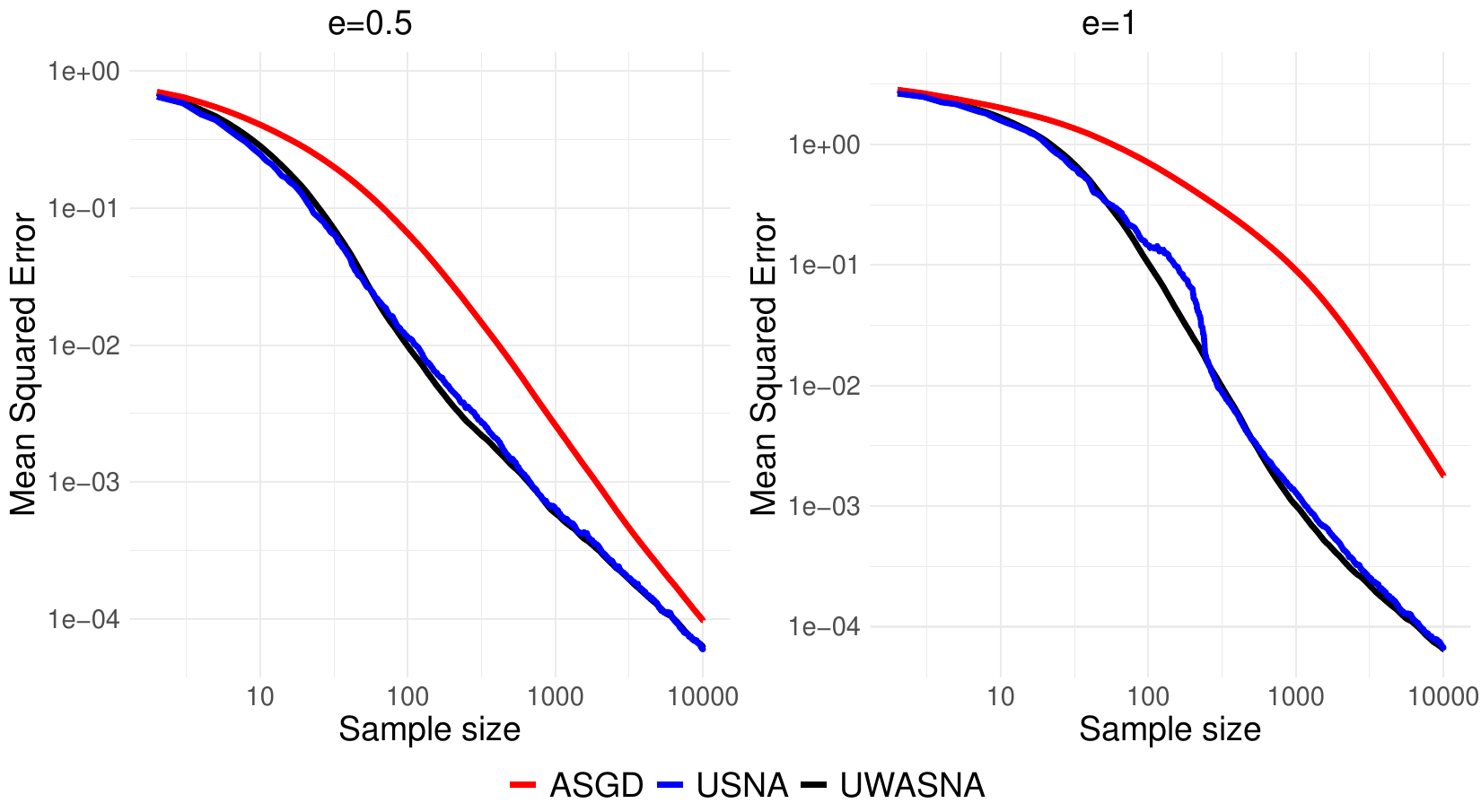}
	\caption{Evolution of the mean squared error with respect to the sample size for parameters estimation in a spherical Gaussian distribution.}
	\label{graph::sphere}
\end{figure}

Figure \ref{graph::sphere} illustrates the comparison between the performances of USNA, UWASNA and ASGD in terms of the mean squared error (MSE) of the estimated parameters. 
Throughout the simulations, UWASNA and USNA demonstrate superior performance in accurately estimating the parameters of the spherical distribution when compared to the gradient-based method.  
Additionally, the Hessian matrix $H$ of the model can be explicitly calculated \citep{godichon2017averaged}, one has
$$H = \begin{pmatrix}
	I_3-\frac{2}{3}I_3\mathbb{E}\cro{W}\mathbb{E}\cro{W^{-1}} & 0 \\
	0 & 1
\end{pmatrix}.$$
Note that this matrix is diagonal, making its inverse computation straightforward. 
Therefore, we investigate the Frobenius norm of the difference between the estimated matrix $A_n$ and the true matrix $H^{-1}$.  From Figure \ref{graph::sphereAn}, it's evident that our methods provide a good estimation of the inverse of the Hessian matrix. Moreover, UWASNA offers a better estimation compared to USNA.

\begin{figure}[h!]
	\centering
	\includegraphics[width=0.8\textwidth]{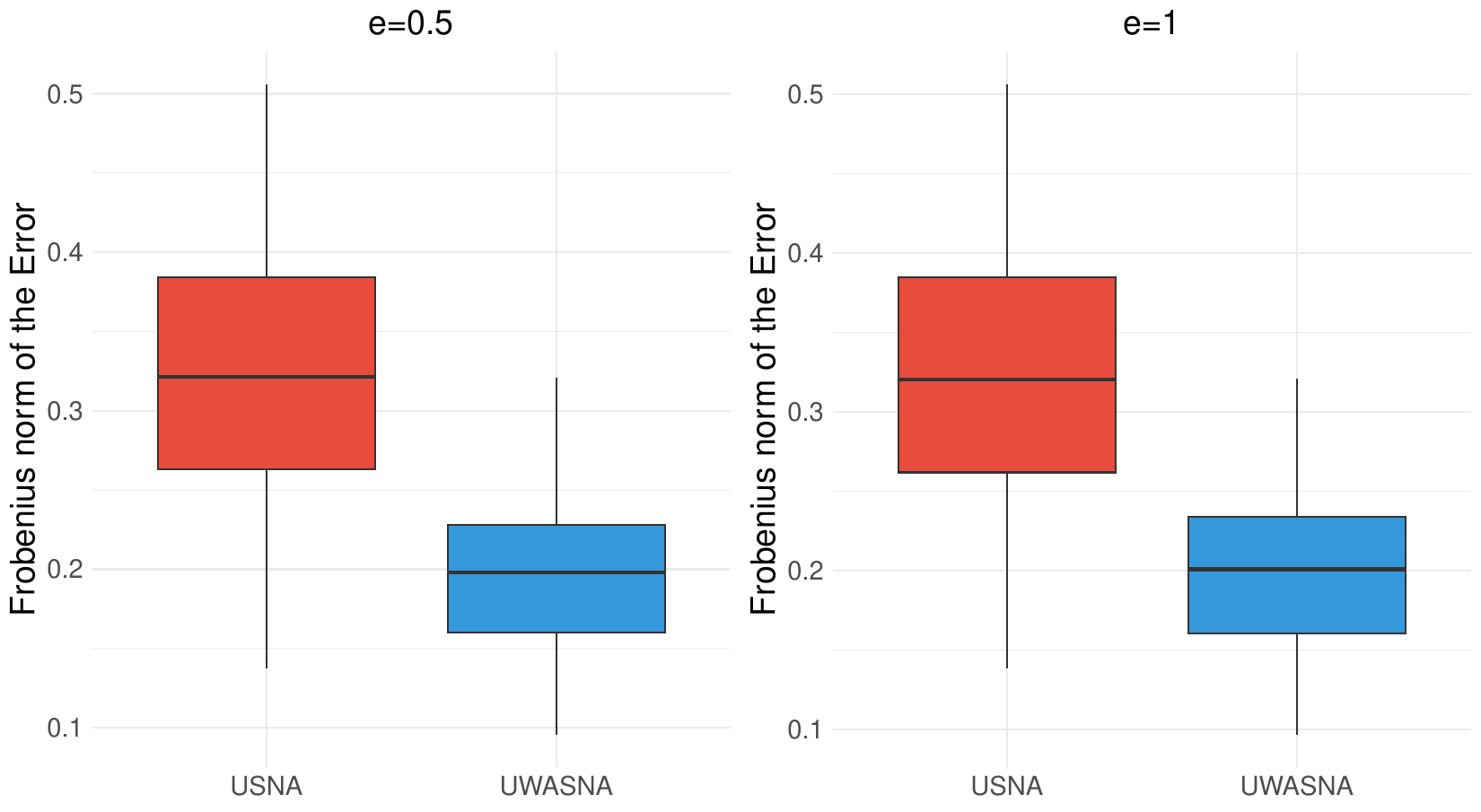}
	\caption{Frobenius norm of the difference between the estimated matrix $A_n$ and the true matrix $H^{-1}$.}
	\label{graph::sphereAn}
\end{figure}

\subsubsection{p-means}
Now we focus on the estimation of p-means of a multivariate distribution
\citep{frechet1948elements}. 
We consider a random vector $X$ of $\mathbb{R}^d$ with $d=40$, 
where $X \sim \mathcal{N}\pa{0,\mathbf{\Sigma}}$ 
with $\mathbf{\Sigma}_{i,j} = |i-j|^{0.5}$. 
The p-mean $m$ of $X$  is defined as the minimizer of the 
functional $G_p : \mathbb{R}^d \rightarrow \mathbb{R}$ 
given for all $h \in \mathbb{R}^d$ by:
$$G_p(h) = \dfrac{1}{p}\,\mathbb{E}\cro{\norm{X-h}^p},$$
where $1 \le p < +\infty$.
Note that in our model $m = (0,\ldots,0)^T$. 
We can easily verify that the gradient and the Hessian of $G_p$ are given by:
\begin{align*}
	\nabla G_p(h) &= -\mathbb{E}\cro{(X-h)\norm{X-h}^{p-2}},\\
	\nabla^2 G_p(h) &= \mathbb{E}\cro{\norm{X-h}^{p-2}\pa{I_d-(2-p)\frac{(X-h)(X-h)^T}{\norm{X-h}^2}}}.
\end{align*}
We aim to compare the performance of USNA, UWASNA, and ASGD 
for estimating  the p-mean of $X$ through simulations.
We consider the case  $p=1.5$ and
we simulate $N=100$ independent samples of size $n=10\,000$.
For all three algorithms, we initialize the estimate with $m_{init} = e\epsilon$ 
where $\epsilon \sim \mathcal{N}(0,I_d)$ and $e=1$ or $2$.

\begin{figure}[h!]
	\centering
	\includegraphics[width=0.8\textwidth]{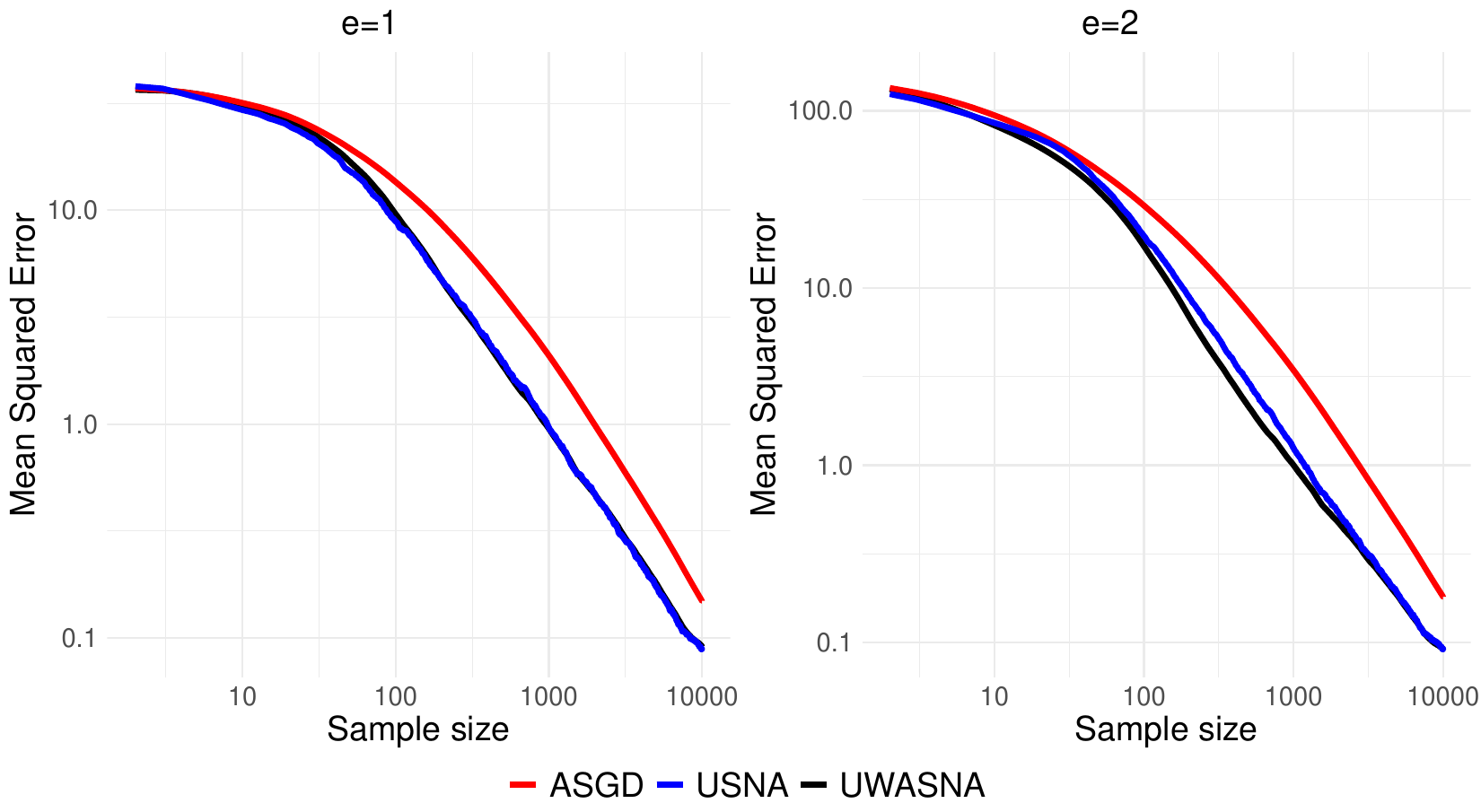}
	\caption{Evolution of the mean squared error with respect to the sample size for p-means estimation.}
	\label{graph::pmeans}
\end{figure}

As depicted in the Figure \ref{graph::pmeans}, 
we plotted the MSE versus sample size. 
It is evident that USNA and UWASNA consistently outperforms ASGD in estimating the p-means. 
Their superior performance can be attributed to the incorporation of information from the Hessian matrix. 
These results further emphasize the advantage of using USNA and UWASNA as alternative methods when it is impossible to use the Riccati formula.

\subsection{Application to real data}
 {
We now compare the performance of various algorithms across several real-world datasets. 
Table~\ref{tab:datasets} summarizes the main characteristics of the different dataset used in our experiments.
\begin{table}[H]
	\centering
	\begin{tabular}{lccc}
		\hline
		Dataset   & \# Features & Training Set Size & Testing Set Size \\ 
		\hline
		Epsilon   & 2000        & 400,000           & 100,000          \\
		Mushrooms & 112         & 6,499             & 1,625            \\
		Web Page  & 300         & 3,470             & 46,279           \\
		MNIST     & 784         & 11,378            & 1,924            \\
		Phishing  & 68          & 5,527             & 5,528            \\
		Adult     & 123         & 32,561            & 16,281           \\
		COVTYPE   & 54          & 290,506           & 290,505          \\
		\hline
	\end{tabular}
	\caption{Key characteristics of the datasets used in this study.}
		\label{tab:datasets}
\end{table}
Epsilon \citep{chang2011libsvm} is a large-scale binary classification dataset from the LIBSVM collection.
Mushrooms \citep{lichman2013uci} consists of morphological attributes of mushrooms used to determine their edibility. 
Web Page \citep{platt199912} involves webpage classification and is widely used in SVM-based text classification research. 
MNIST \citep{lecun1998gradient} contains images of handwritten digits (0--9); for binary classification, we use only the digits 2–5. 
Phishing \citep{phishing_websites_327} is designed to detect malicious websites and consists entirely of categorical features, which we encode in binary form. Adult \citep{dua2017uci} contains census data for income prediction and includes multiple categorical variables converted into binary indicators. 
COVTYPE \citep{blackard1998comparison} originally includes multiple forest cover types; in this study, we focus on distinguishing “Spruce/Fir” (labeled as 1) from all other categories (labeled as 0).
These datasets are frequently adopted to benchmark binary classifiers \citep{toulis2017asymptotic,juan2016field,yuan2012improved}. 
While Epsilon, Web Page, MNIST, and Adult already include predefined training and testing sets, we partition COVTYPE, Mushrooms, and Phishing equally to create our own splits. 
Our main objective is to apply a logistic regression model 
to predict the  binary response variable  for each of the seven datasets.
}

 {
We implement the UWASNA, USNA, WASNA and SNA on the training set. 
As a baseline, we also apply first-order algorithms SGD and AdaGrad on it. 
Since these methods are designed to be online (recursive), 
we process the training set in a single pass and do not use any stopping criterion, unlike iterative (offline) methods.
}
We evaluate the performance of each algorithm by calculating the accuracy
(percentage of good classifications) of each one on both training set and test set. 
The results are summarized in Table \ref{tab::real}.
\begin{table}[H]
	\centering
	\small
	\begin{tabular}{llcccccc}
		\hline
		\multirow{2}{*}{Dataset} & \multirow{2}{*}{Set} & \multirow{2}{*}{SNA} & \multirow{2}{*}{WASNA} & \multirow{2}{*}{USNA} & \multirow{2}{*}{UWASNA} & \multirow{2}{*}{AdaGrad} & \multirow{2}{*}{SGD} \\
		& & & & & & & \\
		\hline
		\multirow{2}{*}{Epsilon}   & Train & 89.94 & 90.05 & 84.51 & 79.83 & 87.18 & 71.79 \\
		& Test  & 89.68 & 89.81 & 84.43 & 79.61 & 87.05 & 71.66 \\
		\hline
		\multirow{2}{*}{Mushrooms} & Train & 99.98 & 100.00 & 99.24 & 98.50 & 99.21 & 97.46 \\
		& Test  & 99.83 & 100.00 & 98.87 & 98.84 & 98.82 & 97.02 \\
		\hline
		\multirow{2}{*}{Web Page}  & Train & 98.56 & 98.04 & 96.92 & 97.00 & 97.00 & 96.92 \\
		& Test  & 98.02 & 97.58 & 97.04 & 97.05 & 97.04 & 97.04 \\
		\hline
		\multirow{2}{*}{MNIST}     & Train & 94.28 & 93.08 & 96.36 & 96.53 & 97.21 & 96.99 \\
		& Test  & 94.59 & 92.88 & 96.99 & 96.36 & 97.35 & 97.04 \\
		\hline
		\multirow{2}{*}{Phishing}  & Train & 93.56 & 93.87 & 92.80 & 92.62 & 92.82 & 85.67 \\
		& Test  & 93.38 & 93.89 & 92.58 & 92.42 & 92.84 & 85.60 \\
		\hline
		\multirow{2}{*}{Adult}     & Train & 84.92 & 84.73 & 84.80 & 84.71 & 84.76 & 83.98 \\
		& Test  & 85.03 & 85.00 & 84.94 & 85.17 & 85.09 & 84.10 \\
		\hline
		\multirow{2}{*}{COVTYPE}   & Train & 75.73 & 75.70 & 75.38 & 75.63 & 75.56 & 75.52 \\
		& Test  & 75.86 & 75.78 & 75.50 & 75.69 & 75.67 & 75.65 \\
		\hline
	\end{tabular}
	\caption{Train and Test accuracy (in \%) for SNA, WASNA, USNA, UWASNA, AdaGrad, and SGD on various datasets.}
		\label{tab::real}
\end{table}
 {
We can observe that UWASNA, USNA, WASNA and SNA demonstrate similar performances, and they achieve higher accuracies on both training set and test set compared to SGD. 
By successfully applying USNA and UWASNA on this dataset, 
we illustrate their practicality in real-world applications.
}

%% file: Conclusion.tex
\section*{Conclusion}
 {
In this article, we introduce a novel second-order stochastic algorithm, 
referred to as universal Newton, and designed to estimate the unique minimizer of a convex function. 
Unlike the classical approach, which involves estimating and explicitly inverting the Hessian matrix, our method directly and recursively estimates its inverse using a Robbins-Monro algorithm. 
By incorporating this estimation into the optimization process, 
we develop the universal stochastic Newton algorithm, along with its averaged and weighted variants.
We establish the convergence of these algorithms and their asymptotic efficiency. 
Numerical experiments on simulated data highlight the advantages of our approach over classical methods, 
while tests on real-world datasets confirm its applicability. 
Compared to standard first- and second-order algorithms, 
our method consistently yields comparable results and, in some scenarios, 
even outperforms them.
}

%% file: proof.tex
\section{Proofs}\label{proof}
\subsection{Notations and preliminary definitions}
The purpose of this section is to list the most frequently used notations in the proofs, aiming to facilitate readability.
In the sequel,
$\| \cdot \|$ denotes the Euclidean norm in $\mathbb{R}^{d}$,
while $\| \cdot \|_{op}$ and $\| \cdot \|_{F}$ stand for the spectral norm and Frobenius norm of matrices, respectively.
Throughout this work, we adopt the following notations:
\begin{itemize}
\item[•] $(Z_n)_{n\geq 1}$ is a sequence of independent and identically distributed random vectors of $\mathbb{R}^d$,
independent of the sequence $(X_n)_{n\geq 1}$, and such that for any $n\geq 1$,
$\| Z_{n} \| \leq M$ and $\mathbb{E}\left[ Z_n Z_n^{T} \right] =I_d$;
\item[•] $\hat{\theta}_{n}$ is an estimate of unknown parameter $\theta$;
\item[•] $T_{n+1} := \nabla^2_h g(X_{n+1}, \hat{\theta}_{n})Z_{n+1}Z_{n+1}^T$ ;
\item[•]  $A_n$, the estimate of $H^{-1}$, is recursively defined for all $n \geq 1$ by:
 \begin{align*}
P_{n+1} & = A_{n} Z_n \\
Q_{n+1} & = \nabla_h^2 g(X_{n+1}, \wtheta_{n}) Z_{n+1}\\
A_{n+1} &=  {\Pi_{\beta_{n+1}'}} \left(  A_{n}-\gamma_{n+1} \pa{ P_{n+1} Q_{n+1}^T + Q_{n+1} P_{n+1}^T - 2\,I_d}
\mathbf{1}_{\{\norm{Q_{n+1}} \norm{Z_{n+1}}\le \beta_{n+1}\}} \right) ,
\end{align*}
where $\Pi_{\beta_{n}'}$ is the projection onto the ball of radius $\beta_{n+1}'$ (with respect to the Frobenius norm);
\item[•] $W_{n+1}:=A_{n}T_{n+1}^T + T_{n+1}A_{n}$.
\end{itemize}

Finally, in the proof of Theorems \ref{consisTheta} and \ref{rateTheta},
we will use the following notations

\begin{align*}
\xi_{n+1}  & = -W_{n+1}+\nabla^2G(\hat{\theta}_n)A_n+A_{n}\nabla^2G(\hat{\theta}_n) \\
\Pi_{n+1}^{\perp} & =A_{n+1} - A_{n}+\gamma_{n+1}(W_{n+1}-2I_d)\mathbf{1}_{\norm{Q_{n+1}}\norm{Z_{n+1}}\le \beta_{n+1}} \\
        r_{1,n}&=(\nabla^2G(\hat{\theta}_n)-H)H^{-1}+H^{-1}(\nabla^2G(\hat{\theta}_n)-H),\\
        r_{2,n} &= (W_{n+1}-2I_d)\mathbf{1}_{\norm{Q_{n+1}}\norm{Z_{n+1}}>\beta_{n+1}}-\mathbb{E}\cro{(W_{n+1}-2I_d)\mathbf{1}_{\norm{Q_{n+1}}\norm{Z_{n+1}}>\beta_{n+1}}|\mathcal{F}_n},\\
        s_n &= \mathbb{E}\cro{\pa{T_{n+1}\pa{A_n-H^{-1}}+\pa{A_n-H^{-1}}T_{n+1}}\mathbf{1}_{\norm{Q_{n+1}}\norm{Z_{n+1}}>\beta_{n+1}}|\mathcal{F}_n} \\
        &+\pa{\nabla^2G(\hat{\theta}_n)-H}\pa{A_n-H^{-1}}+\pa{A_n-H^{-1}}\pa{\nabla^2G(\hat{\theta}_n)-H} -\gamma_{n+1}^2H(A_n-H^{-1})H ,\\
        s'_n &=\mathbb{E}\cro{\pa{T_{n+1}H^{-1}+H^{-1}T_{n+1}-2I_d}\mathbf{1}_{\norm{Q_{n+1}}\norm{Z_{n+1}}>\beta_{n+1}}|\mathcal{F}_n} .
\end{align*}


\subsection{Positivity of $A_{n}$}
 {
Let us prove by induction that   $A_n$ is positive semi-definite for any $n\ge0$.
Let us denote by $\tilde{A}_{n+1} := A_n - \gamma_{n+1} \left( A_n T_{n+1}^T + T_{n+1} A_n - 2   I_d \right) \mathbf{1}_{\lVert Q_{n+1} \rVert \leq \beta_n}$.
	If $\lVert Q_{n+1} \rVert > \beta_n$, then $\tilde{A}_{n+1} = A_n$ and is so positive (as well as the projection $A_{n+1}$).
	Conversely, if $\lVert Q_{n+1} \rVert \leq \beta_n$, we have
	\begin{align*}
		\tilde{A}_{n+1} &= A_n - \gamma_{n+1} \left( A_n T_{n+1}^T + T_{n+1} A_n - 2   I_d \right) \\
		&= \left(  I_d - \gamma_{n+1} T_{n+1}\right) A_n \left(  I_d - \gamma_{n+1} T_{n+1}\right)^T
		- \gamma_{n+1}^2 T_{n+1} A_n T_{n+1}^T + 2 \gamma_{n+1}   I_d 
	\end{align*}
Then, since $\gamma_{n+1}^{2}\beta_{n+1}^{2}\beta_{n+1}' \leq \gamma_{n+1}$,
	\begin{align*}
		\lambda_{\min}\left(\tilde{A}_{n+1}\right) &\geq \lambda_{\min}\left( \left(  I_d - \gamma_{n+1} T_{n+1}\right) A_n \left(  I_d - \gamma_{n+1} T_{n+1}\right)^T \right)
		- \gamma_{n+1}^2 \lambda_{\max}\left( T_{n+1} A_n T_{n+1}^T \right) + 2 \gamma_{n+1} \\
		&\geq 0 - \gamma_{n+1}^2 \beta_{n+1}^2 \beta_{n+1}' + 2 \gamma_{n+1} \\
		&\geq \gamma_{n+1} .
	\end{align*}
Then, $\tilde{A}_{n+1}$ is positive (and its projected version $A_{n+1} $ too).
}

\subsection{Study on the largest eigenvalue of $A_n$}
The following proposition provides an initial asymptotic bound of the largest eigenvalue of $A_n$ without requiring  knowledge on the behavior of the estimate $\hat{\theta}_{n}$. This result is crucial to prove  Theorems \ref{consisAn} and  \ref{consisTheta}.  
\begin{prop}\label{grandvp}
	Under Assumptions \textbf{(A3)} and \textbf{(A4)}, the largest eigenvalue of $A_n$ denoted by $\lambda_{\max}(A_n)$ satisfies for all $\delta >0$
	$$\lambda_{\max}(A_n) =  o\pa{n^{1-\gamma}\ln n^{1+\delta}} a.s. \quad  \text{and} \quad \lambda_{\max}(A_{n,\tau}) =  o\pa{n^{1-\gamma}\ln n^{1+\delta}} a.s.$$
\end{prop}
\begin{proof}[Proof of Proposition \ref{grandvp}]
	Define $W_n:=A_{n-1}T_n^T + T_nA_{n-1}$. By definition of $A_n$,  {and since the projection is $1$-Lipschitz,}
	\begin{align*}
		\norm{A_{n+1}}_F^2 & {\leq} \norm{A_n}_F^2 -2\gamma_{n+1}\braket{A_n,W_{n+1} - 2I_d}_F\mathbf{1}_{\norm{Q_{n+1}}\norm{Z_{n+1}}\le \beta_{n+1}} \\
		&+\gamma_{n+1}^2\norm{W_{n+1} - 2I_d}^2_F\mathbf{1}_{\norm{Q_{n+1}}\norm{Z_{n+1}}\le \beta_{n+1}}.
	\end{align*}
	Considering $\mathcal{F}_n = \sigma \acco{X_1,...,X_n,Z_1,...,Z_n}$, remark that  
	\[
\mathbb{E} \left[ \braket{A_n,W_{n+1} - 2I_d}_F |\mathcal{F}_{n} \right]  = 	\braket{A_n,\nabla^2G(\hat{\theta}_n)A_n + A_n\nabla^2G(\hat{\theta}_n) - 2I_d}_F  .
	\]
Since $\mathbf{1}_{\norm{Q_{n+1}}\norm{Z_{n+1}}\le \beta_{n+1}} = 1- \mathbf{1}_{\norm{Q_{n+1}}\norm{Z_{n+1}}> \beta_{n+1}}$	it comes 
	\begin{align*}
		\mathbb{E}\cro{\norm{A_{n+1}}_F^2|\mathcal{F}_n} &  {\leq}  \norm{A_n}_F^2 - 2	\gamma_{n+1}\braket{A_n,\nabla^2G(\hat{\theta}_n)A_n + A_n\nabla^2G(\hat{\theta}_n) - 2I_d}_F \\ 
		&+ \gamma_{n+1}^2\mathbb{E}\cro{\norm{W_{n+1} -2 I_d}^2\mathbf{1}_{\norm{Q_{n+1}}\norm{Z_{n+1}}\le \beta_{n+1}}|\mathcal{F}_n} \\
		&+2\gamma_{n+1}\mathbb{E}\cro{\braket{A_n,W_{n+1} - 2I_d}_F\mathbf{1}_{\norm{Q_{n+1}}\norm{Z_{n+1}} > \beta_{n+1}}|\mathcal{F}_n} .
	\end{align*}
Since $\|A + B\|_{F}^{2} \leq 2\|A\|_{F}^{2} + 2 \| B\|_{F}^{2}$ (used twice) and $\|AB\|_{F} \leq \| A\|_{F} \| B\|_{F}$, and since $\| Z_{n+1}  \|\leq M$,	Assumption \textbf{(A4)} ensures that 
	\begin{align*}
		\mathbb{E} \cro{\norm{W_{n+1}-2I_d}^2_F\mathbf{1}_{\norm{Q_{n+1}}\norm{Z_{n+1}}\le \beta_{n+1}}|\mathcal{F}_n} & \leq 2 \mathbb{E} \left[ \left\| W_{n+1} \right\|_{F}^{2} |\mathcal{F}_{n} \right] + 2 \left\| 2  I_{d} \right\|_{F}^{2}\\
		&\le 8\pa{\mathbb{E}\cro{\norm{\nabla^2_hg(X_{n+1},\hat{\theta}_n)}_{F}^2  {\norm{Z_{n+1}}^4}\norm{A_n}^2_F |\mathcal{F}_{n}} + d}\\
		&\le 8\pa{C_q^{2/q}M^{4}\norm{A_n}^2_F +  d}.
	\end{align*}
	Since $\norm{Q_{n+1}}$ = $\norm{Q_{n+1}}^q\norm{Q_{n+1}}^{1-q}$ with $q>1$, using Cauchy-Schwarz inequality one has $\langle A_{n} , W_{n+1} \rangle_{F} \leq \| A_{n}\|_{F}\|W_{n+1}\|_{F}$,  and since for any $\zeta_{n+1} > 0$, $\langle A_{n},I_{d}\rangle_{F} \leq \frac{1}{2\zeta_{n+1}} \|A_{n}\|_{F}^{2} + \frac{1}{2}\zeta_{n+1}\| I_{d}\|_{F}^{2}$,
	\begin{align*}
		\gamma_{n+1}&\mathbb{E}\cro{\braket{A_n,W_{n+1} - 2I_d}_F\mathbf{1}_{\norm{Q_{n+1}}\norm{Z_{n+1}} > \beta_{n+1}}|\mathcal{F}_n} \\
		&\le 2\gamma_{n+1}\norm{A_n}^2_F\mathbb{E}\cro{\norm{Q_{n+1}}\norm{Z_{n+1}}\mathbf{1}_{\norm{Q_{n+1}}\norm{Z_{n+1}}\ge \beta_{n+1}}|\mathcal{F}_n}\\
		&+ \frac{\gamma_{n+1}}{\zeta_{n+1}}\norm{A_n}^2_F + \zeta_{n+1}\gamma_{n+1}d \\
		&\le \pa{2C_q M^{2q}\gamma_{n+1}\beta_{n+1}^{1-q}+\frac{\gamma_{n+1}}{\zeta_{n+1}}}\norm{A_n}^2_F+ \gamma_{n+1}\zeta_{n+1}d.
	\end{align*}
    Since $\nabla^2G(\hat{\theta}_n)$ and $A_n$ are positive, one has for any $\zeta_{n+1} > 0$,
	\begin{align*}
		- \gamma_{n+1}\braket{A_n,\nabla^2G(\hat{\theta}_n)A_n + A_n\nabla^2G(\hat{\theta}_n) - 2I_d}_F
		&\le \gamma_{n+1}\braket{A_n,2I_d}_{F} 
		\\&\leq  \frac{\gamma_{n+1}}{\zeta_{n+1}}\norm{A_n}^2_F+ \gamma_{n+1}\zeta_{n+1}d.
	\end{align*}
	Finally,
	\begin{align*}
		\mathbb{E}\cro{\norm{A_{n+1}}_F^2|\mathcal{F}_n} &\le\pa{1+\frac{4\gamma_{n+1}}{\zeta_{n+1}}+4C_q\gamma_{n+1}\beta_{n+1}^{1-q}+8\gamma_{n+1}^2C_q^{2/q}}\norm{A_n}_F^2\\
		&+4\gamma_{n+1}\zeta_{n+1}d+8\gamma_{n+1}^2d.
	\end{align*}
	Setting $\zeta_n=n^{1-\gamma}\ln n^{1+\delta}$ with $\delta >0$, we can apply Lemma \ref{corRS} with $V_n = \norm{A_n}_F^2$, and $a_n = \zeta_{n}^2$. One then obtains $\norm{A_n}_F^2 = o\pa{\zeta_{n}^2} = o\pa{n^{2-2\gamma}\ln n^{2+2\delta}}$, so that $\lambda_{\max}(A_n) =  o\pa{n^{1-\gamma}\ln n^{1+\delta}} a.s.$
\end{proof}
\subsection{Proof of Theorem \ref{consisAn}}
\subsubsection{Rate of convergence of $A_{n}$}
The aim is to provide an initial rate of convergence for $A_n$. A more refined or faster rate will be established later. First, {recalling that $\tilde{A}_{n+1} = A_n  -\gamma_{n+1}(W_{n+1}-2I_d)\mathbf{1}_{\norm{Q_{n+1}}\norm{Z_{n+1}}\le \beta_{n+1}}$, and  denoting $\Pi_{n+1}^{\perp} = A_{n+1} - \tilde{A}_{n+1}$,} observe that
\begin{align*}
	A_{n+1}-H^{-1} &= A_n- H^{-1} -\gamma_{n+1}(W_{n+1}-2I_d)\mathbf{1}_{\norm{Q_{n+1}}\norm{Z_{n+1}}\le \beta_{n+1}}  {+\Pi_{n+1}^{\perp}}\\
	&=A_n-H^{-1}-\gamma_{n+1}(W_{n+1}-2I_d) + \gamma_{n+1}(W_{n+1}-2I_d)\mathbf{1}_{\norm{Q_{n+1}}\norm{Z_{n+1}}>\beta_{n+1}}  {+\Pi_{n+1}^{\perp}}\\
	&=A_n -H^{-1} -\gamma_{n+1}\pa{\mathbb{E}\cro{W_{n+1}|\mathcal{F}_n} -2I_d} + \gamma_{n+1}\xi_{n+1}  {+\Pi_{n+1}^{\perp}}\\
	&\quad+\gamma_{n+1}(W_{n+1}-2I_d)\mathbf{1}_{\norm{Q_{n+1}}\norm{Z_{n+1}}>\beta_{n+1}} \\
	&= A_n -H^{-1} -\gamma_{n+1}\pa{\nabla^2G(\hat{\theta}_n)A_n+A_n\nabla^2G(\hat{\theta}_n) -2I_d} + \gamma_{n+1}\xi_{n+1}  {+\Pi_{n+1}^{\perp}}   \\
	&\quad + \gamma_{n+1}(W_{n+1}-2I_d)\mathbf{1}_{\norm{Q_{n+1}}\norm{Z_{n+1}}>\beta_{n+1}} 
\end{align*}
where $\xi_{n+1} := -W_{n+1}+\nabla^2G(\hat{\theta}_n)A_n+A_{n}\nabla^2G(\hat{\theta}_n)$.
Let $\alpha^*_k$ be the function defined for all $h\in\mathcal{M}_d(\mathbb{R})$ by 
$$\alpha^*_k(h) = \left( I_{d} -  \gamma_{k+1}H \right) h \left( I_{d} - \gamma_{k+1} H \right),$$
we then have
\begin{align}
	  A_{n+1}-H^{-1}  =\alpha_{n}^{*}\pa{A_n-H^{-1}}{+\Pi_{n+1}^{\perp}}  
	&+\gamma_{n+1}r_{1,n}+\gamma_{n+1}r_{2,n}+\gamma_{n+1}s_{n}+\gamma_{n+1}s'_{n}+\gamma_{n+1}\xi_{n+1} \label{decomp}
\end{align} 
where \begin{align*}
	r_{1,n}&=(\nabla^2G(\hat{\theta}_n)-H)H^{-1}+H^{-1}(\nabla^2G(\hat{\theta}_n)-H),\\
	r_{2,n} &= (W_{n+1}-2I_d)\mathbf{1}_{\norm{Q_{n+1}}\norm{Z_{n+1}}>\beta_{n+1}}-\mathbb{E}\cro{(W_{n+1}-2I_d)\mathbf{1}_{\norm{Q_{n+1}}\norm{Z_{n+1}}>\beta_{n+1}}|\mathcal{F}_n},\\
	s_n &= \mathbb{E}\cro{\pa{T_{n+1}\pa{A_n-H^{-1}}+\pa{A_n-H^{-1}}T_{n+1}}\mathbf{1}_{\norm{Q_{n+1}}\norm{Z_{n+1}}>\beta_{n+1}}|\mathcal{F}_n} \\
	&+\pa{\nabla^2G(\hat{\theta}_n)-H}\pa{A_n-H^{-1}}+\pa{A_n-H^{-1}}\pa{\nabla^2G(\hat{\theta}_n)-H} -\gamma_{n+1}^2H(A_n-H^{-1})H ,\\
	s'_n &=\mathbb{E}\cro{\pa{T_{n+1}H^{-1}+H^{-1}T_{n+1}-2I_d}\mathbf{1}_{\norm{Q_{n+1}}\norm{Z_{n+1}}>\beta_{n+1}}|\mathcal{F}_n} 
\end{align*}
By induction,
\begin{align}\label{decdegueulasse}
\notag	A_{n}-H^{-1}&=\Psi^*_{n,0}\left(A_0-H^{-1} \right)+\sum_{k=0}^{n-1}\Psi^*_{n,k+1}\gamma_{k+1}\xi_{k+1}+\sum_{k=0}^{n-1}\Psi^*_{n,k+1}\gamma_{k+1}r_{1,k}  \\
\notag 	&\quad+\sum_{k=0}^{n-1}\Psi^*_{n,k+1}\gamma_{k+1}r_{2,k}+ \sum_{k=0}^{n-1}\Psi^*_{n,k+1}\gamma_{k+1}s_{k} + \sum_{k=0}^{n-1}\Psi^*_{n,k+1}\gamma_{k+1}s'_{k} \\
& \quad {+ \sum_{k=0}^{n-1} \Psi_{n,k+1}^{*} \gamma_{k+1}\Pi_{k+1}^{\perp}}
\end{align}
where $\Psi^*_{n,k}$ is the function defined for all $h\in\mathcal{M}_p(\mathbb{R})$ by :
$$ \Psi^*_{n,k}(h) = \pa{\prod_{j=k+1}^n\alpha^*_j}(h) = \pa{\prod_{j=k+1}^n\left( I_{d} -  \gamma_{j}H \right)}h\pa{\prod_{j=k+1}^n\left( I_{d} -  \gamma_{j}H \right)}.$$
Next, we will determine the rate of convergence for each term on the right-hand side of equation \eqref{decdegueulasse}.

\medskip

\noindent\textbf{Rate of convergence of $M_{n}:=\sum_{k=0}^{n-1}\Psi^*_{n,k+1}\gamma_{k+1}\xi_{k+1}$: }
Recall that for all $\delta>0$, $\lambda_{\max}(A_n) =  o\pa{n^{1-\gamma}\ln n^{1+\delta}} a.s.$ 
Thus, there exists a constant $c'>0$ such that $\mathbb{E}\cro{\norm{\xi_{n+1}}^2_F | \mathcal{F}_n} \le c'\pa{1+\norm{A_n}_F^2}= o\pa{n^{2-2\gamma}\ln n^{2+2\delta}}$ a.s., and according to Lemma \ref{lemmartbeta}, it follows that
\begin{equation}
\label{maj::Mn}
\norm{\sum_{k=0}^{n-1}\Psi^*_{n,k+1}\gamma_{k+1}\xi_{k+1}}^2_F = \mathcal{O}\pa{n^{2-3\gamma}\ln n^{2+2\delta}} \quad a.s.
\end{equation}

\noindent\textbf{Rate of convergence of  $R_{1,n}:=\sum_{k=0}^{n-1}\Psi^*_{n,k+1}\gamma_{k+1}r_{1,k}$: } we have
$$R_{1,n+1} = \pa{I_d-\gamma_{n+1}H}R_{1,n}\pa{I_d-\gamma_{n+1}H}+\gamma_{n+1}r_{1,n+1}.$$
Therefore, for $n$ large enough
\begin{align*}
	\norm{R_{1,n+1}}_{op} &\le \norm{I_d-\gamma_{n+1}H}^2_{op}\norm{R_{1,n}}_{op}+\gamma_{n+1}\norm{r_{1,n+1}}_{op}\\
	&\le \pa{1-\lambda_{\min}(H)\gamma_{n+1}}^2\norm{R_{1,n}}_{op} + \gamma_{n+1}\norm{r_{1,n+1}}_{op} 
\end{align*}
By Assumption \textbf{(A3)}, 
\begin{equation}\label{majr1npetit}
\norm{r_{1,n+1}}_{op} =  o\pa{n^{-a/2}\ln n^{(1+\delta)/2}} \quad  a.s.
\end{equation}
According to Lemma \ref{lemcours}, 
\begin{equation}
\label{maj::Rn}
\norm{R_{1,n+1}}_{op} =  o\pa{n^{-a/2}\ln n^{(1+\delta)/2}} \quad  a.s.
\end{equation}

\medskip

\noindent\textbf{Rate of convergence of $R_{2,n} := \sum_{k=0}^{n-1}\Psi^*_{n,k+1}\gamma_{k+1}r_{2,k}$: } Observe that $\left( r_{2,n} \right)$ is a sequence of martingale differences adapted to the filtration $\mathcal{F}_{n}$. Moreover, since $\|A + B\|_{F}^{2} \leq 2\|A \|_{F}^{2} + 2 \| B \|_{F}^{2}$ (used twice),
\begin{align*}
	\mathbb{E}&\cro{\norm{(W_{n+1}-2I_d)\mathbf{1}_{\norm{Q_{n+1}}\norm{Z_{n+1}}\ge\beta_{n+1}}}_{op}^2|\mathcal{F}_n} \\
	\le &8\mathbb{E}\cro{\pa{\norm{Q_{n+1}}^2\norm{Z_{n+1}}^2\norm{A_n}_{F}^2+ d}\mathbf{1}_{\norm{Q_{n+1}}\norm{Z_{n+1}}\ge\beta_{n+1}}|\mathcal{F}_n}.
\end{align*}
Applying Markov's inequality, and since $\left\| Z_{n+1} \right\| \leq M$, 
\[
\mathbb{E} \left[ \mathbf{1}_{\norm{Q_{n+1}}\norm{Z_{n+1}}\ge\beta_{n+1}}|\mathcal{F}_n \right] \leq \beta_{n+1}^{-q} \mathbb{E} \left[ \left\| \nabla_{h}^{2}g\left( X_{n+1} , \hat{\theta}_{n} \right) \right\|_{op}^{q} \left\| Z_{n+1} \right\|^{2q} |\mathcal{F}_{n} \right] \leq C_{q}M^{2q}\beta_{n+1}^{-q} .
\]
In a same way, 
\begin{align*}
\mathbb{E} \left[ \norm{Q_{n+1}}^2\norm{Z_{n+1}}^2\norm{A_n}^2 \mathbf{1}_{\norm{Q_{n+1}}\norm{Z_{n+1}}\ge\beta_{n+1}}|\mathcal{F}_n \right] &  \leq \beta_{n+1}^{2-q}\norm{A_n}_{F}^2\mathbb{E} \left[ \norm{Q_{n+1}}^q\norm{Z_{n+1}}^q |\mathcal{F}_{n} \right]  \\
& \leq C_{q}\left\| A_{n} \right\|_{F}^{2}M^{2q} \beta_{n+1}^{2-q} .
\end{align*}
Then,
\begin{align*}
	\mathbb{E}\cro{\norm{(W_{n+1}-2I_d)\mathbf{1}_{\norm{Q_{n+1}}\norm{Z_{n+1}}\ge\beta_{n+1}}}^2|\mathcal{F}_n} &\le 8  C_{q}M^{2q}\beta_{n}^{-q}    + 8C_{q}\left\| A_{n} \right\|_{F}^{2}M^{2q}
\end{align*}
Thus, with  the help of Lemma \ref{lemmartbeta},
\begin{equation}\label{vit::r2n}
\norm{R_{2,n}}^2_{op}  = \mathcal{O}\pa{n^{2-3\gamma}\ln n^{2+2\delta}n^{2\beta-q\beta}} \quad   a.s.
\end{equation}
Note that the rate of convergence for $R_{2,n}$ is faster than that of $M_{n}$.\\
 
\medskip

\noindent\textbf{Rate of convergence of $S
_{n}':= \sum_{k=0}^{n-1}\Psi^*_{n,k+1}\gamma_{k+1}s'_{k} $: }
We have similarly
\begin{align}
\notag	 \left\| s'_n \right\| & \leq  \mathbb{E}\cro{ \left\| \pa{T_{n+1}H^{-1}+H^{-1}T_{n+1}-2I_d} \right\|_{F}\mathbf{1}_{\norm{Q_{n+1}}\norm{Z_{n+1}}\ge\beta_{n+1}}|\mathcal{F}_n} \\
\notag	 & \leq  2 \left\| H^{-1} \right\|_{F} \mathbb{E}\left[ \left\| Z_{n+1} \right\|^{2} \left\| \nabla_{h}^{2}g \left( X_{n+1} , \hat{\theta}_{n} \right) \right\|_{F} \mathbf{1}_{\norm{Q_{n+1}}\norm{Z_{n+1}}\ge\beta_{n+1}}|\mathcal{F}_n \right]  \\
\notag	 \quad &+ 2q \mathbb{E}\left[ \mathbf{1}_{\norm{Q_{n+1}}\norm{Z_{n+1}}\ge\beta_{n+1}}|\mathcal{F}_n \right]   \\
\label{majsnprimepetit}	 \quad &   = \mathcal{O}\pa{\beta_{n+1}^{1-q}}.
\end{align}
Therefore, thanks to Lemma \ref{lemcours}, 
\begin{equation}
\label{maj::Snprime} S'_n    = \mathcal{O}\pa{\beta_{n}^{1-q}}.
\end{equation}
\medskip

 {\noindent\textbf{Rate of convergence of $R_{\Pi,n} = \sum_{k=0}^{n-1}\Psi_{n,k+1}^{*}\gamma_{k+1}\Pi_{k+1}^{\perp}$: } Observe that $\Pi_{n+1}^{\perp}$ is different from $0$ if 
\[
\left\| A_{n} - \gamma_{n+1}(W_{n+1}-2I_d)\mathbf{1}_{\norm{Q_{n+1}}\norm{Z_{n+1}}\le \beta_{n+1}} \right\|_{F} \leq \beta_{n+1}'.
\]
Nevertheless, one has
\begin{align*}
\left\| A_{n} -  \gamma_{n+1}(W_{n+1}-2I_d)\mathbf{1}_{\norm{Q_{n+1}}\norm{Z_{n+1}}\le \beta_{n+1}} \right\|_{F} & \leq \left\| A_{n} \right\|_{F} \left( 1+ 2\gamma_{n+1}\beta_{n+1} \right) + 2 d \gamma_{n+1} \\
&  = o \left( \ln (n+1)^{1+\delta}n^{1-\gamma} \right) \quad a.s.
\end{align*}
and since $\beta ' > 1- \gamma$, it comes that $\mathbf{1}_{\Pi_{n+1}^{\perp} \neq 0}$ converges almost surely to $0$. Then
\begin{align*}
\|R_{\Pi,n} \|_{op} \leq \left\| \Psi_{n,0}^{*} \right\|_{op}  \sum_{k=0}^{n-1}\left\| \Psi_{k,0}^{-1} \gamma_{k+1}\Pi_{k+1}^{\perp}\right\|_{op}\mathbf{1}_{\Pi_{k+1}^{\perp} \neq 0}  = O \left(  \left\| \Psi_{n,0}^{*} \right\|_{op}  \right) \quad a.s.
\end{align*}
and this term converges exponentially fast to $0$.
}

\noindent\textbf{A first result for $A_{n}$: }
Thanks to equalities \eqref{maj::Mn}, \eqref{maj::Rn}, \eqref{vit::r2n} and \eqref{maj::Snprime}, one can rewrite $A_{n} - H^{-1}$ as 
$$
A_n-H^{-1} = \sum_{k=0}^{n-1}\Psi^*_{n,k+1}\gamma_{k+1}s_k + \widetilde{S}_n + \Psi_{n,0}^{*} \left( A_{0} - H^{-1} \right) 
$$
with $\left\| \widetilde{S}_n \right\|_{F} =o\pa{n^{\max\pa{1-\frac{3}{2}\gamma,-a/2,\beta(1-q)}}\ln n^{1+\delta}}$ a.s. and $\Psi_{n,0}^{*} \left( A_{0} - H^{-1} \right)  $ converges exponentially fast to $0$.  In addition, since $\hat{\theta}_{n}$ converges almost surely to $\theta$ and since \\
$\mathbb{E}\cro{\norm{Q_{n+1}}\norm{Z_{n+1}}\mathbf{1}_{\norm{Q_{n+1}}\norm{Z_{n+1}}\ge\beta_{n}}\|\mathcal{F}_n}$ converges almost surely to $0$ (one can perform analogous calculations to those used in deriving equality \eqref{maj::Rn}), 
\begin{align*}
	\norm{s_n}_{op}& \le \norm{A_n-H^{-1}}_{op}\mathbb{E}\cro{\norm{Q_{n+1}}\norm{Z_{n+1}}\mathbf{1}_{\norm{Q_{n+1}}\norm{Z_{n+1}}\ge\beta_{n}}\|\mathcal{F}_n}\\
	&\quad+\norm{\nabla^2G(\hat{\theta}_n)-H}\norm{A_n-H^{-1}}_{op} + \gamma_{n+1}^{2} \| H\|_{op}^{2} \left\| A_{n} - H^{-1} \right\|_{op}\\
	& =o \left(\norm{A_n-H^{-1}}_{op} \right)  \quad a.s.
\end{align*}
Define $S_{n} := \sum_{k=0}^{n-1}\Psi^*_{n,k+1}\gamma_{k+1}s_k$. There exists a positive sequence $(\widetilde{r}_n)_{n\ge0}$ with $\widetilde{r}_n \xrightarrow[n\rightarrow\infty]{a.s.}0$ such that
\begin{align*}
	\norm{S_{n+1}}_{op}&\le(1-\lambda_{\min}(H)\gamma_{n+1})^2\norm{S_n}_{op} + \gamma_{n+1}\widetilde{r}_n\norm{A_n-H^{-1}}_{op}\\
	& \leq (1-\lambda_{\min}(H)\gamma_{n+1})^2\norm{S_n}_{op} + \gamma_{n+1}\widetilde{r}_n\pa{\norm{S_n}_{op}+\norm{\widetilde{S}_n}_{op}}.
\end{align*}
Applying Lemma \ref{lemcours}, 
\[
\norm{S_n}_{op} = o\pa{n^{\max\pa{1-\frac{3}{2}\gamma,-a/2,\beta(1-q)}}\ln n^{1+\delta}} ,
\]
 which implies that 
\begin{equation}\label{firstresult}
\norm{A_n-H^{-1}}_{F} = o\pa{n^{\max\pa{1-\frac{3}{2}\gamma,-a/2,\beta(1-q)}}\ln n^{1+\delta}} \quad a.s.
\end{equation}

\medskip

\noindent \textbf{Final rate of convergence of $A_{n}$: }
The aim here is, with the help of this first result on $A_{n}$ given by \eqref{firstresult}, to obtain better rates of convergence for $M_{n}$ and $R_{2,n}$. First, note  that if $\gamma > 2/3$, then $1-\frac{3}{2}\gamma <0$ and we have directly $A_n \xrightarrow[n\rightarrow\infty]{a.s.} H^{-1}$.
If $\gamma \le 2/3$, we have $1-\frac{3}{2}\gamma >0$ , so that $\max\{1-\frac{3}{2}\gamma,-a/2,\beta(1-q)\}= \max\{1-\frac{3}{2}\gamma,\beta(1-q)\}$ and $\norm{A_n} = o\pa{n^{\max\{1-\frac{3}{2}\gamma,\beta(1-q)\}}\ln n^{1+\delta}} a.s.$. 
Thus, when $\gamma \leq 2/3$, applying Lemma \ref{lemmartbeta} and following the same calculus as for obtaining \eqref{maj::Mn} and \eqref{vit::r2n},  one now has 
\[
\norm{M_{n}}_{F}^2 = o\pa{n^{\max\{2-4\gamma,\beta(2-2q)-\gamma\}}\ln n^{2+2\delta}} \quad a.s.
\]   
and
\[
\norm{R_{2,n}}^2_{op} = \mathcal{O}\pa{n^{2-4\gamma}\ln n^{2+2\delta}n^{2\beta-q\beta}} \quad a.s.
\]
Following the same process as before, we now have that for any $\gamma \in (1/2,1)$, $A_{n}$ converges almost surely to $H^{-1}$, and in particular, one has $\norm{A_n}$ = $\mathcal{O}\pa{1}$ a.s.  Applying another time Lemma \ref{lemmartbeta} and following the same calculus as for obtaining \eqref{maj::Mn} and \eqref{vit::r2n}, we have 
\[\norm{M_{n}}_{F}^2 = \mathcal{O}\pa{n^{-\gamma}\ln n^{1+\delta}} \quad a.s.
\]
and
\[
\norm{R_{2,n}}^2_{op} = \mathcal{O}\pa{n^{-\gamma}\ln n^{2+2\delta}n^{2\beta-q\beta}} \quad a.s.\]
Finally, we have 
\begin{equation}\label{vitesseAn}
	\norm{A_n - H^{-1}}_{F}^2 = o\pa{\frac{\ln n^{1+\delta}}{n^{\min\{\gamma,a, 2 \beta (q-1)\}}}} a.s.
\end{equation}

\medskip

\subsubsection{Rate of convergence of $A_{n,\tau}$.} First, note that decomposition \eqref{decomp} can be written as 
\begin{align*}
	H(A_n-H^{-1})+(A_n-H^{-1})H = \frac{A_n-H^{-1}-(A_{n+1}-H^{-1})}{\gamma_{n+1}}+r_{1,n}+r_{2,n}+s_{n}+s'_{n}+\xi_{n+1}  {+ \frac{\Pi_{k+1}^{\perp}}{\gamma_{k+1}}},
\end{align*}
so that
\begin{align}\label{decmoy}
\notag	H(A_{n,\tau}-H^{-1})+(A_{n,\tau}-H^{-1})H = &t_n\sum_{k=0}^nu_{k+1}\pa{(A_k-H^{-1})-(A_{k+1}-H^{-1})} \\
	&+t_n\sum_{k=0}^nu_{k+1}\gamma_{k+1}\left(r_{1,k}+r_{2,k}+s_{k}+s'_{n}+\xi_{k+1}   {+ \frac{\Pi_{k+1}^{\perp}}{\gamma_{k+1}}}\right)
\end{align}
where $t_n = \frac{1}{\sum^n_{k=0}\log(k+1)^\tau}$ and $u_{k+1} = \frac{\ln(k+1)^\tau}{\gamma_{k+1}}$. Our objective is to determine the rate of convergence for each term on the right-hand side of  decomposition   \eqref{decmoy}.

\medskip

\noindent \textbf{Rate of convergence of $t_n\sum_{k=0}^nu_{k+1}\pa{(A_k-H^{-1})-(A_{k+1}-H^{-1})}$: }
with the help of an Abel's transform,
\begin{align*}
t_n\sum_{k=0}^nu_{k+1}\pa{(A_k-H^{-1})-(A_{k+1}-H^{-1})}&  =   -t_n\left(A_{n+1}-H^{-1}\right) u_{n+1}+t_n\left( A_0-H^{-1}\right)u_1 \\
&  +t_n\sum_{k=1}^n(A_k-H^{-1})(u_{k+1}-u_k) .
\end{align*}
It is obvious that $t_{n} \left( A_{0} - H^{-1}\right)u_{1} $ is negligible while thanks to equality \eqref{vitesseAn},
\[
t_{n}\left\|  A_{n+1} - H^{-1}\right\| u_{n+1} = o \left( \frac{\ln n^{(1+\delta)/2}}{n^{1 - \gamma + \frac{1}{2}\min \left\lbrace \gamma ,a ,2\beta (q-1) \right\rbrace }} \right) \quad a.s 
\]
In addition, by the mean value theorem, $|u_{k+1}-u_{k}| \le 2k^{\gamma-1}\max\{1,\tau\}\ln(k+1)^\tau$,
which implies that
\[
\left\| \sum_{k=1}^n(A_k-H^{-1})(u_{k+1}-u_k) \right\|_{F} \leq \sum_{k=1}^n \left\| A_k-H^{-1}\right\|_{F}2k^{\gamma-1}\max\{1,\tau\}\ln(k+1)^\tau
\]
and with the help of equation \eqref{vitesseAn}
\begin{equation}\label{abel}
\norm{t_n\sum_{k=0}^nu_k\pa{(A_k-H^{-1})-(A_{k+1}-H^{-1})}}_F = o \left( \frac{\ln n^{1+\delta}}{n^{1 - \gamma + \frac{1}{2}\min \left\lbrace \gamma ,a ,2\beta (q-1) \right\rbrace }} \right) \quad a.s 
\end{equation}

\medskip

\noindent\textbf{Rate of convergence of $t_n\sum_{k=0}^nu_{k+1}\gamma_{k+1}\xi_{k+1}$. }
Remark that there exists $c'>0$ such that  
$$\mathbb{E}\cro{\norm{\xi_{n+1}}^2_F \| \mathcal{F}_n} \le c'\pa{1+\norm{A_n}_F^2} \le c'\pa{1+2\norm{A_n-H^{-1}}_F^2+2\norm{H^{-1}}_F^2} = \mathcal{O}(1) \quad a.s.$$
By applying a law of large numbers for martingales, one obtains for all $\delta >0$
$$\norm{t_n\sum_{k=0}^nu_{k+1}\gamma_{k+1}\xi_{k+1}}_F^2 = o\pa{\frac{\ln n^{1+\delta}}{n}} a.s.$$

\noindent \textbf{Rate of convergence of $t_n\sum_{k=0}^nu_{k+1}\gamma_{k+1}r_{1,k}$.}
First, let us give a useful equality:
consider $\sum_{k=0}^nk^{-\tilde{a}/2}\ln k^{(1+\delta)/2}$ with $\tilde{a}>0$, since 
$\int_{1}^{n} x^{-\tilde{a}/2}\ln x^{(1+\delta)/2} dx \le c(\ln n)^{(1+\delta)/2}x^{1-\tilde{a}}$ if $\tilde{a}\neq2$ and $\int_{1}^{n} x^{-\tilde{a}/2}\ln x^{(1+\delta)/2} dx \le c'(\ln n)^{(3+\delta)/2}$ if $\tilde{a}=2$
then one can verify that 
\begin{equation}\label{equationatilde}
\norm{t_n\sum_{k=0}^nu_k\gamma_{k+1}k^{-\tilde{a}/2}\ln k^{(1+\delta)/2}}^2_F = o\pa{\frac{\ln n^{1+\delta + 2\mathbf{1}_{\tilde{a} = 2}}}{n^{\min\{2,\tilde{a}\}}}} a.s.
\end{equation}
Then, recalling that $\norm{r_{1,k+1}}_{op} =  o\pa{n^{-a/2}\ln k^{(1+\delta)/2}}$ a.s. (see equality \eqref{majr1npetit}),
\[
\left\| t_{n}\sum_{k=0}^{n} u_{k}\gamma_{k+1}r_{1,k} \right\|_{F}^{2}=o\pa{\frac{\ln n^{1+\delta + 2\mathbf{1}_{ {a} = 2}}}{n^{\min\{2, {a}\}}}} a.s.
\]

\medskip

\noindent \textbf{Rate of convergence of $t_n\sum_{k=0}^nu_{k+1}\gamma_{k+1}s_{k}'$.} Recalling that $s_{k}' = O \left( \beta_{k}^{1-q} \right)$ (see equality \eqref{majsnprimepetit}), it comes
\[
\left\| t_{n}\sum_{k=0}^{n} u_{k}\gamma_{k+1}s_{k}' \right\|_{F}^{2} = o\pa{\frac{\ln n^{ 2\mathbf{1}_{2\beta(q-1) = 2}}}{n^{\min\{2,2\beta (q-1)\}}}} \quad a.s.
\]

\noindent \textbf{Rate of convergence of $t_n\sum_{k=0}^nu_{k+1}\gamma_{k+1}s_{k}$.} Recall that
\begin{align*}
	\norm{s_n}_{op}& \le \norm{A_n-H^{-1}}_{op}\mathbb{E}\cro{\norm{Q_{n+1}}\norm{Z_{n+1}}\mathbf{1}_{\norm{Q_{n+1}}\norm{Z_{n+1}}\ge\beta_{n}} |\mathcal{F}_n}\\
	&\quad +\norm{\nabla^2G(\hat{\theta}_n)-H}\norm{A_n-H^{-1}}_{op} + \gamma_{n+1} \| H\|_{op} \left\| A_{n} - H^{-1} \right\|_{op} ,
\end{align*}
Recalling that $\mathbb{E}\left[ \left\| T_{n+1} \right\|_{op} \mathbf{1}_{\left\| T_{n+1} \right\| \geq \beta_{n} }  |\mathcal{F}_{n} \right] = O \left( \beta_{n}^{1-q} \right)$, and with the help of equality \eqref{vitesseAn}
\begin{align*}
&\norm{A_n-H^{-1}}_{op}\mathbb{E}\cro{\norm{Q_{n+1}}\norm{Z_{n+1}}\mathbf{1}_{\norm{Q_{n+1}}\norm{Z_{n+1}}\ge\beta_{n}} |\mathcal{F}_n} \\
&\quad= o \left( \frac{\ln n^{(1+\delta)/2}}{n^{\beta (q-1) +\frac{1}{2} \min\left\lbrace a, \gamma , 2\beta (q-1)\right\rbrace}} \right) \quad a.s 
\end{align*}
In addition, one has
\[
\norm{\nabla^2G(\hat{\theta}_n)-H}\norm{A_n-H^{-1}}_{op} = o \left( \frac{\ln n^{1+\delta}}{n^{a/2 + \frac{1}{2} \min\left\lbrace a, \gamma , 2\beta (q-1)\right\rbrace}} \right) \quad a.s.
\]
and the term $\gamma_{n+1} \| H\|_{op} \left\| A_{n} - H^{-1} \right\|_{op}$ is negligible,
leading to
\[
\left\| s_{n} \right\| = o \left( \frac{\ln n^{1+\delta}}{n^{\min \left\lbrace a/2 , \beta (q-1) \right\rbrace + \frac{1}{2} \min\left\lbrace a, \gamma , 2\beta (q-1)\right\rbrace}} \right) \quad a.s.
\] 
Then, since $ \min \left\lbrace a/2 , \beta (q-1) \right\rbrace + \frac{1}{2} \min\left\lbrace a, \gamma , 2\beta (q-1)\right\rbrace \geq  \min\left\lbrace a, \gamma , 2\beta (q-1)\right\rbrace$,
\[
\left\| t_{n}\sum_{k=0}^{n} u_{k}\gamma_{k+1}s_{k} \right\|_{F}^{2} = o\pa{\frac{\ln n^{1+\delta}}{n^{ 2\min \left\lbrace a,\gamma , 2\beta (q-1) \right\rbrace }}} a.s.
\]

\medskip

\noindent \textbf{Rate of convergence of $t_n\sum_{k=0}^nu_{k+1}\gamma_{k+1}r_{2,k}$.} 
Let us recall that $r_{2,n}$ is a martingale difference, and with the help of a law of large numbers for martingales it comes that
\[
\left\| t_n\sum_{k=0}^nu_{k+1}\gamma_{k+1}r_{2,k} \right\|_{F}^{2} = o \left( \frac{\ln n^{1+\delta}}{n} \right) \quad a.s.
\]

\noindent {\textbf{Rate of convergence of $t_{n}\sum_{k=0}^{n}\frac{\Pi_{k+1}^{\perp}}{\gamma_{k+1}}$: } Since $\mathbf{1}_{\Pi_{k+1}^{\perp} \neq 0}$ converges almost surely to $0$, 
\[
\left\| t_{n}\sum_{k=0}^{n}\frac{\Pi_{k+1}^{\perp}}{\gamma_{k+1}} \right\| = O (t_{n}) \quad a.s.
\]
and this term is so negligible.
 }

\noindent\textbf{Rate of convergence of $A_{n,\tau}$.} We so have

$$\norm{H(A_{n,\tau}-H^{-1})+(A_{n,\tau}-H^{-1})H}^2 = o\pa{\frac{\ln n^{1+\delta}}{n^{\min\{1,a, 2\beta (q-1)\}}}} a.s.$$
Note that $H$ is diagonalisable in an orthonormal basis, and its eigenvalues are denoted by $\lambda_1,...,\lambda_d$. Thus, for any matrix $B \in \mathcal{M}_d$, we have
\begin{align*}
	\norm{HB+BH}_F^2 &= \sum_{i=1}^n\sum_{j=1}^n\pa{HB+BH}^2_{i,j} \\
	& =\sum_{i=1}^n\sum_{j=1}^n\pa{\lambda_i+\lambda_j}^2B^2_{i,j}\\
	& \ge 4\lambda_{\min}(H)^2\sum_{i=1}^n\sum_{j=1}^nB^2_{i,j} =  4\lambda_{\min}(H)^2\norm{B}^2_F.
\end{align*}
\begin{align*}
\norm{A_{n,\tau} - H^{-1}}^2 &\leq \frac{1}{4\lambda_{\min}(H)^{2}} \norm{H(A_{n,\tau}-H^{-1})+(A_{n,\tau}-H^{-1})H}^2 
\\ &= o\pa{\frac{\ln n^{1+\delta}}{n^{\min\{1,a, 2\beta (q-1)\}}}} a.s.
\end{align*}
\subsection{Proof of Theorem \ref{consisTheta}}
In order to prove this theorem, we show that the eigenvalues of $A_n$ and $A_{n,\tau}$ are well controlled, which implies the consistency of $\theta_{n}$. Then, we give an initial rate of convergence of $\theta_{n}$, which allows us to obtain the rate of convergence of $A_n$ and $A_{n,\tau}$. The last step is to deduce the rate of convergence of  $\theta_{n}$ and  $\theta_{n,\tau}$. 

\begin{prop}\label{petitvp}
	Under Assumptions\textbf{(A3)} and \textbf{(A5)}, the smallest eigenvalue of $A_n$ (resp. $A_{n,\tau}$) denoted by $\lambda_{\min}(A_n)$ (resp. $\lambda_{\min}(A_n,\tau)$) satisfies
	$$\lambda_{\min}(A_n)^{-1} =  \mathcal{O}\pa{n^\beta} a.s. \quad \quad \text{ and } \quad \quad  \lambda_{\min}(A_{n,\tau})^{-1} =  \mathcal{O}\pa{n^\beta} a.s.$$
\end{prop}
The proof is given in Section \ref{sec::proof::prop}. 
 According to Taylor's theorem, and using Assumption $\textbf{(A2)}$, we obtain
\begin{align*}
	G(\theta_{n+1}) &= G(\theta_n)+\nabla G(\theta_n)^T(\theta_{n+1}-\theta_{n})\\
	&\quad+\frac{1}{2}(\theta_{n+1}-\theta_{n})^T\int_{0}^1\nabla^2G(\theta_{n+1}+t(\theta_{n}-\theta_{n+1}))dt(\theta_{n+1}-\theta_{n}) \\
	&\le G(\theta_n)+\nabla G(\theta_n)^T(\theta_{n+1}-\theta_{n})+\frac{L_{\nabla G}}{2}\norm{\theta_{n+1}-\theta_{n}}^2.
\end{align*}
From the definition of $\theta_{n}$ \eqref{thetan},
\begin{align*}
	G(\theta_{n+1}) \le G(\theta_n)-\nu_{n+1}\nabla G(\theta_{n})^TA_{n,\tau}\nabla g(X_{n+1},\theta_{n})+\frac{L_{\nabla G}}{2}\nu_{n+1}^2\norm{A_{n,\tau}}^2_{op}\norm{\nabla g(X_{n+1},\theta_{n})}^2.
\end{align*}
Letting $K_n:=G(\theta_{n})-G(\theta)$, we can rewrite the inequality as
\begin{align*}
	K_{n+1} \le K_n -\nu_{n+1}\nabla G(\theta_{n})^TA_{n,\tau}\nabla^2g(X_{n+1},\theta_{n}) +\frac{L_{\nabla G}}{2}\nu_{n+1}^2\norm{A_{n,\tau}}^2_{op}\norm{\nabla g(X_{n+1},\theta_{n})}^2,
\end{align*}
which implies
\begin{align*}
	\mathbb{E}\cro{K_{n+1}|\mathcal{F}_n} &\le K_n -\nu_{n+1}\nabla G(\theta_{n})^TA_{n,\tau}\nabla G\left( \theta_{n}\right) \\
	&\quad+\frac{L_{\nabla G}}{2}\nu_{n+1}^2\norm{A_{n,\tau}}^2_{op}\mathbb{E}\cro{\norm{\nabla g(X_{n+1},\theta_{n})}^2|\mathcal{F}_n}.
\end{align*}
According to Assumption $\textbf{(A1)}$, we have
\begin{align*}
		\mathbb{E}\cro{K_{n+1}|\mathcal{F}_n} &\le \pa{1+\frac{CL_{\nabla G}}{2}\nu_{n+1}^2\norm{A_{n,\tau}}^2_{op}}K_n-\nu_{n+1}\lambda_{\min}(A_{n,\tau})\norm{\nabla G(\theta_{n})}^2\\
			&\quad+\frac{CL_{\nabla G}}{2}\nu_{n+1}^2\norm{A_{n,\tau}}^2_{op}.
\end{align*}
With the help of Proposition \ref{grandvp},
$$\sum_{n=0}^\infty\nu_{n+1}^2\norm{A_{n,\tau}}^2_{op} < +\infty \quad a.s.$$
Subsequently, according to Robbins-Siegmund Theorem, $K_n$ is guaranteed to converge almost surely to a finite random variable, and 
$$\sum_{n=0}^\infty\nu_{n+1}\lambda_{\min}(A_{n,\tau})\norm{\nabla G(\theta_{n})}^2 < +\infty \quad a.s.$$
 With the help of Proposition \ref{petitvp},
 $$\sum_{n=0}^\infty\nu_{n+1}\lambda_{\min}(A_{n,\tau}) = +\infty \quad a.s.$$
It suggests that $\lim \inf_n\norm{\nabla G(\theta_{n})}=0$ almost surely and $\lim \inf_nK_n=0$ almost surely. As $K_n$ converges almost surely to a random variable, $G(\theta_n)$ converges almost surely to $G(\theta)$. By local strong convexity,
	$$\theta_{n}\xrightarrow[n\rightarrow\infty]{a.s.}\theta \quad\text{ and }\quad \theta_{n,\tau'}\xrightarrow[n\rightarrow\infty]{a.s.}\theta. $$
\subsection{Proof of Theorem \ref{rateTheta}}
\subsubsection{Rate of convergence of $\theta_{n}$.} Recall that
\begin{align*}
	\mathbb{E}\cro{G(\theta_{n+1})-G(\theta)|\mathcal{F}_n} &\le \pa{1+\frac{L_{\nabla G}C}{2}\nu_{n+1}^2\lambda_{\max}(A_{n,\tau})^2}\pa{G(\theta_n)-G(\theta)} \\
	&-\nu_{n+1}\nabla G(\theta_n)^TA_{n,\tau}\nabla G(\theta_n) + \frac{L_{\nabla G}C}{2}\nu_{n+1}^2\lambda_{\max}(A_{n,\tau})^2.
\end{align*}
Since the functional $G$ is locally strongly convex meaning that there is a neighborhood of $\theta$ where smallest eigenvalues of the Hessian are uniformly bounded below by a positive constant. That is, the Polyak-Lojasiewicz inequality is locally satisfied:   there is $r_{0}> 0$ and $c_{0} > 0$ such that for all $h \in \mathcal{B}(\theta,  r_{0})$, 
\[
\| \nabla G (h ) \|^{2} 
 \geq  c_{0}(G (h) - G(\theta) ).
 \]
 Then, since $\mathbf{1}_{\| \theta_{t} - \theta \| \leq r_{0}} = 1- \mathbf{1}_{\| \theta_{t} - \theta \| > r_{0}}$,
 \begin{align*}
 \left\| \nabla F \left( \theta_{t} \right) \right\|^{2}  & \geq c_{0}(G \left( \theta_{t} \right) - G(\theta) )\mathbf{1}_{\| \theta_{t} - \theta \| \leq r_{0}} \\
 & = c_{0}(G \left( \theta_{t} \right) - G(\theta) )  - c_{0}(G \left( \theta_{t} \right) - G(\theta)))\mathbf{1}_{\| \theta_{t} - \theta \| > r_{0}}
 \end{align*}
  Recall that $2\gamma+2\nu-2>1$, so that there exists $\eta>0$ such that $\eta < 2\gamma+2\nu-3$. We define $\widetilde{V}_n := n^{\eta}\pa{G(\theta_n)-G(\theta)}$, so that
\begin{align*}
	\mathbb{E}\cro{\widetilde{V}_{n+1}|\mathcal{F}_n} &= \pa{\frac{n+1}{n}}^\eta\pa{\pa{1+\frac{L_{\nabla GC}}{2}\nu_{n+1}^2\lambda_{\max}(A_{n,\tau})^2} - c_0\nu_{n+1}\lambda_{\min}(A_{n,\tau})}\widetilde{V}_n\\
	&+ \frac{L_{\nabla GC}}{2}\nu_{n+1}^2\lambda_{\max}(A_{n,\tau})^2(n+1)^\eta + 2^{\eta}c_0\nu_{n+1}\lambda_{\min}(A_{n,\tau})\mathbf{1}_{\| \theta_{n,\tau} - \theta \| > r_{0}} \widetilde{V}_n.
       \end{align*}
Let $\widetilde{U}_n :=\pa{\frac{n+1}{n}}^\eta\pa{\pa{1+\frac{L_{\nabla GC}}{2}\nu_{n+1}^2\lambda_{\max}(A_{n,\tau})^2} - c_0\nu_{n+1}\lambda_{\min}(A_{n,\tau})}$, then
\begin{align*}
	\mathbb{E}\cro{\widetilde{V}_{n+1}|\mathcal{F}_n} &\le \widetilde{V}_n + \frac{L_{\nabla GC}}{2}\nu_{n+1}^2\lambda_{\max}(A_{n,\tau})^2(n+1)^\eta + \widetilde{V}_n (\mathbf{1}_{\widetilde{U}_n>1} + 2^{\eta}c_0\nu_{n+1}\lambda_{\min}(A_{n,\tau})\mathbf{1}_{\| \theta_{n,\tau} - \theta \| > r_{0}}).
\end{align*}
As $\nu+\beta <1$, one can easily verify that $\mathbf{1}_{\widetilde{U}_n>1} \xrightarrow{a.s.}0$ and $\mathbf{1}_{\| \theta_{n,\tau} - \theta \| > r_{0}}\xrightarrow{a.s.}0$. Then,
\[
\sum_{t\geq 1}\widetilde{V}_n (\mathbf{1}_{\widetilde{U}_n>1} + 2^{\eta}c_0\nu_{n+1}\lambda_{\min}(A_{n,\tau})\mathbf{1}_{\| \theta_{n,\tau} - \theta \| > r_{0}}) < + \infty \quad a.s.
\] 
In addition, since $2\gamma +2\nu -\eta  -2 > 1$, it comes
\[
\sum_{t\geq 1}\frac{L_{\nabla GC}}{2}\nu_{n+1}^2\lambda_{\max}(A_{n,\tau})^2(n+1)^\eta < + \infty \quad a.s.
\]
 We can now apply the Robbins-Siegmund theorem. We therefore have $G(\theta_{n})-G(\theta) = o\pa{n^{-\eta}}$ for all $\eta < 2\gamma+2\nu-3$. Thanks to the local strong convexity of $G$, this leads to $\norm{\theta_n-\theta}^2 = o\pa{n^{-\eta}}$ a.s. We are now able to apply Theorem 3.1 and one obtains
	$$\norm{A_n - A}^2 = o\pa{\frac{\ln n^{1+\delta}}{n^{\min\{\gamma,\eta, 2 \beta (q-1)\}}}} a.s. \quad\text{and}\quad \norm{A_{n,\tau} - A}^2 = o\pa{\frac{\ln n^{1+\delta}}{n^{\min\{\eta,2 \beta (q-1)\}}}} a.s. $$
	According to Theorem 4.2  and Theorem 4.3 in \cite{boyer2022asymptotic}, one so has

\begin{equation}
\label{vitessethetan}
\norm{\theta_{n} - \theta}^2 = o\pa{\frac{\ln n^{1+\delta}}{n^{\nu}}} a.s.
\end{equation}

\subsubsection{Rate of convergence of $\theta_{n,\tau'}$.}
For all non-negative integer $n$,
$$\theta_{n+1}-\theta_n = -\nu_{n+1}A_{n,\tau}\nabla_hg(X_{n+1},\theta_{n}),$$
so that 
$$\frac{\theta_{n}-\theta-(\theta_{n+1}-\theta)}{\nu_{n+1}} = \frac{\theta_n-\theta_{n+1}}{\nu_{n+1}}= A_{n,\tau}\nabla_hg(X_{n+1},\theta_{n}).$$ 
In addition,
\begin{align*}
\theta_n-\theta &= H^{-1}\nabla_hg(X_{n+1},\theta_{n})+\pa{H^{-1}\nabla G(\theta_n)-H^{-1}\nabla_hg(X_{n+1},\theta_{n})}\\
\quad &-\pa{H^{-1}\nabla G(\theta_n)-(\theta_n-\theta)}.
\end{align*}
Finally,
\begin{align*}
\theta_{n} - \theta &= H^{-1}A_{n,\tau}^{-1}\frac{\theta_{n}-\theta-(\theta_{n+1}-\theta)}{\nu_{n+1}}+ H^{-1}\pa{\nabla G(\theta_n)-\nabla_hg(X_{n+1},\theta_{n})}
\\&\quad- H^{-1}\pa{\nabla G(\theta_n)-H(\theta_n-\theta)},
\end{align*}
which implies that
\begin{align*}
\theta_{n,\tau} -\theta&= t'_n\sum_{k=0}^n\ln (k+1)^{\tau'}H^{-1}A_{k,\tau}^{-1}\frac{\theta_{k}-\theta-(\theta_{k+1}-\theta)}{\nu_{k+1}}\\
&\quad+t'_n\sum_{k=0}^n\ln (k+1)^{\tau'}H^{-1}\pa{\nabla G(\theta_k)-\nabla_hg(X_{k+1},\theta_{k})}\\&\quad-t'_n\sum_{k=0}^n\ln (k+1)^{\tau'}H^{-1}\pa{\nabla G(\theta_k)-H(\theta_k-\theta)}
\end{align*}
where
$t'_n = \frac{1}{\sum^n_{k=0}\log(k+1)^{\tau'}}.$

\noindent\textbf{Rate of convergence of $t'_n\sum_{k=0}^n\log(k+1)^{\tau'}H^{-1}\pa{\nabla G(\theta_k)-\nabla_hg(X_{k+1},\theta_{k})}$.}
Analogous to the proof of Theorem \ref{consisAn}, since $\pa{\nabla G(\theta_k)-\nabla_hg(X_{k+1},\theta_{k})}_{k \geq 0}$ is a sequence of martingale differences, one can check with the help of a law of large numbers for martingales that 
$$\norm{t'_n\sum_{k=0}^n\log(k+1)^{\tau'}H^{-1}\pa{\nabla G(\theta_k)-\nabla_hg(X_{k+1},\theta_{k})}}_F^2 = o\pa{\frac{\ln n^{1+\delta}}{n}}.$$

\noindent\textbf{Rate of convergence of $t'_n\sum_{k=0}^n\log(k+1)^{\tau'}H^{-1}\pa{\nabla  G(\theta_k)-H(\theta_k-\theta)}$.} With the help of Assumption \textbf{(A3)} and since $\theta_{n}$ converges almost surely to $\theta$, $\left\| \nabla  G(\theta_n)-H(\theta_n-\theta) \right\| = O \left( \left\| \theta_{n} - \theta \right\|^{2} \right)$ a.s. Then, with the help of 
  equality \eqref{vitessethetan} and since $\nu < 1$, we can prove with the help of equality \eqref{equationatilde} that 
$$\norm{t'_n\sum_{k=0}^n\log(k+1)^{\tau'}H^{-1}\pa{\nabla  G(\theta_k)-H(\theta_k-\theta)}}_F = o\pa{\frac{\ln n^{1+\delta}}{n^\nu}},$$
and this term is negligible as soon as $\nu > 1/2$.

\noindent\textbf{Rate of convergence of $t'_n\sum_{k=0}^n\ln (k+1)^{\tau'}H^{-1}A_{k+1,\tau}^{-1}\frac{\theta_{k}-\theta-(\theta_{k+1}-\theta)}{\nu_{k+1}}$.}
First, remark that
\begin{align*}
	&t'_n\sum_{k=0}^n\ln (k+1)^{\tau'}H^{-1}A_{k,\tau}^{-1}\frac{\theta_{k}-\theta-(\theta_{k+1}-\theta)}{\nu_{k+1}} \\&= t'_n\sum_{k=0}^n\ln (k+1)^{\tau'}H^{-1}\frac{A_{k,\tau}^{-1}(\theta_{k}-\theta)-A_{k+1,\tau}^{-1}(\theta_{k+1}-\theta)}{\nu_{k+1}}\\
	&\quad+t'_n\sum_{k=0}^n\ln (k+1)^{\tau'}H^{-1}\frac{(A_{k+1,\tau}^{-1}-A_{k,\tau}^{-1})(\theta_{k+1}-\theta)}{\nu_{k+1}}.
\end{align*}
With the help of an Abel's transform and since $A_{n,\tau}$ converges almost surely to $H^{-1}$, using equality \eqref{vitessethetan} and performing calculations analogous to those used in deriving equality \eqref{abel}, one has
$$\norm{t'_n\sum_{k=0}^n\ln (k+1)^{\tau'}H^{-1}\frac{A_{k,\tau}^{-1}(	\theta_{k}-\theta)-A_{k+1,\tau}^{-1}(\theta_{k+1}-\theta)}{\nu_{k+1}}}_F^2 = o\pa{\frac{\ln n^{1+\delta}}{n^{2-\nu}}}.$$
Observe that
$$A_{k+1,\tau}^{-1}-A_{k,\tau}^{-1} = A_{k+1,\tau}^{-1}(I_d-A_{k+1,\tau}A_{k,\tau}^{-1})=A_{k+1,\tau}^{-1}(A_{k,\tau}-A_{k+1,\tau})A_{k,\tau}^{-1}.$$
Therefore,
\begin{align*}
	&t'_n\sum_{k=0}^n\ln (k+1)^{\tau'}H^{-1}\frac{(A_{k+1,\tau}^{-1}-A_{k,\tau}^{-1})(\theta_{k+1}-\theta)}{\nu_{k+1}}\\
	 &= -\sum_{k=0}^n\ln (k+1)^{\tau'}H^{-1}\frac{A_{k+1,\tau}^{-1}(A_{k+1}-A_{k,\tau})A_{k,\tau}^{-1}(\theta_{k+1}-\theta)}{\nu_{k+1}}\ln (k+1)^\tau t_k.
\end{align*}
Then,
\begin{align*}
	&\norm{t'_n\sum_{k=0}^n\ln (k+1)^{\tau'}H^{-1}\frac{(A_{k+1,\tau}^{-1}-A_{k,\tau}^{-1})(\theta_{k+1}-\theta)}{\nu_{k+1}}}_F\\
	&\le t'_n\ln(n+1)^{\tau'}\sum_{k=0}^n\frac{\norm{H^{-1}A_{k+1,\tau}^{-1}(A_{k+1}-A_{k,\tau})A_{k,\tau}^{-1}}_{op}\norm{\theta_{k+1}-\theta}}{\nu_{k+1}}\ln (k+1)^\tau t_k.
\end{align*}
Thus, as $\norm{\theta_{n} - \theta}^2 = o\pa{\frac{\ln n^{1+\delta}}{n^{\nu}}} a.s.$ and since  $A_{k+1}$ and $A_{k,\tau}$ converge to $H^{-1}$,  we have for all $\delta >0$
$$\norm{t'_n\sum_{k=0}^n\ln (k+1)^{\tau'}H^{-1}\frac{(A_{k+1,\tau}^{-1}-A_{k,\tau}^{-1})(\theta_{k+1}-\theta)}{\nu_{k+1}}}_F^2 = o\pa{\frac{\ln(n+1)^{1+\delta+ 2 \tau}}{n^{2-\nu}}},$$
which is negligible as soon as $2-\nu > 1$.
Finally, we can conclude that
$$\norm{\theta_{n,\tau} - \theta}^2 = o\pa{\frac{\ln n^{1+\delta}}{n}} a.s.$$

\subsection{Proof of Proposition \ref{petitvp}}\label{sec::proof::prop} 
Note that for all $n \geq 0$, $\lambda_{\min}\left( A_{n} \right)> 0$. 
	According to the definition of $A_n$, one has 
	\begin{align*}
		\lambda_{\min}(A_{n+1}) &= \lambda_{\min}\pa{A_n-\gamma_{n+1}\pa{T_{n+1}A_n + A_n T_{n+1}^T- 2I_d}\mathbf{1}_{\norm{Q_{n+1}}\norm{Z_{n+1}}\le\beta_{n+1}}} {- \lambda_{\max}\left(  \Pi_{n+1}^{\perp} \right)} \\
		&\ge \lambda_{\min} (\pa{I_d -\gamma_{n+1}T_{n+1}\mathbf{1}_{\norm{Q_{n+1}}\norm{Z_{n+1}}\le\beta_{n+1}}}A_n  \pa{I_d -\gamma_{n+1}T_{n+1}^T\mathbf{1}_{\norm{Q_{n+1}}\norm{Z_{n+1}}\le\beta_{n+1}}} \\ 
		&\quad+ \pa{2\gamma_{n+1}I_d-\gamma_{n+1}^2T_{n+1}A_nT_{n+1}^T}\mathbf{1}_{\norm{Q_{n+1}}\norm{Z_{n+1}}\le\beta_{n+1}} )    {- \lambda_{\max}\left(  \Pi_{n+1}^{\perp} \right)}\\
		&\ge \lambda_{\min} \pa{I_d -\gamma_{n+1}T_{n+1}\mathbf{1}_{\norm{Q_{n+1}}\norm{Z_{n+1}}\le\beta_{n+1}}}A_n \pa{I_d -\gamma_{n+1}T_{n+1}^T\mathbf{1}_{\norm{Q_{n+1}}\norm{Z_{n+1}}\le\beta_{n+1}}} ) \\
		&\quad+2\gamma_{n+1} - \gamma_{n+1}^2\beta_{n+1}^2\lambda_{\max}(A_n) - 2\gamma_{n+1}\mathbf{1}_{\norm{Q_{n+1}}\norm{Z_{n+1}}\ge\beta_{n+1}}    {- \lambda_{\max}\left(  \Pi_{n+1}^{\perp} \right)}
	\end{align*}
	For all $h \in \mathbb{R}^p$, one has 
	\begin{align*}
		h^T&\pa{I_d -\gamma_{n+1}T_{n+1}\mathbf{1}_{\norm{Q_{n+1}}\norm{Z_{n+1}}\le\beta_{n+1}}}A_n\pa{I_d -\gamma_{n+1}T_{n+1}^T\mathbf{1}_{\norm{Q_{n+1}}\norm{Z_{n+1}}\le\beta_{n+1}}}h\\
		&\ge \norm{A_n^{1/2}\pa{I_d -\gamma_{n+1}T_{n+1}\mathbf{1}_{\norm{\norm{Q_{n+1}}\norm{Z_{n+1}}\le\beta_{n+1}}}}h}^2\\
		&\ge \lambda_{\min}(A_n)\norm{\pa{I_d-\gamma_{n+1}T_{n+1}\mathbf{1}_{\norm{Q_{n+1}}\norm{Z_{n+1}}\le\beta_{n+1}}}h}^2\\
		&\ge\lambda_{\min}\pa{A_n}\pa{1-\gamma_{n+1}\beta_{n+1}}^2\norm{h}^2
	\end{align*}
	Thus, 
	\begin{align*}
		&\lambda_{\min}\pa{\pa{I_d -\gamma_{n+1}T_{n+1}\mathbf{1}_{\norm{Q_{n+1}}\norm{Z_{n+1}}\le\beta_{n+1}}}A_n\pa{I_d -\gamma_{n+1}T_{n+1}^T\mathbf{1}_{\norm{Q_{n+1}}\norm{Z_{n+1}}\le\beta_{n+1}}}}\\ &\quad\ge\lambda_{\min}\pa{A_n}\pa{1-\gamma_{n+1}\beta_{n+1}}^2.
	\end{align*}
	Therefore,
	\begin{align*}
		\lambda_{\min}(A_{n+1}) &\ge\pa{1-\gamma_{n+1}\beta_{n+1}}^2\lambda_{\min}\pa{A_n}+2\gamma_{n+1}-\gamma_{n+1}^2\beta_{n+1}^2\lambda_{\max}(A_n) \\
		&\quad- 2\gamma_{n+1}\mathbf{1}_{\norm{Q_{n+1}}\norm{Z_{n+1}}> \beta_{n+1}}   {- \lambda_{\max}\left(  \Pi_{n+1}^{\perp} \right)}.
	\end{align*}
	Note that $(\beta_n\gamma_n)_n$ is a decreasing sequence. Let $U_n$, $V_n$  {and $W_{n}$} be sequences defined as $$U_n := \sum_{k=1}^{n}\prod_{j=k+1}^{n}\pa{1-\gamma_{j}\beta_{j}}^2\gamma_{k}^2\beta_{k}^2\lambda_{\max}(A_n),$$
	$$V_n := 2\sum_{k=1}^{n}\prod_{j=k+1}^{n}\pa{1-\gamma_{j}\beta_{j}}^2\gamma_{k}\mathbf{1}_{\norm{Q_{n+1}}\norm{Z_{n+1}}\ge \beta_{k+1}} $$
	and
	$$
	W_{n}:=  \sum_{k=1}^{n}\prod_{j=k+1}^{n}\pa{1-\gamma_{j}\beta_{j}}^2  {  \lambda_{\max}\left(  \Pi_{n+1}^{\perp} \right)}
	$$
	One can prove by induction that
	\begin{align*}
		\lambda_{\min}(A_{n}) &\ge\lambda_{\min}\pa{A_{0}}\prod_{k=1}^{n}\pa{1-\gamma_{k}\beta_{k}}^2 + 2\sum_{k=1}^{n}\prod_{j=k+1}^{n}\pa{1-\gamma_{j}\beta_{j}}^2\gamma_{k} - U_n - V_n {-W_{n}},
	\end{align*}
with the convention $\prod_{j=n+1}^{n} \left( 1- \gamma_{j} \beta_{j} \right)^{2} = 1$.
One has  { since $\mathbf{1}_{\Pi_{n}^{\perp} \neq 0}$ converges almost surely to $0$,
\[
  W_{n} \leq \prod_{k=1}^{n} \left( 1- \gamma_{k}\beta_{k} \right)^{2} \sum_{k=1}^{n} \prod_{j=1}^{k} \left( 1-  \gamma_{j}\beta_{j} \right)^{-2} \lambda_{\max} \left( \Pi_{k}^{\perp} \right) = O \left(  \prod_{k=1}^{n} \left( 1- \gamma_{k}\beta_{k} \right)^{2} \right) \quad a.s.
\]
and this term converges exponentially fast to $0$.
In addition}
	$$U_{n+1} = \pa{1-\gamma_{n+1}\beta_{n+1}}^2U_n + \gamma_{n+1}^2\beta_{n+1}^2\lambda_{\max}(A_n),$$
	since $\lambda_{\max}(A_n) = o\pa{\ln n^{1+\delta}n^{1-\gamma}} a.s.$,
	$$U_n = o\pa{\ln n^{1+\delta}\beta_{n+1}n^{1-\gamma}\gamma_{n+1}} = o\pa{\frac{1}{\beta_n}}$$
	since $\beta < \gamma -1/2$. One also has
	$$V_{n+1}= \pa{1-\gamma_{n+1}\beta_{n+1}}^2V_n + \gamma_{n+1}\mathbf{1}_{\norm{Q_{n+1}}\norm{Z_{n+1}}\ge\beta_{n+1}}.$$
	We define $V'_n := n^{v}\ln (n+1)^{-(1+\delta)} V_n$ for some $v > 0$, then, for $n$ large enough,
	\begin{align*}
		\mathbb{E}\cro{V'_{n+1}|\mathcal{F}_n} &\le \frac{\ln n^{1+\delta}}{ \ln (n+1)^{1+\delta}} \frac{(n+1)^{v}}{n^{v}}  \pa{1-\gamma_{n+1}\beta_{n+1}}^2V'_n + \frac{\gamma_{n+1}(n+1)^{v}C_q}{\ln (n+1)^{1+\delta}\beta_n^q} M^{2q}\\
& \le  V'_n + \frac{\gamma_{n+1}(n+1)^{v}C_q}{\ln (n+1)^{1+\delta}\beta_n^q} M^{2q}		. 
	\end{align*}
In order to apply the Robbins-Siegmund theorem, let us take $v = \gamma + q\beta -1$, then $V_{n}'$ converges almost surely to a finite random variable for all $\delta$, which can be translated by 
\[
V_n = o\left(\frac{\ln (n+1)^{1+\delta}}{n^{\gamma + q\beta -1}} \right) \quad a.s
\]
which is  negligible as soon as $\beta > \frac{1-\gamma}{q-1}$.
	It is obvious that $\lambda_{\min}\pa{A_{0}}\prod_{k=1}^{n}\pa{1-\gamma_{k}\beta_{k}}^2 \ge 0$. Finally
	\begin{align*}
		\sum_{k=1}^{n}\prod_{j=k+1}^{n}\pa{1-\gamma_{j}\beta_{j}}^2\gamma_{k}
		&\ge\sum_{k=1}^{n}\frac{1}{2\beta_k}\prod_{j=k+1}^{n}\pa{1-\gamma_{j}\beta_{j}}^22\gamma_{k}\beta_k\\
		&\ge\sum_{k=1}^{n}\frac{1}{2\beta_k}\prod_{j=k+1}^{n}\pa{1-\gamma_{j}\beta_{j}}^2\pa{2\gamma_{k}\beta_{k}-\beta_k^2\gamma_k^2}\\
		&= \sum_{k=1}^{n}\frac{1}{2\beta_k}\pa{\prod_{j=k+1}^{n}\pa{1-\gamma_{j}\beta_{j}}^2 - \prod_{j=k}^{n}\pa{1-\gamma_{j}\beta_{j}}^2}
	\end{align*}
	Since $\pa{\frac{1}{\beta_{n}}}_n$ is a decreasing sequence, one has
	\begin{align*}
		\sum_{k=1}^{n}\prod_{j=k+1}^{n}\pa{1-\gamma_{j}\beta_{j}}^2\gamma_{k} &\geq	\frac{1}{2\beta_n}\sum_{k=1}^{n}\pa{\prod_{j=k+1}^{n}\pa{1-\gamma_{j}\beta_{j}}^2 - \prod_{j=k}^{n}\pa{1-\gamma_{j}\beta_{j}}^2}\\
		&\ge\frac{1}{2\beta_n}\pa{1-\prod_{j=1}^{n}\pa{1-\gamma_{j}\beta_{j}}^2}\\
		&\ge \frac{1-\pa{1-\gamma_{1}\beta_{1}}^2}{2\beta_n}.
	\end{align*}
	Thus, $\lambda_{\min}(A_{n})\ge\frac{c_{1}}{\beta_n} + o \left( \frac{1}{\beta_{n}} \right)$ a.s with $c_{1} = 1-\pa{1-\gamma_{1}\beta_{1}}^2$, which implies that 
	$$\frac{1}{\lambda_{\min}(A_{n})}  = O \left(  \beta_{n} \right) \quad a.s.$$

%% file: Newton_General.bbl
\begin{thebibliography}{}

\bibitem[Agarwal et~al., 2017]{agarwal2017second}
Agarwal, N., Bullins, B., and Hazan, E. (2017).
\newblock Second-order stochastic optimization for machine learning in linear
  time.
\newblock {\em Journal of Machine Learning Research}, 18(116):1--40.

\bibitem[Bach, 2014]{bach2014adaptivity}
Bach, F. (2014).
\newblock Adaptivity of averaged stochastic gradient descent to local strong
  convexity for logistic regression.
\newblock {\em The Journal of Machine Learning Research}, 15(1):595--627.

\bibitem[Bercu et~al., 2023]{bercu2023stochastic}
Bercu, B., Bigot, J., Gadat, S., and Siviero, E. (2023).
\newblock A stochastic gauss--newton algorithm for regularized semi-discrete
  optimal transport.
\newblock {\em Information and Inference: A Journal of the IMA},
  12(1):390--447.

\bibitem[Bercu et~al., 2020a]{bercu2020stochastic}
Bercu, B., Costa, M., and Gadat, S. (2020a).
\newblock Stochastic approximation algorithms for superquantiles estimation.
\newblock {\em arXiv preprint arXiv:2007.14659}.

\bibitem[Bercu et~al., 2020b]{bercu2020efficient}
Bercu, B., Godichon, A., and Portier, B. (2020b).
\newblock An efficient stochastic newton algorithm for parameter estimation in
  logistic regressions.
\newblock {\em SIAM Journal on Control and Optimization}, 58(1):348--367.

\bibitem[Blackard, 1998]{blackard1998comparison}
Blackard, J.~A. (1998).
\newblock {\em Comparison of neural networks and discriminant analysis in
  predicting forest cover types}.
\newblock Colorado State University.

\bibitem[Boyer and Godichon-Baggioni, 2022]{boyer2022asymptotic}
Boyer, C. and Godichon-Baggioni, A. (2022).
\newblock On the asymptotic rate of convergence of stochastic newton algorithms
  and their weighted averaged versions.
\newblock {\em Computational Optimization and Applications}, pages 1--52.

\bibitem[Byrd et~al., 2016]{byrd2016stochastic}
Byrd, R.~H., Hansen, S.~L., Nocedal, J., and Singer, Y. (2016).
\newblock A stochastic quasi-newton method for large-scale optimization.
\newblock {\em SIAM Journal on Optimization}, 26(2):1008--1031.

\bibitem[Cardot et~al., 2017]{CCG2015}
Cardot, H., C{\'e}nac, P., and Godichon-Baggioni, A. (2017).
\newblock Online estimation of the geometric median in hilbert spaces:
  Nonasymptotic confidence balls.
\newblock {\em The Annals of Statistics}, 45(2):591--614.

\bibitem[Cardot et~al., 2013]{HC}
Cardot, H., C{\'e}nac, P., and Zitt, P.-A. (2013).
\newblock Efficient and fast estimation of the geometric median in {H}ilbert
  spaces with an averaged stochastic gradient algorithm.
\newblock {\em Bernoulli}, 19(1):18--43.

\bibitem[Cardot and Godichon-Baggioni, 2017]{CG2015}
Cardot, H. and Godichon-Baggioni, A. (2017).
\newblock Fast estimation of the median covariation matrix with application to
  online robust principal components analysis.
\newblock {\em Test}, 26(3):461--480.

\bibitem[C{\'e}nac et~al., 2020]{cenac2020efficient}
C{\'e}nac, P., Godichon-Baggioni, A., and Portier, B. (2020).
\newblock An efficient averaged stochastic gauss-newton algorithm for
  estimating parameters of non linear regressions models.
\newblock {\em arXiv preprint arXiv:2006.12920}.

\bibitem[Chang and Lin, 2011]{chang2011libsvm}
Chang, C.-C. and Lin, C.-J. (2011).
\newblock Libsvm: a library for support vector machines.
\newblock {\em ACM transactions on intelligent systems and technology (TIST)},
  2(3):1--27.

\bibitem[Cohen et~al., 2017]{cohen2017projected}
Cohen, K., Nedi{\'c}, A., and Srikant, R. (2017).
\newblock On projected stochastic gradient descent algorithm with weighted
  averaging for least squares regression.
\newblock {\em IEEE Transactions on Automatic Control}, 62(11):5974--5981.

\bibitem[Costa and Gadat, 2020]{costa2020non}
Costa, M. and Gadat, S. (2020).
\newblock Non asymptotic controls on a recursive superquantile approximation.

\bibitem[Dua et~al., 2017]{dua2017uci}
Dua, D., Graff, C., et~al. (2017).
\newblock Uci machine learning repository.

\bibitem[Duflo, 1990]{duflo1990methodes}
Duflo, M. (1990).
\newblock M{\'e}thodes r{\'e}cursives al{\'e}atoires.
\newblock {\em (No Title)}.

\bibitem[Fr{\'e}chet, 1948]{frechet1948elements}
Fr{\'e}chet, M. (1948).
\newblock Les {\'e}l{\'e}ments al{\'e}atoires de nature quelconque dans un
  espace distanci{\'e}.
\newblock In {\em Annales de l'institut Henri Poincar{\'e}}, volume~10, pages
  215--310.

\bibitem[Gadat and Panloup, 2017]{gadat2017optimal}
Gadat, S. and Panloup, F. (2017).
\newblock Optimal non-asymptotic bound of the ruppert-polyak averaging without
  strong convexity.
\newblock {\em arXiv preprint arXiv:1709.03342}.

\bibitem[Godichon-Baggioni, 2019]{godichon2019online}
Godichon-Baggioni, A. (2019).
\newblock Online estimation of the asymptotic variance for averaged stochastic
  gradient algorithms.
\newblock {\em Journal of Statistical Planning and Inference}, 203:1--19.

\bibitem[Godichon-Baggioni and Lu, 2023]{godichon2023online}
Godichon-Baggioni, A. and Lu, W. (2023).
\newblock Online stochastic newton methods for estimating the geometric median
  and applications.
\newblock {\em arXiv preprint arXiv:2304.00770}.

\bibitem[Godichon-Baggioni and Portier, 2017]{godichon2017averaged}
Godichon-Baggioni, A. and Portier, B. (2017).
\newblock An averaged projected robbins-monro algorithm for estimating the
  parameters of a truncated spherical distribution.

\bibitem[Godichon-Baggioni et~al., 2022]{godichon2022recursive}
Godichon-Baggioni, A., Portier, B., and Lu, W. (2022).
\newblock Recursive ridge regression using second-order stochastic algorithms.

\bibitem[Godichon-Baggioni et~al., 2023]{godichon2023non}
Godichon-Baggioni, A., Werge, N., and Wintenberger, O. (2023).
\newblock Non-asymptotic analysis of stochastic approximation algorithms for
  streaming data.
\newblock {\em ESAIM: Probability and Statistics}, 27:482--514.

\bibitem[Gower et~al., 2019a]{gower2019rsn}
Gower, R., Kovalev, D., Lieder, F., and Richt{\'a}rik, P. (2019a).
\newblock Rsn: randomized subspace newton.
\newblock {\em Advances in Neural Information Processing Systems}, 32.

\bibitem[Gower et~al., 2019b]{gower2019sgd}
Gower, R.~M., Loizou, N., Qian, X., Sailanbayev, A., Shulgin, E., and
  Richt{\'a}rik, P. (2019b).
\newblock Sgd: General analysis and improved rates.
\newblock In {\em International conference on machine learning}, pages
  5200--5209. PMLR.

\bibitem[Juan et~al., 2016]{juan2016field}
Juan, Y., Zhuang, Y., Chin, W.-S., and Lin, C.-J. (2016).
\newblock Field-aware factorization machines for ctr prediction.
\newblock In {\em Proceedings of the 10th ACM conference on recommender
  systems}, pages 43--50.

\bibitem[LeCun et~al., 1998]{lecun1998gradient}
LeCun, Y., Bottou, L., Bengio, Y., and Haffner, P. (1998).
\newblock Gradient-based learning applied to document recognition.
\newblock {\em Proceedings of the IEEE}, 86(11):2278--2324.

\bibitem[Leluc and Portier, 2020]{leluc2020asymptotic}
Leluc, R. and Portier, F. (2020).
\newblock Asymptotic optimality of conditioned stochastic gradient descent.
\newblock {\em arXiv preprint arXiv:2006.02745}.

\bibitem[Lichman et~al., 2013]{lichman2013uci}
Lichman, M. et~al. (2013).
\newblock Uci machine learning repository.

\bibitem[Mohammad and McCluskey, 2012]{phishing_websites_327}
Mohammad, R. and McCluskey, L. (2012).
\newblock {Phishing Websites}.
\newblock UCI Machine Learning Repository.
\newblock {DOI}: https://doi.org/10.24432/C51W2X.

\bibitem[Mokkadem and Pelletier, 2011]{mokkadem2011generalization}
Mokkadem, A. and Pelletier, M. (2011).
\newblock A generalization of the averaging procedure: The use of
  two-time-scale algorithms.
\newblock {\em SIAM Journal on Control and Optimization}, 49(4):1523--1543.

\bibitem[Moulines and Bach, 2011]{moulines2011non}
Moulines, E. and Bach, F. (2011).
\newblock Non-asymptotic analysis of stochastic approximation algorithms for
  machine learning.
\newblock {\em Advances in neural information processing systems}, 24.

\bibitem[Pelletier, 1998]{pelletier1998almost}
Pelletier, M. (1998).
\newblock On the almost sure asymptotic behaviour of stochastic algorithms.
\newblock {\em Stochastic processes and their applications}, 78(2):217--244.

\bibitem[Pelletier, 2000]{pelletier2000asymptotic}
Pelletier, M. (2000).
\newblock Asymptotic almost sure efficiency of averaged stochastic algorithms.
\newblock {\em SIAM Journal on Control and Optimization}, 39(1):49--72.

\bibitem[Pilanci and Wainwright, 2016]{pilanci2016iterative}
Pilanci, M. and Wainwright, M.~J. (2016).
\newblock Iterative hessian sketch: Fast and accurate solution approximation
  for constrained least-squares.
\newblock {\em Journal of Machine Learning Research}, 17(53):1--38.

\bibitem[Pilanci and Wainwright, 2017]{pilanci2017newton}
Pilanci, M. and Wainwright, M.~J. (2017).
\newblock Newton sketch: A near linear-time optimization algorithm with
  linear-quadratic convergence.
\newblock {\em SIAM Journal on Optimization}, 27(1):205--245.

\bibitem[Pinelis, 1994]{pinelis1994optimum}
Pinelis, I. (1994).
\newblock Optimum bounds for the distributions of martingales in banach spaces.
\newblock {\em The Annals of Probability}, pages 1679--1706.

\bibitem[Platt, 1999]{platt199912}
Platt, J.~C. (1999).
\newblock 12 fast training of support vector machines using sequential minimal
  optimization.
\newblock {\em Advances in kernel methods}, pages 185--208.

\bibitem[Polyak and Juditsky, 1992]{polyak1992acceleration}
Polyak, B.~T. and Juditsky, A.~B. (1992).
\newblock Acceleration of stochastic approximation by averaging.
\newblock {\em SIAM journal on control and optimization}, 30(4):838--855.

\bibitem[Robbins and Monro, 1951]{robbins1951stochastic}
Robbins, H. and Monro, S. (1951).
\newblock A stochastic approximation method.
\newblock {\em The annals of mathematical statistics}, pages 400--407.

\bibitem[Robbins and Siegmund, 1971]{robbins1971convergence}
Robbins, H. and Siegmund, D. (1971).
\newblock A convergence theorem for non negative almost supermartingales and
  some applications.
\newblock In {\em Optimizing methods in statistics}, pages 233--257. Elsevier.

\bibitem[Schraudolph et~al., 2007]{schraudolph2007stochastic}
Schraudolph, N.~N., Yu, J., and G{\"u}nter, S. (2007).
\newblock A stochastic quasi-newton method for online convex optimization.
\newblock In {\em Artificial intelligence and statistics}, pages 436--443.
  PMLR.

\bibitem[Shanno, 1970]{shanno1970conditioning}
Shanno, D.~F. (1970).
\newblock Conditioning of quasi-newton methods for function minimization.
\newblock {\em Mathematics of computation}, 24(111):647--656.

\bibitem[Toulis and Airoldi, 2017]{toulis2017asymptotic}
Toulis, P. and Airoldi, E.~M. (2017).
\newblock Asymptotic and finite-sample properties of estimators based on
  stochastic gradients.

\bibitem[Ye et~al., 2017]{ye2017approximate}
Ye, H., Luo, L., and Zhang, Z. (2017).
\newblock Approximate newton methods and their local convergence.
\newblock In {\em International Conference on Machine Learning}, pages
  3931--3939. PMLR.

\bibitem[Yuan et~al., 2012]{yuan2012improved}
Yuan, G.-X., Ho, C.-H., and Lin, C.-J. (2012).
\newblock An improved glmnet for l1-regularized logistic regression.
\newblock {\em The Journal of Machine Learning Research}, 13(1):1999--2030.

\end{thebibliography}
